\documentclass[reqno]{amsart}
\usepackage[T1]{fontenc}
\usepackage{latexsym,amssymb,amsmath}

\usepackage[all,ps]{xy}
\usepackage{amsthm}
\usepackage{longtable}
\usepackage{epsfig}
\usepackage{mathrsfs}
\usepackage{hhline}
\usepackage{epic}
\usepackage{enumerate}
\usepackage{pgf,tikz}
\usepackage{mathrsfs}
\usetikzlibrary{arrows}

 \newcommand{\Alt}{\mathfrak{A}}

 \newcommand{\sym}{\mathfrak{S}}

 \newcommand{\Ind}{\operatorname{Ind}}
 \newcommand{\Res}{\operatorname{Res}}

\newcommand{\sgn}{\operatorname{sgn}}

 \newcommand{\C}{\mathbb{C}}

\newcommand{\N}{\mathbb{N}}
 \newcommand{\Q}{\mathbb{Q}}
 \newcommand{\Z}{\mathbb{Z}}
 
 \newcommand{\Sym}{\mathfrak{S}}

 \newcommand{\Irr}{\operatorname{Irr}}

  \newcommand{\GL}{\operatorname{GL}}

\newtheorem{theorem}{Theorem}[section] 
\newtheorem{lemma}[theorem]{Lemma}     
\newtheorem{corollary}[theorem]{Corollary}
\newtheorem{proposition}[theorem]{Proposition}

\newtheorem{example}[theorem]{Example}
\theoremstyle{definition}
\newtheorem{remark}[theorem]{Remark}

\newcommand{\pu}[1]{{\color{purple}{#1}}}

\title[]
{The Navarro Conjecture for the alternating groups}
\author{Olivier Brunat}

\address{Universit\'e Paris-Diderot Paris 7\\ Institut de math\'ematiques de
         Jussieu -- Paris Rive Gauche\\ UFR de math\'e\-matiques\\ Case
7012\\ 75205 Paris Cedex 13\\
         France.}
\email{olivier.brunat@imj-prg.fr}

\author{Rishi Nath}

\address{York College, City University of New York, 
94--20 Guy R. Brewer Blvd. \\
Jamaica, NY 11435\\
USA
}
\email{rnath@york.cuny.edu}


\begin{document}
\begin{abstract}
Recently Navarro proposed a strengthening of the unsolved McKay
conjecture using Galois automorphisms. We prove that the Navarro
conjecture and its blockwise version hold for the
alternating groups.
\end{abstract}
\maketitle

\section{Introduction} 

    Let $G$ be a finite group of order $n$ and $p$ be a prime divisor of
$n$. We denote by $\Irr(G)$ the set of irreducible complex characters of
$G$, and by $\Irr_{p'}(G)$ the subset of irreducible characters with
degree prime to $p$. In 1972, John McKay conjectured that
$|\Irr_{p'}(G)|=|\Irr_{p'}(\operatorname{N}_G(P))|$, where $P$ is a Sylow
$p$-subgroup of $G$. Although the conjecture remains open, there is strong
evidence in its favor. In 2007, I. M. Isaacs, G. Malle and G.
Navarro~\cite{IMN} reduced the problem to a question on finite simple
groups. In particular, they assert that if a set of conditions holds for
all non abelian finite simple groups, then the original conjecture holds
for all finite groups. Using this strategy, Malle and Sp\"ath {recently
proved}~\cite{MalleSpaeth} that McKay conjecture holds at $p=2$ for all
finite groups.

The McKay conjecture {has} lead to a family of other conjectures on finite
groups. For example, the conjectures of Alperin-McKay, of Dade, of Brou\'e
and of Isaacs-Navarro are of a similar flavor. This paper is concerned
with a refinement of the McKay conjecture due to
Navarro~\cite{NavarroGalois}, which posits not only a correspondence
between the set of global-and-local irreducible characters of $p'$-degree,
but also between their character values.

In order to state the conjecture more precisely, we introduce some
notation. Let $\Q_n=\Q(\omega_n)$ be the cyclotomic subfield of $\C$,
where $\omega_n=e^{2i\pi/n}$, and $\mathcal
G_n=\operatorname{Gal}(\Q_n|\Q)$. For any $f\in \mathcal G_n$,
$\chi\in\Irr(G)$ and $g\in G$, we set ${}^f\chi(g)=f(\chi(g))$, inducing
an action of $\mathcal G_n$ on $\Irr(G)$ and then on $\Irr_{p'}(G)$.
Furthermore, if $H$ is a subgroup of $G$ of order $d$, then $d$ divides
$n$ and $\Q_d$ is a subfield of $\Q_n$. Note also that, if $f\in{\mathcal
G}_n$, then $f(\omega_d)$ is a primitive $d$-root of unity, that is, there
is some integer $r$ prime to $d$ such that $f(\omega_d)=\omega_d^r$. In
particular, $f(\Q_d)=\Q_d$ and $f|_{\Q_d}\in\mathcal G_d$. Hence,
$\mathcal G_n$ acts on $\Irr(H)$ through $\mathcal G_n\rightarrow\mathcal
G_d,\,f\mapsto f|_{\Q_d}$. 

Even though there cannot exist a bijection $\Irr_{p'}(G)\rightarrow
\Irr_{p'}(\operatorname{N}_G(P))$ that commutes with ${\mathcal G}_n$,
Gabriel Navarro observed in~\cite{NavarroGalois} that there should exist a
bijection commuting with a special subgroup $\mathcal H_n$ of $\mathcal
G_n$. More precisely, if we write $n=p^{\ell}m$ with $m$ prime to $p$, then
$\omega_n$ can be uniquely writen as a product $\omega\delta$, where
$\omega$ has order $p^{\ell}$ and $\delta$ has order $m$. It follows that
$\mathcal G_n=\mathcal K_n\times\mathcal J_n$, where $\mathcal K_n$ and
$\mathcal J_n$ are respectively the subgroups of $\mathcal G_n$ fixing
$\delta$ and $\omega$. Let $\sigma_n$ be the element of $\mathcal J_n$
such that $\sigma_n(\delta)=\delta^p$. If we set $\mathcal H_n=\mathcal
K_n\times \langle \sigma_n\rangle$, then $\mathcal K_n$ is isomorphic to
$\operatorname{Gal}(\Q_{p^{\ell}}|\Q)$, and $\mathcal H_n$ is thus the subgroup
of $\mathcal G_n$ which acts on the $p'$-roots of unity of $\Q_n$ by a
power of $p$. In~\cite[Conjecture A]{NavarroGalois}, Navarro conjectured
that for any $f\in\mathcal H_n$, there are the same number of characters
of $\Irr_{p'}(G)$ and of $\Irr_{p'}(\operatorname{N}_G(P))$ fixed by $f$.
In the following, elements of $\mathcal H_n$ will be called
\emph{Navarro automorphisms}.

While significant progress has been made on the McKay conjecture, evidence
of the veracity of the Navarro refinement has been limited to a handful of
cases: groups of odd order by Isaacs~\cite{IsaacsSolv}, for solvable
groups (E. Dade), for sporadic groups, for symmetric groups (P. Fong), for
simple groups of Lie type in defining characteristic (L.
Ruhstorfer~\cite{Ruhstorfer}), and for alternating groups for $p=2$ (by
the second author ~\cite{RishiNathAlt}). 
A. Turull gave {in~\cite{Turull} a conjecture which implies the
Navarro conjecture. He proved in~\cite{Turull} his conjecture for the
special linear groups in defining characteristic and in~\cite{Turull2}
for $p$-solvable groups.}
Recently, Navarro, Spaeth and Vallejo proved a reduction theorem of
Navarro refinement to the quasisimple groups~\cite{NavarroSpaethVallejo}. 

In this paper, we verify that when $p$ is odd the conjecture holds for an
important family of simple groups, the alternating groups. More precisely,
we will prove the following general result.

\begin{theorem}
\label{thm:main}
Let $n$ be a positive integer, and {$2<p\leq n$} be an odd prime number. 
Fix {a Sylow $p$-subgroup $P$} of
$\Alt_n$. Then there is a natural $\mathcal H_{n!/2}$-equivariant bijection
$$\Phi:\Irr_{p'}(\Alt_n)\rightarrow\Irr_{p'}(\operatorname{N}_{\Alt_n}(P)).$$  
\end{theorem}

The paper is organized as follows. In Section~\ref{sec:sym} we discuss the
Navarro conjecture for the symmetric groups. It was noted
in~\cite{NavarroGalois} that this was checked by Fong. However, since the
details will be useful in our work, we supply them for the convenience of
the reader.\medskip

In Section~\ref{sec:local}, after studying the representation theory of
the alternating subgroup of a direct product of groups contained in some
symmetric group, we describe the action of automorphisms on
{$\Irr_{p'}(\operatorname{N}_{\Alt_n}(P))$}.\medskip

In Section~\ref{sec:global}, we describe the action of Navarro
automorphisms on $\Irr_{p'}(\Alt_n)$. To this end, we obtain an explicit
formula for the diagonal hook lengths of a symmetric partition of $n$ in
terms of the diagonal hooks of the $p$-core and $p$-quotient. These
results are of independent interest : many
partition-theoretic questions about Ramanujan-type congruences, monotonicity and the Durfee square can be answered using the relationship
between a partition and its $p$-core and $p$-quotient (see for example
the work of F. Garvan, D. Kim and D. Stanton \cite{GarvanKimStanton}).

\medskip

Finally, Section~\ref{sec:preuve} and 6 are devoted to the proof of
the Navarro conjecture and its blockwise version for the
alternating groups with no condition over the prime $p$.
 
\section{Verification of the conjecture for the symmetric groups}
\label{sec:sym}

Let $n$ be a positive integer and $p$ be a prime number. Let $P$ be a
Sylow $p$-subgroup of $\sym_n$, and set 
$N=\operatorname{N}_{\sym_n}(P)$. First, following~\cite[\S1 and \S2]{Fong} we
describe a parametrization of $\Irr_{p'}(\sym_n)$ and of
$\Irr_{p'}(N)$. 
The irreducible characters of $\sym_n$ are naturally labeled by the set
of partitions of $n$. For any such partition $\lambda$, we denote by
$\chi_{\lambda}$ the corresponding character of $\sym_n$.

For any partition $\lambda$ of $n$, we write $|\lambda|=n$
for the size of $\lambda$.
We also will denote by $Y(\lambda)$ for the Young diagram of
$\lambda$. Using
matrix notation, we associate to any $(i,j)$-box of $Y(\lambda)$, {an
$(i,j)$-hook} with {\it hook-length} $h_{ij}$. We denote by
$\mathfrak{D}(\lambda)$ the set of {\it diagonal hooks} of $\lambda$,
that is, the hooks in positions $(i,i)$. Such a hook will be known as the
$i$th-diagonal hook. We denote by
\begin{equation}
\label{eq:hooklength}
\mathfrak d(\lambda)=\{|h|\mid h\in\mathfrak D(\lambda)\}
\end{equation}
the set of hook-lengths of the diagonal hooks of $\lambda$.

Recall that $\lambda$ is \emph{symmetric} if $\lambda=\lambda^*$. When
$\lambda$ is symmetric the $i$th-diagonal hook is uniquely determined by
its hook-length.

Recall that any partition $\lambda$ is completely determined by its
$p$-core ${\mathcal Cor}_p(\lambda)$ and its $p$-quotient ${\mathcal
Quo}_p(\lambda)=(\lambda_0,\ldots,\lambda_{p-1})$;~\cite[\S3]{Olsson}.
Write $I=\{0,\ldots,p-1\}$ and $I^0=\emptyset$. Set
$\lambda^{\emptyset}=\lambda$.
Let $k$ be a non-negative integer, and assume $\lambda^{\underline j}$
is constructed for any $\underline j\in I^k$. Then we define $\lambda_{\underline
j}=\mathcal Cor_p(\lambda^{\underline j})$ and for $\underline
j=(j_1,\ldots,j_{k+1})\in I^{k+1}$, we write $\lambda^{\underline
j}=\mathcal Quo_p(\lambda^{j_1,\ldots,j_k})_{j_{k+1}}$. 
For any $k\geq 0$, write 
${\mathcal Cor}^{(k)}_p(\lambda)=\{\lambda_{\underline j}\,|\, \underline
j\in I^k\}$.
The set 
\begin{equation}
\label{eq:coretower}
\mathcal{CT}(\lambda)=\bigcup_{k\geq 0}{\mathcal Cor}^{(k)}_p(\lambda)
\end{equation}
is called the \emph{p-core tower} of $\lambda$. 
For more details, we refer to~\cite[p.\,41]{Olsson}.\medskip

On the other hand, recall that by~\cite[Proposition 1.1]{Fong},
$\chi_{\lambda}\in\Irr_{p'}(\sym_n)$ if and only if $0\leq
c_k(\lambda)\leq p-1$ for all $k\geq 0$, where
$c_k(\lambda)=\sum_{\underline j\in I^k}|\lambda_{\underline j}|$.
\medskip

Let $n=n_0+n_1p+n_2p^2+n_3p^3+\cdots$ be the $p$-adic expansion of $n$.
Note that the $p'$-irreducible characters of $N$ are exactly the ones
that have $P'$ in their kernel; that is, the irreducible characters of
$N$ which can be lifted from the projection $N\rightarrow N/P'$.
Furthermore, by~\cite[\S2]{Fong}, one has
\begin{equation}
\label{eq:pprimenormalizer}
N/P'
\simeq
\sym_{n_0}\times\prod_{k\geq 1}Y^k\wr\sym_{n_k},
\end{equation}
where {$X$ is a Sylow $p$-subgroup of $\sym_p$}
and $Y=\operatorname{N}_{\sym_{p}}(X)$.

Let $k\geq 1$. Write $N_k=Y^k\wr\sym_{n_k}$. The elements of $N_k$ are
denoted by $(y;\sigma)$, where $y=(y_1,\ldots,y_{n_k})\in (Y^k)^{n_k}$
and $\sigma\in \sym_{n_k}$. For any $\sigma\in\sym_{n_k}$, we denote by
$C(\sigma)$ the set of cycles of $\sigma$ {with respect to}
its canonical decomposition into cycles with disjoint supports. 
For $\tau\in C(\sigma)$, the corresponding ``cycle'' of $N_k$ is $(y_{\tau};\tau)$,
where $(y_{\tau})_j=y_j$ if $j\in\operatorname{supp}(\tau)$ and
$(y_{\tau})_j=1$ otherwise. 
For any $\tau\in C(\sigma)$,  we also define  the \emph{cycle product}
$\mathfrak c((y;\sigma),\tau)=\prod_{j\in\operatorname{supp}(\tau)}y_j$
of $(y;\sigma)$ with respect to $\tau$. 

Note that $Y=\langle a\rangle\rtimes \langle b\rangle$ with $a$ and $b$
of order $p$ and $p-1$ respectively.  Recall that $Y$ has $p-1$ linear
characters obtained by lifting {the ones} of $\langle b\rangle$ through
$Y\rightarrow Y/\langle a\rangle\simeq \langle b\rangle$, and one
character of degree $p-1$ obtained by inducing any non-trivial
characters of $\langle a\rangle$ to $Y$. Write
\begin{equation}
\label{eq:irrY}
\Irr(Y)=\{\xi_0,\ldots,\xi_{p-1}\}.
\end{equation}
Then these characters are $\Q(\omega_{p-1})$-valued by
construction.\medskip

Let $k\geq 1$. For $\underline
j=(j_1,\ldots,j_k)\in I^k$, we set  $\xi_{\underline
j}=\xi_{j_1}\otimes\xi_{j_2}\otimes\cdots\otimes\xi_{j_k}$. Then
\begin{equation}
\label{eq:irrYk}
\Irr(Y^k)=\{\xi_{\underline j}\,|\,\underline j\in I^k\}.
\end{equation}
Let $\mathcal{MP}(p^k,n_k)$ be the set of $p^k$-multipartitions of $n_k$
that is, multipartitions $\underline
\lambda=(\lambda_{\underline j};\,\underline j\in I^k)$ such that
$\sum_{\underline j\in I^k}|\lambda_{\underline j}|=n_k.$\\ 

\begin{remark}
\label{rk:lexico}
In the following, we {will always} assume that the $\lambda_{\underline
j}$'s in $\underline\lambda$ appear in increasing lexicographic order. 
\end{remark}

By~\cite{James-Kerber}, the irreducible characters of $N_k$ can be
labeled by $\mathcal{MP}(p^k,n_k)$ as follows. Let $\underline
\lambda=(\lambda_{\underline j};\,\underline j\in I^k)$ be such that
$\sum_{\underline j\in I^k}|\lambda_{\underline j}|=n_k$.  Consider the
irreducible character
$$\xi_{\underline \lambda}=\bigotimes_{\underline j\in I^k}\underbrace{\xi_{\underline
j}\otimes\cdots\otimes\xi_{\underline j}}_{|\lambda_{\underline j}|\ \text{times}}$$
of $(Y^k)^{n_k}$. If we set $\sym_{\underline \lambda}=\prod_{\underline
j\in I^k}\sym_{|\lambda_{\underline j}|}$, then the inertial subgroup of
$\xi_{\underline \lambda}$ in $N_k$ is 
$$N_{k,\underline
\lambda}=(Y^k)^{n_k}\rtimes \sym_{\underline\lambda}=\prod_{\underline
j\in I^k}Y^k\wr\sym_{|\lambda_{\underline j}|}.$$

We denote by
$E(\xi_{\underline\lambda})$ the James-Kerber extension of $\xi_{\underline \lambda}$ to $N_{k,\underline\lambda}$ described
in~\cite[\S4.3]{James-Kerber}. Note that
$E(\xi_{\underline\lambda})=\bigotimes E(\xi_{\underline
j}^{|\lambda_{\underline j}|})$ and~\cite[Lemma 4.3.9]{James-Kerber}
gives 
\begin{align}E(\xi_{\underline\lambda})\left(\prod_{\underline j\in I^k}(y_{\underline
j};\sigma_{\underline j})\right)
&=\prod_{\underline j\in I^k}E(\xi_{\underline
j}^{|\lambda_{\underline j}|})(y_{\underline j};\sigma_{\underline j})
\label{eq:jamesextension}
\\
&=
\prod_{\underline j\in I^k,|\lambda_{\underline j}|\neq 0}\
\prod_{\tau\in
C(\sigma_{\underline j})}\xi_{\underline j}(\mathfrak c((y_{\underline
j};\sigma_{\underline j}),\tau))\nonumber.
\end{align}\medskip

Now, write $\chi_{\underline\lambda}$ for the characters
$\prod_{\underline j\in I^k}\chi_{\lambda_{\underline
j}}\in\sym_{\underline\lambda}$ lifted through the canonical projection
$N_{k,\underline\lambda}\rightarrow
N_{k,\underline\lambda}/(Y^k)^{n_k}\simeq \sym_{\underline\lambda}$, and
define
\begin{equation}
\label{eq:psiNk}
\psi_{\underline\lambda,k}=\Ind_{N_{k,\underline\lambda}}^{N_k}\left(E(\xi_{\underline\lambda})
\chi_{\underline\lambda}\right).
\end{equation}
Then
$$\Irr(N_k)=\left\{\psi_{\underline\lambda,k}\,|\,\underline\lambda\in\mathcal{MP}(p^k,n_k)\right\}.$$
{The following result will be useful.}
\begin{lemma}
\label{lemma:induction}
Let $G$ be a finite group of order $n$ and $H$ a subgroup of {$G$}. Let
$f\in\mathcal G_n$. Then for any class function $\phi$ on $H$, we have
$${}^f\Ind_H^G(\phi)=\Ind_H^G({}^f\phi).$$ 
\end{lemma}
%

\begin{proposition}
\label{prop:navarrosym}
The Navarro conjecture holds for the symmetric groups. 
\end{proposition}
\begin{proof}Let $n$ be a positive integer and $p\leq n$ be a prime
number.  Since the characters of $\sym_n$ are rational-valued, {they} are
fixed by any automorphisms of $\mathcal H_{n!}$. It remains to show that
any $p'$-order irreducible characters of $N$ {is} also fixed.
From~(\ref{eq:pprimenormalizer}), it is sufficient to show that for
$k\geq 1$, the irreducible characters of $N_k$ are fixed under any $f\in
\mathcal H_{n!}$. If $f\in \mathcal K_{n!}$, then $f$ fixes any
$p'$-roots of unity. However, the values of the characters of $\Irr(Y)$
lie in $\Q(\omega_{p-1})$, and are thus fixed by $f$. Write
$\sigma=\sigma_{n!}$. We have $\sigma(x)=x^p$ for any $p'$-root of unity
$x$. Since $\omega_{p-1}$ is a $p'$-root of unity, we deduce that
$\sigma(\omega_{p-1})=\omega_{p-1}^p=\omega_{p-1}$, and $\sigma$ fixes
the characters of $\Irr(Y)$. In either case, the irreducible characters of
$Y$ are fixed by $\mathcal H_{n!}$. {Let $n=n_0+n_1p+\cdots$ be} the
$p$-adic expansion of $n$ as above and fix $k\geq 1$. Let $\underline
\lambda\in\mathcal{MP}(p^k,n_k)$. From~(\ref{eq:jamesextension}) and
the fact that $\chi_{\underline \lambda}$ is rational-valued, we obtain
that $E(\xi_{\underline\lambda})\chi_{\underline\lambda}$ is fixed under
any $f\in\mathcal H_{n!}$. Finally, we conclude using
Lemma~\ref{lemma:induction}.  
\end{proof}

\section{Alternating groups. The local case}
\label{sec:local}

For any subgroup $G$ of $\sym_n$, we set $G^+=G\cap\Alt_n$. In
particular, $[G:G^+]\leq 2$ and $G^+$ is the kernel of the restriction
to $G$ of the {sign character} of $\sym_n$, also denoted by $\sgn:G\rightarrow
\{-1,1\}$. Suppose $G^+\neq G$.  For $\chi\in\Irr(G)$, we write
$\chi^*=\chi\otimes\sgn$. By Clifford theory, if $\chi\neq\chi^*$, then
$\chi$ and $\chi^*$ restrict to an irreducible character of $G^+$ also
denoted by $\chi$. If $\chi=\chi^*$, then the restriction of $\chi$ to
$G^+$ is the sum of two irreducible characters denoted by $\chi^+$ and
$\chi^-$. Note that in this case, $\chi(g)=0$ for all $g\notin G^+$, and
$\chi(g)=\chi^+(g)+\chi^-(g)$ for $g\in G^+$. All irreducible characters
of $G^+$ are obtained {exactly} once by this process.  A \emph{split class}
$c$ of $G$ is a conjugacy class of $G$ contained in $G^+$ such that $c$
is the union of two $G^+$-classes $c^+$ and $c^-$.

Let $f$ be a Galois automorphism and $\chi\in\Irr(G)$ be such that
$\chi=\chi^*$ and ${}^f\chi=\chi$. Then {$f$ acts} on
$\{\chi^+,\chi^-\}$. We define $\varepsilon(\chi,f)\in\{-1,1\}$ by
setting $\varepsilon(\chi,f)=1$ if ${}^f\chi^+=\chi^+$ and
$\varepsilon(\chi,f)=-1$ otherwise. In particular, for $\eta\in\{-1,1\}$
we have
\begin{equation}
\label{eq:signebouge}
{}^f\chi^{\eta}=\chi^{\varepsilon(\chi,f)\eta}.
\end{equation}

\subsection{Reduction of the problem}
\label{subsec:redpb}
Let $G_1,\ldots,G_r$ be subgroups of $\sym_n$ such that
$G=G_1\times\cdots\times G_r\subseteq \sym_n$ is a direct product and assume $G_j^+\neq G_j$ for
{all} $1\leq j\leq r$. Fix $\sigma_j\in G_j\backslash G_j^+$.  Let
$\chi=\chi_1\otimes\cdots\otimes\chi_r\in\Irr(G)$ be such that
$\chi_j\in\Irr(G_j)$ for all $1\leq j\leq r$.  First, we remark that
$\sgn=\sgn\otimes\cdots\otimes\sgn\in\Irr(G)$, thus
\begin{align}
\chi^*&=\chi\otimes\sgn
\label{eq:conjugateproduit}\\
&=(\chi_1\otimes\cdots\otimes\chi_r)\otimes(\sgn\otimes\cdots\otimes\sgn)
\nonumber\\
&=(\chi_1\otimes\sgn)\otimes\cdots\otimes(\chi_r\otimes\sgn)\nonumber\\
&=\chi^*_1\otimes\cdots\otimes\chi^*_r.\nonumber
\end{align}
In particular
\begin{equation}
\label{eq:fixdirect}
\chi=\chi^*\quad\Longleftrightarrow\quad
\chi_j=\chi^*_j\quad
{\text{for all }} 1\leq j\leq r. 
\end{equation}

Suppose $\chi=\chi^*$. Write $N=G_1^+\times\cdots\times G_r^+$. Then $N$
is a normal subgroup of $G$. For
$\underline\epsilon=(\epsilon_1,\ldots,\epsilon_r)\in\{-1,1\}^r$, we set 
$$\chi_{\underline
\epsilon}=\chi_1^{\epsilon_1}\otimes\cdots\otimes\chi_r^{\epsilon_r}.$$
Consider a constituent $\phi$ of $\Res_N^G(\chi)$. There is
$\underline\alpha=(\alpha_1,\ldots,\alpha_r)\in\{-1,1\}^r$ such that
$\phi=\chi_{\underline\alpha}$. Furthermore,
$${}^x\phi=\chi_1^+\otimes\chi_2^+\otimes\cdots\otimes\chi_r^+,$$
where $x=\prod_{\alpha_j=-1}\sigma_j$. 
It follows that the $G$-orbit
of $\phi$ is $\mathcal
O=\{\chi_{\underline\epsilon}\,|\,\underline\epsilon\in\{-1,1\}^r\}$,
and Clifford theory gives
$$\Res_N^G(\chi)=\sum_{\underline\epsilon\in\{-1,1\}^r}\chi_{\underline\epsilon}.$$

On the other hand, $N$ is contained in $G^+$, thus
$\chi_1^+\otimes\chi_2^+\otimes\cdots\otimes\chi_r^+$ is a constituent
of the restriction to $N$ of either $\chi^+$ or $\chi^-$.  
Without loss of generality, we choose it to be a constituent of the restriction of $\chi^+$.
Now, for $\eta\in\{-1,1\}$, we set
$$R^\eta=\{(\epsilon_1,\ldots,\epsilon_r)\in\{-1,1\}^r\,|\,\epsilon_1\cdots\epsilon_r=\eta\}\quad\text{and}\quad\mathcal
O^\eta=\{\chi_{\underline\epsilon}\,|\,\underline\epsilon\in
R^{\eta}\}.$$
 Let $(\epsilon_1,\ldots,\epsilon_r)\in R^+$. Define
$x=\prod_{\epsilon_j=-1}\sigma_j$. Since the number of $1\leq j\leq r$ with
$\epsilon_j=-1$ is even, we deduce that $x\in G^+$ and
${}^x(\chi_1^+\otimes\chi_2^+\otimes\cdots\otimes\chi_r^+)=\chi_1^{\epsilon_1}\otimes\cdots\otimes\chi_r^{\epsilon_r}$.  
In particular, the characters of $\mathcal O^+$ lie in the same
$G^+$-orbit.  By Clifford theory, $\mathcal O$ decomposes into two
$G^+$-orbits of the same size. {Since $|\mathcal O^+|=|R^+|=
|R^-|=|\mathcal O^-|$, 
and $\mathcal O^+\sqcup \mathcal O^-=\mathcal O$}, we deduce that
$\mathcal O^+$ and $\mathcal O^-$ are the two $G^+$-orbits of $\mathcal O$. Again, by Clifford theory, we obtain that
\begin{equation}
\label{eq:restoN}
\Res_N^{G^+}(\chi^{\eta})=\sum_{\underline\epsilon\in
R^\eta}\chi_{\underline\epsilon}. 
\end{equation}

\begin{remark}
\label{rk:splitclassproduit}
{Let $g=g_1\cdots g_r\in G^+$ with $g_j\in G_j$ for $1\leq j\leq r$}. Then $g$
lies in a split class of $G$ if and only if $g_j$ lies in a split class
of $G_j$ for all $1\leq j\leq r$. 
Indeed, $g$ lies in a split class of $G$ if and only if
$\operatorname{C}_G(g)=\operatorname{C}_{G+}(g)$. Assume some $g_j$ does
not belong {to} a split class of $G_j$. If $g_j\notin G_j^+$, then
$g_j\in\operatorname{C}_{G}(g)\backslash\operatorname{C}_{G^+}(g)$. If
$g_j\in G_j^+$, then there is
$x\in\operatorname{C}_{G_j}(g_j)\backslash\operatorname{C}_{G_j^+}(g_j)$,
and $x=1\cdots 1 x 1\cdots
1\in\operatorname{C}_{G}(g)\backslash\operatorname{C}_{G^+}(g)$.
Conversely, suppose that $g_j$ lies in a split class of $G_j$ for all
$1\leq j\leq r$. Then
$\operatorname{C}_{G_j}(g_j)=\operatorname{C}_{G_j^+}(g_j)$, so that 
$$\operatorname{C}_G(g)=\operatorname{C}_{G_1}(g_1)\cdots\operatorname{C}_{G_r}(g_r)=\operatorname{C}_{G_1^+}(g_1)\cdots\operatorname{C}_{G_r^+}(g_r)=\operatorname{C}_N(g)\leq
\operatorname{C}_{G^+}(g)\leq\operatorname{C}_{G}(g),$$
and  $\operatorname{C}_G(g)=\operatorname{C}_{G^+}(g)$, as required.
\end{remark}

\begin{proposition}
\label{prop:reducbouge}
Write $G=G_1\times\cdots\times G_r\subseteq\sym_n$ as above. 
Let $\chi\in\Irr(G)$ be such that $\chi=\chi^*$, and let
$\chi_j^{+}$ and $\chi_j^-$ be as above. {For $f\in \mathcal G_{|G|}$
such that ${}^f\chi=\chi$,
with the notation~(\ref{eq:signebouge})}, we have
$$\varepsilon(\chi,f)=\prod_{j=1}^r\varepsilon(\chi_j,f).$$
\end{proposition}

\begin{proof}
Let $\eta\in\{-1,1\}$.
First, we remark that either ${}^f\chi^{\eta}=\chi^{\eta}$ or
${}^f\chi^{\eta}=\chi^{-\eta}$ because $f$ fixes $\chi$. So, the set 
$\mathcal
O$ is $f$-stable, and $f$ acts on $\{\mathcal O^+,\mathcal O^-\}$.
Furthermore, ${}^f\chi^{\eta}=\chi^{\eta}$ if
and only if $f(\mathcal O^\eta)=\mathcal O^\eta$. However, we
have $f(\mathcal O^+)=\mathcal O^+$ if and only if
$f(\chi_1^+\otimes\chi_2^+\otimes\cdots\otimes\chi_r^+)\in\mathcal O^+$
if and only if the number of $1\leq j\leq r$ such that $\chi_j^{+}$ are
{not} fixed by $f$ is even.  The result follows. 
\end{proof}

\subsection{Irreducible characters of $Y^k$ and of $(Y^k)^+$}
\label{subsec:irrYk}
Write $I=\{0,\ldots,p-1\}$ as above. We now describe how to construct
the characters of $\Irr(Y)$ in~(\ref{eq:irrY}). {For $0\leq j\leq
p-2$}, define the linear character $\zeta_j:Y \rightarrow \C^*$ by setting
$\zeta_j(a^ub^v)=\omega_{p-1}^{jv}$, and write $\zeta$ for the induced
character of any non-trivial character of $\langle a\rangle$ to $Y$. In
particular, $\zeta_j(1)=1$ for all $0\leq j\leq p-2$, and
$\zeta(1)=p-1$.  Set $p^*=(p-1)/2$. Since $\sgn$ is the only linear
character of $Y$ of order $2$, we have $\sgn=\zeta_{p^*}$ and
$\zeta_j^*=\zeta_{p^*+j}$.  On the other hand,
$\{0,\ldots,p-2\}=\{0,\ldots, p^*-1\}\cup \{p^*,p^*+1,\ldots,2p^*-1\}$.
So, in~(\ref{eq:irrY}) we set $\xi_{p^*}=\zeta$, $\xi_j=\zeta_j$ and
$\xi_{p-1-j}=\zeta_{p^*+j}$ for all $j\in\{0,\ldots,p^*-1\}$.

Note that $\langle a\rangle$ and $\langle b^2\rangle$ are subgroups of
$Y^+$. By an order argument, we obtain that $Y^+=\langle a\rangle \rtimes
\langle b^2\rangle$. By Clifford theory, the characters $\xi_{j}$ and
$\xi_{p-1-j}$ for $0\leq j\leq p^*-1$ restrict to the same linear
character of $Y^+$, also denoted by $\xi_{j}$, and $\xi_{p^*}$ splits
into two irreducible characters $\xi_{p^*}^+$ and $\xi_{p^*}^-$ of
degree $(p-1)/2$. Now, we will specify the values of  $\xi_{p^*}^+$ and
$\xi_{p^*}^-$. For every $0\leq j\leq p-1$, set $\alpha_j:\langle
a\rangle \rightarrow \C^*,\,a^k\mapsto \omega_p^{jk}$. 
Write $u$ for the integer such that $bab^{-1}=a^u$. Since for all $0\leq
j\leq p-1$, $0\leq k\leq p-1$, and $0\leq l\leq p-2$
$${}^{b^l}\alpha_j(a^k)=\alpha_j(a^{u^l k})=\omega_p^{u^lkj}=\alpha_{u^l j}(a^k),$$
we deduce that the $\langle b^2\rangle$-orbits on $\Irr(\langle a\rangle)$
are
$$\{\alpha_0\},\quad\{\alpha_j\,|\,j\in
S\}\quad\text{and}\quad\{\alpha_j\,|\,j\in
\overline S\},$$
where $S$ is the set of
square elements of $(\Z/(p-1)\Z)^{\times}$, and $\overline S$ the non-squares. Then by Clifford theory with
respect to the normal subgroup $\langle a\rangle$ of $Y^+$, we can
choose the labels such that
\begin{equation}
\label{eq:eqresyplustoa}
\Res_{\langle a\rangle}^{Y^+}(\xi_{p^*}^{+})=\sum_{j\in
S}\alpha_j\quad\text{and}\quad
\Res_{\langle a\rangle}^{Y^+}(\xi_{p^*}^{-})=\sum_{j\in
\overline S}\alpha_j,
\end{equation}
and the inertial subgroup in $Y^+$ of $\alpha_j$ with $j\neq 0$ is
$\langle a\rangle$, hence $\xi_{p^*}^+$ and $\xi_{p^*}^-$ are the
induced characters to $Y^+$ of $\xi_j$ with $j\in S$ and
$j\in\overline S$, respectively. Thus, $\xi_{p^*}^{\pm}$ vanishes
outside $\langle a\rangle$, and
using~(\ref{eq:eqresyplustoa}) and~\cite[Thm.\,1 p.75]{KennethRosenNrTh}, we obtain
\begin{equation}
\label{eq:cardiffY}
(\xi_{p^*}^+-\xi_{p^*}^-)(a)=\sum_{j\in
S}\alpha_j(a)-\sum_{j\in
\underline
S}\alpha_j(a)=\sum_{j=1}^{p}\left(\frac{j}p\right)\omega_p^j=i^{(p-1)/2}\sqrt
p,
\end{equation}
where $\left(\frac{j}p\right)$ is the Legendre symbol and {$i$ is a
complex square root of $-1$}. This in
particular proves that $Y$ has only one split class, with
representative $a$. We write $a^+=a$, and $a^-\in\langle a\rangle$
for an element conjugate to $a$ in $Y$ but not in $Y^+$. 

Let $k\geq 1$ {be} an integer. By Remark~\ref{rk:splitclassproduit}, the
group $Y^k$ has only one split class with representative
$\underline a={(a,a,\ldots,a)}$. Furthermore, Section \S\ref{subsec:redpb}
implies that $Y^k$ has only one $\sgn$-stable character
$\xi_{\underline p^*(k)}=\xi_{p^*}\otimes\cdots\otimes \xi_{p^*}$,
where 
\begin{equation}
\label{eq:firstetoile}
{\underline p^*(k)=((p-1)/2,\ldots,(p-1)/2)\in I^k}.
\end{equation}
Set $\alpha=\xi_{p^*}^+(a)$ and $\beta=\xi_{p^*}^-(a)$, and for
$\underline\epsilon\in\{-1,1\}^k$, denote by $\mathfrak
n(\underline\epsilon)$ {the} number of $1\leq j\leq k$ such that
$\epsilon_j=-1$. {Now,} Equation~(\ref{eq:restoN}) gives
\small
\begin{align*}
(\xi_{\underline p^*(k)}^+-\xi_{\underline p^*(k)}^-)(\underline
a)&=\sum_{\underline \epsilon\in
R^+}\chi_{\underline \epsilon}(a)-\sum_{\underline\epsilon\in
R^-}\chi_{\underline\epsilon}(a)\\
&=\sum_{0\leq 2j\leq k}\sum_{\mathfrak{n}(\underline
\epsilon)=2j}\alpha^{2j}\beta^{k-2j}
-\sum_{0<2j+1\leq k}\sum_{\mathfrak{n}(\underline
\epsilon)=2j+1}\alpha^{2j+1}\beta^{k-2j-1}
\\
&=\sum_{0\leq 2j\leq k}\binom{k}{2j}(-1)^{2j}\alpha^{2j}\beta^{k-2j}
+\sum_{0<2j+1\leq k}\binom{k}{2j+1}(-1)^{2j+1}\alpha^{2j+1}\beta^{k-2j-1}\\
&=\sum_{j=0}^{k}\binom{k}{j}(-1)^{j}\alpha^{j}\beta^{k-j}\\
&=(\alpha-\beta)^k
\end{align*}
\normalsize
by {Newton's binomial formula}. Finally we deduce
from~(\ref{eq:cardiffY}) that 
\begin{equation}
\label{eq:diffcarYk}(\xi_{\underline p^*(k)}^+-\xi_{\underline p^*(k)}^-)
(\underline a)=i^{k(p-1)/2}\sqrt{p^k}.
\end{equation}
\subsection{Irreducible characters of $(Y^k\wr\sym_w)^+$}

Let $k$ and  $w$ be two positive integers. 
In this section, we set $$N=Y^k\wr\sym_w\quad\text{and}\quad M=(Y^k)^w.$$
For $\underline j=(j_0,\ldots,j_{k-1})\in I^k$, define
$$\pi_p(\underline j)=j_{k-1}+j_{k-2}p+\cdots+j_{0}p^{k-1}.$$
By the uniqueness of the $p$-adic expansion of a positive integer, we note that the
map $\pi_p:I^k\rightarrow \{0,\ldots,p^k-1\}$ is a bijection. 

We now generalize the Equation~(\ref{eq:firstetoile}) 
defining an involution $*$ on $I^k$, by setting
$$\underline j^*=(p-1-j_0,\ldots,p-1-j_{k-1})\in I^k.$$

\begin{lemma}
\label{lemma:starYk}
For $\underline j\in I^k$, one has
$$\pi_p(\underline j^*)=p^k-1-\pi_p(\underline
j)\quad\text{and}\quad\xi_{\underline j}^*=\xi_{\underline j^*}.$$ 
\end{lemma}

\begin{proof}
We have
\begin{align*}
\pi_p(\underline
j^*)&=(p-1-j_{k-1})+(p-1-j_{k-2})p+\cdots+(p-1-j_{0})p^{k-1}\\
&={(p-1)(1+p+\cdots+p^{k-1})}-\pi_p(\underline j)\\
&=p^k-1-\pi_p(\underline j).
\end{align*} 
\end{proof}

We follow the convention of
Remark~\ref{rk:lexico} to label $\Irr(N)$. Moreover,
for any $\underline
\lambda=(\lambda_0,\ldots,\lambda_{p^k-1})\in\mathcal{MP}(p^k,w)$, define
$$\underline\lambda^*=(\lambda_{p^k-1}^*,\lambda_{p^k-2}^*,\ldots,\lambda_1^*,\lambda_0^*)\in\mathcal{MP}(p^k,w),$$
where $\lambda^*$ denotes the conjugate partition of $\lambda$.
To simplify the notation~(\ref{eq:psiNk}), we set
$\psi_{\underline\lambda,k}=\psi_{\underline \lambda}$.

\begin{lemma}
\label{lemma:starlocal}
If $\underline j\in I^k$ and $\underline{\lambda}=(\lambda_{\underline
j};\underline j\in I^k)\in\mathcal{MP}(p^k,w)$, then
$$\psi_{\underline\lambda}^*=\psi_{\underline\lambda^*}.$$ 
\end{lemma}

\begin{proof}
Let $g=\prod_{\underline j\in I^k}(y_{\underline
j};\sigma_{\underline j})\in N_{k,\underline\lambda}$.
Then $g=\prod_{\underline j\in
I^k}\prod_{\tau\in C(\sigma_{\underline j})}(y_{\underline
j,\tau},\tau)$,
$$\sgn(g)=\prod_{\underline j\in
I^k}\prod_{\tau\in C(\sigma_{\underline j})}\sgn(y_{\underline
j,\tau})\sgn(\tau),$$
because $\sgn$ is a group homomorphism, and $\sgn(y_{\underline
j,\tau})=\sgn(\mathfrak c((y_{\underline j};\sigma_{\underline
j}),\tau))$. Hence
\begin{align*}
(E(\xi_{\underline \lambda})\chi_{\underline\lambda})^*(g)&=\sgn(g)E(\xi_{\underline
\lambda})(g)\chi_{\underline\lambda}(g)\\
&=\sgn(g)\prod_{\underline j\in I^k}\prod_{\tau\in C(\sigma_{\underline
j})}\xi_{\underline j}{(\mathfrak c((y_{\underline j};\sigma_{\underline
j}),\tau))}\chi_{\underline\lambda}(\sigma)\\
&=\sgn(\sigma)\prod_{\underline j\in I^k}\prod_{\tau\in C(\sigma_{\underline
j})}\sgn(\mathfrak c((y_{\underline j};\sigma_{\underline
j}),\tau))\xi_{\underline j}(\mathfrak c((y_{\underline j};\sigma_{\underline
j}),\tau))\chi_{\underline\lambda}(\sigma)\\
&=\sgn(\sigma)\prod_{\underline j\in I^k}\prod_{\tau\in C(\sigma_{\underline
j})}\xi_{\underline j}^*(\mathfrak c((y_{\underline j};\sigma_{\underline
j}),\tau))\chi_{\underline\lambda}(\sigma)\\
&=\sgn(\sigma)\prod_{\underline j\in I^k}\prod_{\tau\in C(\sigma_{\underline
j})}\xi_{\underline j^*}(\mathfrak c((y_{\underline j};\sigma_{\underline
j}),\tau))\chi_{\underline\lambda}(\sigma)\quad(\text{by
Lemma~\ref{lemma:starYk}})\\
&=\prod_{\underline j\in I^k}\prod_{\tau\in C(\sigma_{\underline
j})}\xi_{\underline j^*}(\mathfrak c((y_{\underline j};\sigma_{\underline
j}),\tau))\chi_{\underline\lambda}^*(\sigma),\\
\end{align*}
where $\sigma=\prod \sigma_{\underline j}$. 
 
Let $w_{\underline\lambda}\in \sym_w$ be the permutation that sends the
support of $\lambda_{\underline j}$ to that of $\lambda^*_{\underline j^*}$.
So,
$\sym_{\underline\lambda^*}
={}^{w_{\underline\lambda}}\sym_{\underline\lambda}$, and the decomposition of 
${}^{w_{\underline\lambda}}g$ with respect to $N_{k,\underline\lambda^*}$ is
$${}^{w_{\underline\lambda}}g=\prod_{\underline j\in
I^k}({}^{w_{\underline\lambda}}y_{\underline
j};{}^{w_{\underline\lambda}}\tau_{\underline j}),$$ and since  
{$\mathfrak c((y_{\underline j};\sigma_{\underline j}),\tau)=
\mathfrak c(({}^{w_{\underline\lambda}}y_{\underline
j};{}^{w_{\underline\lambda}}\sigma_{\underline
j}),{}^{w_{\underline\lambda}}\tau)$, we deduce that}
\begin{align*}
{}^{w_{\underline\lambda}}E(\xi_{\underline\lambda^*})(g)&=\prod_{\underline j\in
I^k}\prod_{\tau\in C(\sigma_j)}\xi_{\underline j^*}(\mathfrak
c(({}^{w_{\underline\lambda}}y_{\underline j};{}^{w_{\underline\lambda}}\sigma_{\underline
j}),{}^{w_{\underline\lambda}}\tau))
\\&=\prod_{\underline j\in
I^k}\prod_{\tau\in C(\sigma_j)}\xi_{\underline j^*}(\mathfrak c((y_{\underline j};\sigma_{\underline
j}),\tau)).
\end{align*}
Since
${}^{w_{\underline\lambda}}\chi_{\underline\lambda^*}=\chi_{\underline\lambda}^*$, we obtain
\begin{equation}
\label{eq:eqpr}
(E(\xi_{\underline\lambda})\chi_{\underline\lambda})^*={}^{w_{\underline\lambda}}(E(\xi_{\underline\lambda^*})\chi_{\underline\lambda^*}). 
\end{equation}
It follows that
\begin{align*}
\psi_{\underline\lambda}^*&=\sgn\Ind_{N_{k,\underline\lambda}}^{N}(E(\xi_{\underline\lambda})\chi_{\underline\lambda})\\
&=\Ind_{N_{k,\underline\lambda}}^{N}(\sgn
E(\xi_{\underline\lambda})\chi_{\underline\lambda})\\
&=\Ind_{N_{k,\underline\lambda}}^{N}(E(\xi_{\underline\lambda})\chi_{\underline\lambda})^*\\
&=\Ind_{w_{\underline\lambda}^{-1}N_{k,\underline\lambda^*}w_{\underline\lambda}}^{N}\!\!\!{}^{w_{\lambda}}(E(\xi_{\underline\lambda^*})\chi_{\underline\lambda^*})\\
&=\Ind_{N_{k,\underline\lambda^*}}^{N}\!\!\!(E(\xi_{\underline\lambda^*})\chi_{\underline\lambda^*})\\
&=\psi_{\underline\lambda^*},
\end{align*}as required.
\end{proof}

\bigskip
\begin{lemma}
\label{lemma:extensionsemidirect}
Let $G$ be a group, and $H$ and $K$ be subgroups of $G$. Let
$x\in G$ be such that $x$ normalizes $H$ and $K$, and $H\cap \langle
x\rangle=K\cap \langle x\rangle=1$. {Let $t\in\operatorname{N}_K(H)$}.
For every 
$g\in\langle tx\rangle$, write $g=g_tg_x$
for unique
$g_t\in K$ and $g_x\in\langle x\rangle$. Assume there is a
representation $\rho:H\rightarrow \GL(V)$ that extends to a
representation
$\widetilde{\rho}:H\rtimes\langle x\rangle\mapsto\GL(V)$. 
If $\rho({}^{g_t}h)=\rho(h)$ for all $h\in H$ and $g\in\langle
tx\rangle$,
then the map
{$$\varphi:H\rtimes\langle tx\rangle\rightarrow\GL(V),\,hg\mapsto
\widetilde{\rho}(hg_x)$$}
is a representation of $H\rtimes\langle tx\rangle$. 
\end{lemma}

\begin{proof}
First, we remark that if $g=\langle tx\rangle$, then there is an integer
$j$  such that
$g=(tx)^j=t{}^xt\cdots {}^{x^{j-1}}t x^j$, so $g_t=t{}^xt\cdots
{}^{x^{j-1}}t\in K$ and $g_x=x^j$ because $x$ normalizes $K$.
Furthermore, this expression is unique because $K\cap \langle x\rangle=1$. 
Note also that
if $g,\,g'\in \langle tx\rangle$, then $(gg')_{x}=g_xg'_x$. 
Now, for $h,\,h'\in H$ and $g,\,g'\in \langle tx\rangle$, we have
$hgh'g'=h{}^gh'gg'=h{}^gh'(gg')_{t}g_xg'_x$. Thus,
\begin{align*}
\varphi(hgh'g')&=\widetilde\rho(h{}^gh'g_xg'_{x})\\
&=\rho(h)\rho({}^{g_t}({}^{g_x}h'))\widetilde{\rho}(g_xg'_x)\\
&=\rho(h)\rho({}^{g_x}h')\rho(g_xg'_x)\\
&=\widetilde{\rho}(h{}^{g_x}h'g_xg'_x)\\
&=\widetilde{\rho}(hg_xh'g'_x)\\
&=\widetilde{\rho}(hg_x)\widetilde{\rho}(h'g'_x)\\
&=\varphi(hg)\varphi(h'g'),
\end{align*} 
as required.
\end{proof}

A multipartition $\underline\lambda=(\lambda_{\underline j};\,j\in
I^k)\in\mathcal{MP}(p^k,w)$ is called \emph{symmetric} if
$\underline\lambda^*=\underline\lambda$. We denote by
{$\mathcal{SP}({p^k,w})$} the set of symmetric multipartitions of
{$\mathcal{MP}({p^k,w})$.}
Let $\underline c=(c_{\underline j}\in\N;\,\underline j\in I^k)$ be such that
$\sum_{\underline j}c_{\underline j}=w$, and
$c_{\underline j}=c_{\underline j^*}$ for all $\underline j\in I^k$.
Define 
$$\mathcal P_{\underline
c}=\{\underline\lambda\in{\mathcal{SP}({p^k,w})}\,|\,\ \forall \underline
j\in I^k, |\lambda_{\underline
j}|=c_{\underline j}\}.$$

For any $\underline\lambda\in\mathcal P_{\underline c}$, the characters
$\xi_{\underline\lambda}$ and their inertial subgroup
$N_{k,\underline\lambda}$ depend only on $\underline c$.
We write
$\xi_{\underline c}$ and $N_{\underline c}$ in the following.

\begin{proposition}
\label{prop:signeregular}
Let $\underline\lambda\in\mathcal{SP}(p^k,w)$ be such that
$\lambda_{\underline p^*(k)}=\emptyset$.
If $f\in\mathcal K_{n!/2}$, then 
$\varepsilon(\psi_{\underline\lambda},f)=1$.
Furthermore,
$$\varepsilon(\psi_{\underline\lambda},\sigma_{n!/2})=(-1)^{\frac{(p-1)w}{4}}.$$
\end{proposition}

\begin{proof}
Let $\underline c=(c_{\underline j};\,\underline j\in I^k)$ be such that
$c_{\underline p^*(k)}=0$. {Furthermore, since $\underline\lambda$ is a
symmetric multipartition, $c_{\underline j}=c_{\underline j^*}$ and it
follows that $$w=\sum_{\{\underline
j,\underline j^*\},\, \underline j\neq \underline p^*(k)}(c_{\underline
j}+c_{\underline j^*})=2\sum_{\{\underline
j,\underline j^*\},\, \underline j\neq \underline p^*(k)}c_{\underline
j},$$
hence $w$ is even.}
By Clifford theory with respect to the normal subgroup $M$ of $N$, the
characters $\psi_{\underline\lambda}$ for $\underline\lambda\in\mathcal
P_{\underline c}$ are the constituents of
$\Ind_M^{N}(\xi_{\underline c})$. Write
$\vartheta_{\underline c}$ for the restriction of
$\xi_{\underline c}$ to $M^+$. Since $\xi_{\underline c}$ is
not $\sgn$-stable, we have $\vartheta_{\underline c}\in\Irr(M^+)$
by Clifford theory {with respect to} 
$M^+\unlhd M$. Furthermore, Mackey's formula
gives
{$$\Res_{N^+}^{N}\Ind_{M}^{N}(\xi_{\underline
c})=\Ind_{M^+}^{N^+}(\vartheta_{\underline c}).$$}
Hence, the irreducible characters $\psi_{\underline\lambda}^+$ and
$\psi_{\underline\lambda}^-$ for $\underline\lambda\in\mathcal
P_{\underline c}$ appear in the Clifford theory with respect to
$M^+\unlhd N^+$ associated to the character
$\vartheta_{\underline c}$.
Denote by $T_{\underline c}$ the inertial subgroup of
$\vartheta_{\underline c}$ with respect to $M^+\unlhd N^+$.

Let $\underline\lambda\in\mathcal P_{\underline c}$. The character
$\vartheta_{\underline c}$ is $M$-stable, thus $$\langle
\Ind_{M^+}^{G^+}(\vartheta_{\underline c}),\psi_{\underline\lambda}^+\rangle=
\langle
\Ind_{M^+}^{G^+}(\vartheta_{\underline c}),\psi_{\underline\lambda}^-\rangle.$$
We also have
$\Ind_{M^+}^M(\vartheta_{\underline c})=\xi_{\underline
c}+\xi_{\underline c}^*$,
and the last two characters are $N$-conjugate, in particular,
$\Ind_M^{N}(\xi_{\underline c})=\Ind_M^{N}(\xi_{\underline c}^*)$.
Now, we deduce from Frobenius reciprocity that
\begin{align*}
\langle\Ind_{M^+}^{N^+}(\vartheta_{\underline c}),\psi_{\underline\lambda}^+\rangle&=
\frac 1 2
\langle\Ind_{M^+}^{N^+}(\vartheta_{\underline c}),\psi_{\underline\lambda}^++\psi_{\underline\lambda}^-\rangle\\&=
\frac 1 2
\langle\Ind_{M^+}^{N^+}(\vartheta_{\underline c}),\Res_{N^+}^{N}(\psi_{\underline\lambda})\rangle\\&=
\frac 1 2
\langle\Ind_{M^+}^{N}(\vartheta_{\underline c}),\psi_{\underline\lambda}\rangle\\
&=
\frac 1 2
\langle\Ind_{M}^{N}\Ind_{M^+}^M(\vartheta_{\underline c}),\psi_{\underline\lambda}\rangle\\
&=\frac 1 2
\langle \Ind_{M}^{N}(\xi_{\underline c}+\xi_{\underline c}^*),\psi_{\underline\lambda}\rangle\\
&=\frac 1 2
\langle 2\Ind_{M}^{N}(\xi_{\underline c}),\psi_{\underline\lambda}\rangle\\
&=\langle\Ind_{M}^{N}(\xi_{\underline c}),\psi_{\underline\lambda}\rangle.
\end{align*} 
Let $t$ and $t'$ be the number\ of $N$-conjugate characters of
$\xi_{\underline c}$ and of $N^+$-conjugate characters of
$\vartheta_{\underline c}$, respectively. Then, by Clifford theory,
if
$e=\langle\Ind_{M}^{N}(\xi_{\underline c}),\psi_{\underline c}\rangle$,
then
$$\psi_{\underline\lambda}(1)=et\xi_{\underline c}(1)\quad\text\quad
\psi_{\underline\lambda}(1)^+=et'\vartheta_{\underline c}(1).$$
Hence, $2t'=t$ because $\vartheta_{\underline c}(1)=\xi_{\underline
c}(1)$ and
$2\psi_{\underline\lambda}(1)^+=\psi_{\underline\lambda}(1)$.

Note that $N_{\underline c}^+\leq T_{\underline c}$ and that
$N/N_{\underline c}\simeq N^+/N_{\underline c}^+$, and
$$t=\frac{|N|}{|N_{\underline c}|}=\frac{|N^+|}{|N_{\underline c}^+|}\quad\text{and}\quad
t'=\frac{|N^+|}{|T_{\underline c}|}.$$
Then $T_{\underline c}$ is an extension of degree $2$ of
$N_{\underline c}^+$. 
Since $\underline\lambda$ is symmetric and $\lambda_{\underline
p^*(k)}=\emptyset$, the permutation
$w_{\underline\lambda}$ defined in the proof of
Lemma~\ref{lemma:starlocal} is an involution that exchanges the supports of
$\lambda_{\underline j}$ and $\lambda_{\underline j^*}$ for all
$\underline j\in I^k$. We remark that $w_{\underline\lambda}$ is the
same element for any $\underline\lambda\in \mathcal P_{\underline c}$, we will denote it by $w_{\underline c}$.
Denote by $\theta_{\underline\lambda}$ the
restriction  of $E(\xi_{\underline\lambda})\chi_{\underline\lambda}$ to
$N_{\underline\lambda,c}^+$ which is irreducible because
$E(\xi_{\underline\lambda})\chi_{\underline\lambda}\neq
(E(\xi_{\underline\lambda})\chi_{\underline\lambda})^*$. 
Then for all $g\in N^+$, 
\begin{equation}
\label{eq:rest}
\theta_{\underline\lambda}(g)=E(\xi_{\underline\lambda})\chi_{\underline\lambda}(g)=(E(\xi_{\underline\lambda})\chi_{\underline\lambda})^*(g)={}^{w_{\underline
c}}(E(\xi_{\underline\lambda})\chi_{\underline\lambda})(g)={}^{w_{\underline
c}}\theta_{\underline\lambda}(g)
\end{equation}
by Equation~(\ref{eq:eqpr}).
Let $h\in N_{\underline c}\backslash N_{\underline c}^+$.
We set $t_{\underline c}=w_{\underline c}$ if $w_{\underline c}\in
N^+$, and $t_{\underline c}=hw_{\underline c}$ otherwise. We remark
that
\begin{equation}
\label{eq:sgnwc} 
\sgn(w_{\underline c})=(-1)^{w/2}.
\end{equation}

Now, we define $\underline\mu$ as follows. For any $\underline j\in I^k$
such that $\pi_{p}(\underline j)<(p^k-1)/2$, set $\mu_{\underline
j}=(c_{\underline j})$ and $\mu_{\underline
j^*}=(1^{c_{\underline j}})$, and $\mu_{\underline p^*(k)}=0$. So, $\mu\in\mathcal P_{\underline c}$, and
$\Res_{M}^{N_{\underline c}}(\theta_{\underline\mu})=\xi_{\underline c}$. In
particular, Equation~(\ref{eq:rest}) gives
$$T_{\underline c}=\langle
N_{\underline c}^+,t_{\underline c}\rangle.$$ 
Since $T_{\underline c}$ is a cyclic extension of $N_{\underline
c}^+$, by~\cite[11.22]{isaacs} we can extend $\theta_{\underline\mu}$ to a character
$\widetilde{\theta}_{\underline\mu}$ of $T_{\underline c}$. Thus, by Gallagher's theorem (see~\cite[6.17]{isaacs}), the constituents of
$\Ind_{M^+}^{N^+}(\vartheta_{\underline c})$ are 

\begin{equation}
\label{eq:defrhoalpha}
\rho_{\alpha}=\Ind_{T_{\underline
c}}^{N^+}(\widetilde{\theta}_{\underline\mu}\otimes\alpha),
\end{equation}
where $\alpha$ is any irreducible characters of $T_{\underline c}/M^+$
lifted through $T_{\underline c}\rightarrow T_{\underline c}/M^+$.

If we write
$H_{\underline c}=\langle N_{\underline c},t_{\underline
c}\rangle$, then $H_{\underline c}^+=T_{\underline c}$ and
$$H_{\underline c}/M\simeq T_{\underline c}/M^+.$$
However, if we choose $h\in M\backslash M^+$, then $t_{\underline
c}^2\in M$ and the image of $t_{\underline c}$ in $H_{\underline c}/M$
has order $2$, and can be identified with $w_{\underline c}$. It follows
that
$$H_{\underline c}/M\simeq \sym_{\underline c}\rtimes \langle
w_{\underline c}\rangle.$$
Set $L=\sym_{\underline c}\rtimes\langle w_{\underline c}\rangle$. We
now will prove that the irreducible characters of $L$ are
integer-valued. Let $\phi\in\Irr(\sym_{\underline c})$. If $\phi$ is not
$L$-stable, then $\widetilde{\phi}=\Ind_{\sym_{\underline c}}^{L}(\phi)\in\Irr(L)$ 
and $\widetilde{\phi}(g)=\phi(g)+{}^{w_{\overline c}}\phi(g)\in \Z$ if
$g\in \sym_{\underline c}$ and $0$ otherwise.

Assume $\phi$ is $L$-stable. Then $\phi$ extends to $L$ (because $L$ is
a cyclic extension of $\sym_{\underline c}$) and has exactly two
extensions $\widetilde{\phi}$ and $\widetilde{\phi}\otimes\varepsilon$,
where $\varepsilon$ is the lift of the non-trivial character of $\langle
w_{\underline c}\rangle$.
Now, for $\underline j\in I^k$ such that $\pi_p(\underline j)<(p^k-1)/2$,
write $\tau_{\underline j}$ for the involution that exchanges the
supports of $c_{\underline j}$ and $c_{\underline j^*}$. One has
$w_{\underline c}=\prod \tau_{\underline j}$. Since $c_{\underline
j}=c_{\underline j^*}$, the group $L$ can be viewed as a subgroup of
$$L'=\prod_{\pi_p(\underline
j)<(p^k-1)/2}\sym_{c_{\underline j}}\wr \langle\tau_{\underline
j}\rangle.$$
Since $\phi$ is $L$-stable, we must have $\phi_{\underline j}=\phi_{\underline
j^*}$, and $\phi$ is $\tau_{\underline j}$ stable for all $\underline
j$. Thus, $\phi$ is $L'$-stable and can be extend to $L'$ because $L'$
is a direct product of wreath products isomorphic to
$\sym_{c_{\underline j}}\wr\sym_2$. Denote by $E(\phi)$
the James-Kerber extension as above. By~(\ref{eq:jamesextension}),
$E(\phi)$ takes integer values. However,
$\Res^{L'}_{L}(E(\phi))$ is either $\widetilde \phi$ or
$\widetilde\phi\otimes\varepsilon$. Thus, $\widetilde\phi$ and
$\widetilde\phi\otimes\varepsilon$ also take integer values.

The argument above implies that any $\alpha\in\Irr(T_{\underline c}/M^+)$ takes integer
values. Let $f\in\mathcal H_{n!/2}$.  
By Proposition~\ref{prop:navarrosym}, $\theta_{\underline\mu}$ is
$f$-fixed. The two extensions of $\theta_{\underline\mu}$ to
$T_{\underline c}$ are $\widetilde{\theta}_{\underline\mu}$ and
$\widetilde{\theta}_{\underline\mu}\otimes \varepsilon$. Thus, either
$f(\widetilde{\theta}_{\underline\mu})=\widetilde{\theta}_{\underline\mu}$
or
$f(\widetilde{\theta}_{\underline\mu})=\widetilde{\theta}_{\underline\mu}\otimes\varepsilon$.
Then (\ref{eq:defrhoalpha}) and Lemma~\ref{lemma:induction} give
$f(\rho_{\alpha})=\rho_{\alpha}$ in the first case, and
$f(\rho_{\alpha})=\rho_{\alpha\otimes\varepsilon}$ in the second case.

On the other hand, 
$f(\widetilde{\theta}_{\underline\mu})=\widetilde{\theta}_{\underline\mu\otimes\varepsilon}$
if and only if $f(\widetilde{\theta}_{\underline\mu}(gt_{\underline
c}))=-\widetilde{\theta}_{\underline\mu}(gt_{\underline
c})$ for all {$g\in N_{\underline c}^+$} if and only if there exists
{$g_0\in N_{\underline c}^+$} such that
\begin{equation}
\label{eq:criterebouge}
\widetilde{\theta}_{\underline\mu}(g_0t_{\underline
c})\neq 0\quad \text{and}
\quad f(\widetilde{\theta}_{\underline\mu}(g_0t_{\underline
c}))=-\widetilde{\theta}_{\underline\mu}(g_0t_{\underline
c}).
\end{equation}

We will use this {criterion} to understand the action of $f$ on
$\widetilde\theta_{\underline\mu}$. 
Set $H=((Y^k)^+)^w$ and 
$$G=H\rtimes \langle w_{\underline c}\rangle.$$
For $\underline j$ such that $\pi_p(\underline j)<(p^k-1)/2$, define
$Y_{c_{\underline j}}=((Y^k)^+)^{2c_{\underline j}}\leq H$ 
corresponding to the supports of $\sym_{c_{\underline j}}$ and $\sym_{c_{\underline
j^*}}$. Then $G$ can be viewed as a subgroup of 
$$G'=\prod_{\pi_p(\underline j)<(p^k-1)/2} Y_{c_{\underline j}}\wr\langle \tau_{\underline j}\rangle,$$
where $\tau_{\underline j}$ is defined as before.

The character $\xi_{\underline c}$ is not $\sgn$-stable. It takes
non-zero values outside $M^+$, hence outside $H$, and
the restriction $\eta_{\underline c}$ of $\xi_{\underline c}$ to
$H$ is irreducible by Clifford theory with respect to
$H\unlhd M^+$. 
Moreover, if we write 
$\eta_{\underline j}=\Res^{Y^k}_{(Y^k)^+}(\xi_{\underline j})$, then
$\eta_{\underline j^*}=\eta_{\underline j}$. In particular, 
$$\eta_{\underline c}=\prod_{\pi_p(\underline
j)<(p^k-1)/2}\eta_{\underline j}^{c_{\underline
j}}\otimes\eta_{\underline j}^{c_{\underline j}}.$$
It follows that $\eta_{\underline c}$ extends to $G^+$, and
the James-Kerber extension $E(\eta_{\underline c})$ has integer values.
Hence $\Res_G^{G'}(E(\eta_{\underline c}))$ takes integer values, and by
Gallagher's theorem, the extension of $\eta_{\underline c}$ to $G$ takes a
non-zero and integer value on $w_{\underline c}$. 

Suppose $w\equiv 0\mod 4$. Then
by~(\ref{eq:sgnwc}), we take $t_{\underline c}=w_{\underline
c}$. 
By the previous discussion, $$\theta_{\underline c}(w_{\underline
c})=\pm
\Res_G^{G'}(E(\eta_{\underline c}))(w_{\underline c})\in\Z$$ is {a
non-zero integer}. We deduce from {criterion~(\ref{eq:criterebouge})} that
the characters of $N^+$ are fixed by all $f\in\mathcal H_{n!/2}$.

Suppose $w\equiv 2\mod 4$. 
Let $y\in M$. We label the components of $y$ as follows. For $\underline
j\in I^k$ such that $c_{\underline j}\neq 0$, write
$y_{\underline j}=(y_{\underline j,1},\ldots,y_{\underline
j,c_{\underline j}})\in (Y^k)^{c_{\underline j}}$, where $y_{\underline
j,i}=(y_{j_1,i},\ldots,y_{j_k,i})\in Y^k$ for all $1\leq i\leq
c_{\underline j}$. 
One has
$$\xi_{\underline c}(y)=\prod_{\underline j}\xi_{\underline
j}^{c_{\underline j}}(y_{\underline j}).$$  
Let $\underline u$ be such that
$c_{\underline u}\neq 0$. So $\underline u\neq \underline p^*(k)$, and
there is $u_r\neq 0$ with $r\neq (p-1)/2$. Let $h$ be the element of
$M$ that is trivial on any {component} of $Y^{kw}$ except $h_{u_r,1}=b$.
Set $h'={}^{w_{\underline c}}h$, which is the element of $M$ all of whose
components are trivial except $h'_{u_r^*,1}=b$. Since $h\notin
N^+$, by~(\ref{eq:sgnwc}) we take $t_{\underline
c}=hw_{\underline c}$. Remark that
$w_{\underline c}$ normalizes $H$ and $M^+$, $\langle w_{\underline
c}\rangle\cap H=\langle w_{\underline
c}\rangle\cap M^+=1$, and $h\in M^+$ normalizes $H$.

For any $1\leq j\leq p$, denote by $\mathcal
X_j$ a representation of $Y$ with character $\xi_j$. Then
$$R_{\underline c}=\prod_{\underline j} (\mathcal
X_{j_1}\otimes\cdots\otimes \mathcal X_{j_k})^{c_{\underline j}}$$
is a representation of $M$ with character $\xi_{\underline c}$. 
For any positive integer
$l$,
$$t_{\underline c}^{2l}=h^lh'^l\quad\text{and}\quad t_{\underline
c}^{2l+1}=h^{l+1}h'^lw_{\underline c}.$$
Then $t_{\underline c}$ has order $2(p-1)$ and if $g\in\langle
t_{\underline c}\rangle$, then $g_h$ (see the notation of
Lemma~\ref{lemma:extensionsemidirect} with $t=h$) has possibly
{non zero} values
only on the components of $Y^{kw}$ labeled by
$(u_r,1)$ and $(u^*_r,1)$.

However, for any $x\in Y$, we have ${}^x\mathcal X_{u_r,
1}=\mathcal X_{u_r,1}$ and  ${}^x\mathcal X_{u^*_r,
1}=\mathcal X_{u^*_r,1}$ because these two representations have
dimension $1$. Hence, if we denote by $\rho_{\underline
c}$ the restriction of $R_{\underline c}$ to $H$, then
${}^{g_h}\rho_{\underline c}=\rho_{\underline c}$ for all $g\in\langle
t_{\underline c}\rangle$.  
Thus, by Lemma~\ref{lemma:extensionsemidirect}, we can extend $\rho_{\underline
c}$ to $Q=H\rtimes\langle t_{\underline c}\rangle$, and the character
$\widetilde{\eta}_{\underline c}$ of this extension takes the same
values as $E(\eta_{\underline c})$. Moreover, by Gallagher's theorem,
every extension of $\eta_{\underline c}$ to $Q$ is of the form
$\widetilde{\eta}_{\underline c}\otimes \beta$, where $\beta$ is an
irreducible character of $\langle t_{\underline c}\rangle$. The
irreducible characters of $\Irr(\langle t_{\underline c}\rangle)$ are
$\beta_j:\langle t_{\underline c}\rangle\rightarrow \C^*$ for $0\leq
j\leq 2p-3$ defined by $\beta_j(t_{\underline
c}^l)=\omega_{2(p-1)}^{jl}$.
Since
$\Res_Q^{T_{\underline c}}(\widetilde{\theta}_{\underline c})$ is such
an extension, there is $0\leq s\leq 2p-3$ such
that
\begin{equation}
\label{eq:resQ}
\Res_Q^{T_{\underline c}}(\widetilde{\theta}_{\underline
c})=\widetilde{\rho}_{\underline c}\otimes \beta_s.
\end{equation}
We notice that $\widetilde{\rho}_{\underline c}(t_{\underline
c}^l)$ is equal to $E(\eta_{\underline c})(1)$ if $l$ is even, and to
$E(\eta_{\underline c})(t_{\underline c})$ if $l$ is odd. In either
case,~(\ref{eq:jamesextension}) implies that theses values are positive
integers. 

Recall that $t_{\underline c}^2=hh'$ is the element whose every
component is trivial except {those} labeled $(u_r,1)$ and $(u_r^*,1)$ taking the
value $b$.
By~(\ref{eq:jamesextension}), we have
\begin{equation}
\label{eq:valcarre}
\widetilde{\theta}_{\underline c}(t_{\underline
c}^2)={\theta}_{\underline c}(t_{\underline
c}^2)=-\omega_{p-1}^r(b)^2\theta_{\underline
c}(1)=-\omega_{p-1}^{2r}\eta_{\underline c}(1).
\end{equation}
Using~(\ref{eq:resQ}), we also have
$$\widetilde{\theta}_{\underline c}(t_{\underline
c}^2)=\Res_Q^{T_{\underline c}}(\widetilde\theta_{\underline
c})(t_{\underline c}^2)=\widetilde{\rho}_{\underline c}(t_{\underline
c}^2)\beta_s(t_{\underline c}^2)=\omega_{2(p-1)}^{2s}\eta_{\underline
c}(1)=\omega_{p-1}^{s}\eta_{\underline
c}(1).$$
Comparing with~(\ref{eq:valcarre}), we obtain
{$$\omega_{p-1}^{s-2r}=-1.$$} However, $-1\in \langle
\omega_{p-1}^2\rangle=\langle \omega_{(p-1)/2}\rangle$ if and only if
$(p-1)/2$ is even. Hence, if $p\equiv 1\mod 4$, $s$ has to be even, and if
$p\equiv 3\mod 4$, $s$ has to be odd.

On the other hand,~(\ref{eq:resQ}) gives
\begin{equation}
\label{eq:valext}
\widetilde{\theta}_{\underline c}(t_{\underline c})=E(\eta_{\underline
c})(t_{\underline c})\beta_s(t_{\underline c})=E(\eta_{\underline
c})(t_{\underline c})\omega_{2(p-1)}^s.
\end{equation}
Since $E(\eta_{\underline
c})(t_{\underline c})$ is fixed by any $f\in\mathcal H_{n!/2}$ because
$\eta_{\underline c}$ is, and $\omega_{2(p-1)}^s$ is
fixed by any $f\in \mathcal K_{n!/2}$, we deduce
from~(\ref{eq:criterebouge}) that the characters $\psi_{\lambda}^{\pm}$ are fixed by
$f\in\mathcal K_{n!/2}$ for all $\underline\lambda\in\mathcal
P_{\underline c}$.

Finally, we remark that $\omega_{2(p-1)}^{p-1}=\omega_2=-1$. Thus,
$\omega_{2(p-1)}^p=-\omega_{2(p-1)}$, and
$\omega_{2(p-1)}^{2p}=\omega_{2(p-1)}^2$. Then by~(\ref{eq:valext}), 
$\sigma_{n!/2}$ fixes $\widetilde{\theta_{\underline c}}(t_{\underline
c})$ if $s$ is even, that is when $p\equiv 1\mod 4$ and $$\sigma_{n!/2}(\widetilde{\theta_{\underline c}}(t_{\underline
c}))=-\widetilde{\theta_{\underline c}}(t_{\underline
c})$$
if $s$ is odd, that is $p\equiv 3\mod 4$. The result follows from {the
criterion}~(\ref{eq:criterebouge}).
\end{proof}

Since $\sqrt p$ is a root of the polynomial $x^2-p\in\Q[x]$, we have $f(\sqrt p)=\pm\sqrt p$ for $f\in\mathcal K_{n!}$. Denote by
$\epsilon_f\in\{-1,1\}$ the sign such that $f(\sqrt p)=\epsilon_f \sqrt
p$.
\begin{proposition}
\label{prop:signesingular}
Let $\underline\lambda\in\mathcal{SP}(p^k,w)$ be such that
$\lambda_{\underline j}=\emptyset$ for all $\underline j\neq \underline
p^*(k)$. If $f\in\mathcal K_{n!}$, then
$$\varepsilon(\psi_{\underline\lambda},f)=\epsilon_f^{kd}\cdot
\varepsilon(\chi_{\lambda_{\underline p^*(k)}},f),$$
where $\varepsilon(\chi_{\lambda_{\underline p^*(k)}},f)$ is defined
in~(\ref{eq:signebouge}), and $d$ is the number of diagonal hooks {in
the Young diagram }of $\lambda_{\underline p^*(k)}$.
Moreover,
$$\varepsilon(\psi_{\underline\lambda},\sigma_{n!/2})=(-1)^{dk(p-1)/2}\cdot
\varepsilon(\chi_{\lambda_{\underline p^*(k)}},\sigma_{n!/2}).$$
\end{proposition}

\begin{proof}
As in the proof of the Proposition~\ref{prop:signesingular}, we consider
the group  $H=((Y^k)^+)^w$.
Write $\xi=\xi_{\underline
p^*(k)}^w\in\Irr(M)$. This is the unique split character of $M$
by~(\ref{eq:fixdirect}) and \S\ref{subsec:irrYk}. Denote by $\xi^+$ the
constituent of $\Res_{M^+}^M(\xi)$ such that $(\xi_{\underline
p^*(k)}^+)^w\in\Irr(H)$ is a constituent of $\Res^{M^+}_H(\xi^+)$.
First, we remark that the subgroup $U=M^+\rtimes\Alt_w$ is a normal
subgroup of $N^+$ because it has index $2$. Moreover, the inertial
subgroup in $U$ of $\xi^+$ and $\xi^-$ is $U$. Let $s\in N^+\backslash
U$. Then $s=(h;\tau)$ with $h\in M\backslash M^+$ and $\tau\in
\sym_w\backslash \Alt_w$, and ${}^s\xi^+=\xi^{-}$. It follows that 
$${}^s\Ind_{M^+}^U(\xi^+)=\Ind_{M^+}^U(\xi^-),$$
because $M^+\unlhd N^+$ and $U\unlhd N^+$. Furthermore,
$\Ind_{M^+}^U(\xi^+)$ and $\Ind_{M^+}^U(\xi^-)$ have no constituents in
common by Clifford theory with respect to $M^+\unlhd U$. It follows that
if $\chi$ is a constituent of $\Ind_{M^+}^U(\xi^+)$, then
${}^s\chi\neq\chi$. Hence, $\Ind_U^{N^+}(\chi)$ is irreducible by
Clifford theory with respect to $U\unlhd N^+$. By the transitivity
of induction and Mackey's formula, 
$$\Res_{N^+}^N\Ind_{M}^{N}(\xi)=\Res_{N^+}^N\Ind_{M}^{N}\Ind_{M^+}^M(\xi^+)=\Res_{N^+}^N\Ind_{M^+}^{N}(\xi^+)=\Ind_{M^+}^{N^+}(\xi^+).$$
Hence, $\psi_{\underline\lambda}^+$ and
$\psi_{\underline\lambda}^-$ restrict to $U$ into two irreducible components. We
write $\psi_{\underline\lambda,\pm}^{\pm}$ for the constituent of
$\Res^{N^+}_U(\psi_{\underline \lambda}^{\pm})$ which belongs to
$\Ind_{M^+}^U(\xi^{\pm})$.

Now we show how to extend $\xi^+$ and $\xi^-$ to $U$. Consider
the wreath product $V=H\rtimes \Alt_w$. Denote by $\nu^+=(\xi_{\underline
p^*(k)}^+)^w\in\Irr(H)$ and $\nu^-={}^s\nu^+$. By Clifford theory with respect to $H\unlhd M^+$
and the previous choice of labeling, we have
$$\Ind_H^{M^+}(\nu^+)=\xi^+\quad\text{and}\quad\Ind_H^{M^+}(\nu^-)=\xi^-.$$
Write $E(\nu^+)$ for the James-Kerber extension of $\nu^+$ to $V$.
Therefore, Mackey's formula gives
$$\Res_{M^+}^U\Ind_{V}^U(E(\nu^+))=\Ind_H^{M^+}(\nu^+)=\xi^+.$$
Thus, $\mathcal V^+=\Ind_V^U(E(\nu^+))$ is an extension of $\xi^+$ to $U$.
By Gallagher's theorem~\cite[Corollary 6.17]{isaacs}, 
the constituents of $\Ind_{M^+}^U(\xi^+)$ are of
the form $\zeta_{\mu,+}=\mathcal V^+\otimes \chi_{\mu}$ if $\mu\neq\mu^*$
and $\zeta_{\mu,+}^{\pm}=\mathcal V^+\otimes \chi_{\mu}^{\pm}$ if
$\mu=\mu^*$. Here, $\chi_{\mu}$ and $\chi_{\mu}^{\pm}$ are the
irreducible characters of $\Alt_w$.  If we set $\mathcal V^-={}^s\mathcal V^+$, then $\mathcal
V^+\neq\mathcal V^-$ because it is a constituent of
$\Ind_{M^+}^U(\xi^-)$. Thus,
$${}^s\psi_{\underline\lambda,\pm}^{\pm}={}^s(\mathcal
V^{\pm}\otimes\chi_{\lambda_{\underline p^*(k)}}^{\pm})={}^h(\mathcal
V^{\pm})\otimes {}^\tau(\chi_{\lambda_{\underline p^*(k)}}^{\pm})
=\mathcal
V^{\mp}\otimes\chi_{\lambda_{\underline
p^*(k)}}^{\mp}=\psi_{\underline\lambda,\mp}^{\mp},$$
and
\begin{equation}
\label{eq:respsilambda}
\Res^{N^+}_U(\psi_{\underline\lambda}^+)=\psi_{\underline\lambda,+}^++\psi_{\underline\lambda,-}^-\quad\text{and}\quad
\Res^{N^+}_U(\psi_{\underline\lambda}^-)=\psi_{\underline\lambda,+}^-+\psi_{\underline\lambda,-}^+.
\end{equation}
Consider the element $g=(u,\pi)$ such that the cycle lengths of $\pi$
are the diagonal hook lengths of $\lambda_{\underline p^*(k)}$, and $u$
{is}
such that every cycle of $g$ has cyclic product equal to $\underline a$.
Then $g\in U$ and
$$\mathcal V^+(g)=\Ind_V^U(E(\nu^+))(g)=
\sum\limits_{\renewcommand{\arraystretch}{0.5}%
     \begin{array}{c}
        \scriptstyle t\in [U/V] \\
        \scriptstyle {}^tg\in V
     \end{array}}
E(\nu^+)({}^tg)=
\sum\limits_{\renewcommand{\arraystretch}{0.5}%
     \begin{array}{c}
        \scriptstyle t\in [U/V] \\
        \scriptstyle {}^tg\in V
     \end{array}}
\prod_{\gamma \in C(\pi)}\nu^+(\mathfrak c({}^tg,\gamma)).$$
However, $U/V\simeq M^+/H$. Hence, we can take for transversal of $U$
mod $V$
the set
$$[U/V]=\{t_{\underline\alpha}=(b^{\alpha_1},\ldots,b^{\alpha_w})\,|\,\underline\alpha\in\{0,1\}^w,\,\alpha_1+\cdots+\alpha_w\equiv
0\pmod 2\}.$$  
Moreover, $^{t_{\alpha}}g\in U$ if and only if
$b^{\alpha_j}u_jb^{-\alpha_{\pi^{-1}(j)}}\in (Y^k)^+$ for all $1\leq
j\leq w$, if and only if
$b^{\alpha_j}b^{-\alpha_{\pi^{-1}(j)}}\in (Y^k)^+$ (because ${}^{b^{\alpha_j}}u_j\in
(Y^k)^+$) if and only if $b^{\alpha_j-\alpha_{\pi^{-1}(j)}}\in
(Y^k)^+$, {\emph{i.e.}} $\alpha_j=\alpha_{\pi^{-1}(j)}$, that is
{all} $\alpha_j$ are
equal on the cycles of $\pi$. Denote by $T$ the set of elements of
$[U/V]$ that satisfy this property. By~\cite[4.2.6]{James-Kerber}, for
any $\gamma\in C(\pi)$ and $t_{\underline\alpha}\in T$, 
$$\mathfrak
c({}^{t_{\underline\alpha}}g,\gamma)={}^{b^{\alpha_{\gamma}}}\mathfrak
c(g,\gamma)={}^{b^{\alpha_{\gamma}}}\underline a.$$
Thus
{$$\mathcal V^+(g)=\sum_{t_{\underline\alpha\in T}}\prod_{\gamma\in
C(\pi)}\nu^+({}^{b^{\alpha_{\gamma}}}\underline a).$$}
Let $\gamma_0$ be the cycle of $C(\pi)$ whose support contains $1$, and
define $y\in M$ such that $y_i=b$ if $i\in\operatorname{supp}(\gamma_0)$ 
and $1$ otherwise. Since $|\gamma_0|$ is odd, $y\in M\backslash M^+$.
Using that $\mathcal V^-(g)=\mathcal V^+({}^yg)$,
the same computation as above shows that
{$$\mathcal V^-(g)=\sum_{t_{\underline\alpha}\in \overline T}\prod_{\gamma\in
C(\pi)}\nu^+({}^{b^{\alpha_{\gamma}}}\underline a),$$}
where $\overline T$ is the set of $t_{\underline\alpha}$ such that the $\alpha_i$
are constant on the cycle of $\pi$ and $\alpha_1+\cdots+\alpha_w\equiv
1\pmod 2$. 
Since the lengths of the cycles of $\pi$ are odd, we have
$$\sum_{j=1}^w\alpha_j\equiv\sum_{\gamma\in
C(\pi)}\alpha_{\gamma}\pmod 2$$
for every $t_{\underline\alpha}\in T\cup \overline T$. Therefore, by a 
computation similar to that proving~(\ref{eq:diffcarYk}), we obtain
\begin{equation}
\label{eq:valextV}
(\mathcal V^+-\mathcal V^-)(g)=i^{dk(p-1)/2}\sqrt{p^{dk}}, 
\end{equation}
where $d=|C(\pi)|$. By~{\cite[2.5.13]{James-Kerber}}, we also have
\begin{equation}
\label{eq:splitpi}
(\chi_{\lambda_{\underline p^*(k)}}^+-\chi_{\lambda_{\underline
p^*(k)}}^-)(\pi)\neq 0.
\end{equation}

Let $f\in\mathcal G_{n!/2}$. By~(\ref{eq:respsilambda}), if
${}^f\psi_{\underline\lambda,\pm}^{\pm}=\psi_{\underline\lambda,\pm}^{\pm}$
or 
${}^f\psi_{\underline\lambda,\pm}^{\pm}=\psi_{\underline\lambda,\mp}^{\mp}$,
then $\varepsilon(\psi_{\underline\lambda,f})=1$, and if 
${}^f\psi_{\underline\lambda,\pm}^{\pm}=\psi_{\underline\lambda,\pm}^{\mp}$,
then $\varepsilon(\psi_{\underline\lambda,f})=-1$. However,
\begin{equation}
\label{eq:valpsig}
\psi_{\underline\lambda,\pm}^{\pm}(g)=\mathcal
V^{\pm}(g)\chi_{\lambda_{\underline p^*(k)}}^{\pm}(\pi)
\end{equation}
is non-zero, and $f(\psi_{\underline\lambda,\pm}^{\pm}(g))=f(\mathcal
V^{\pm}(g))\cdot f(\chi_{\lambda_{\underline p^*(k)}}^{\pm})$. Thus, by
equalities~(\ref{eq:valextV}),(\ref{eq:splitpi}) and~(\ref{eq:valpsig}), 
we have
${}^f\psi_{\underline\lambda,\pm}^{\pm}=\psi_{\underline\lambda,\pm}^{\mp}$
if and only if $f(\mathcal V^{\pm}(g))=\mathcal V^{\pm}(g)$ and
$f(\chi_{\lambda_{\underline p^*(k)}}^{\pm}(\pi))=
\chi_{\lambda_{\underline p^*(k)}}^{\mp}(\pi)$ or
$f(\mathcal V^{\pm}(g))=\mathcal V^{\mp}(g)$ and
$f(\chi_{\lambda_{\underline p^*(k)}}^{\pm}(\pi))=
\chi_{\lambda_{\underline p^*(k)}}^{\pm}(\pi)$.

Now, if $f\in\mathcal K_{n!/2}$, then $f(i)=i$. Note also that
$\sigma_{n!/2}(i)=(-1)^{(p-1)/2}i$ and that $\sigma_{n!/2}$ fixes $\sqrt
p$. The result then follows
from~(\ref{eq:valextV}).
\end{proof}

Let $\underline\lambda\in\mathcal{SP}(p^k,w)$. Let
$w'=|\lambda_{\underline p^*(k)}|$ and $w''=w-w'$. Define
$\underline\lambda''\in\mathcal{SP}(p^k,w')$ such that each part is
empty except $\lambda''_{\underline p^*(k)}=\lambda_{\underline p^*(k)}$, and
$\underline\lambda'\in\mathcal{SP}(p^k,w'')$  such that
$\lambda'_{\underline j}=\lambda_{\underline j}$ when $\underline
p\neq \underline p^*(k)$ and $\lambda_{\underline p^*(k)}=\emptyset$. Denote
by $\psi_{\underline\lambda'}$ and $\psi_{\underline\lambda''}$ the
corresponding irreducible characters of $N_{k,w'}$ and $N_{k,w''}$,
respectively.

\begin{theorem}
\label{thm:irrationaliteNk}
Let $\underline\lambda\in\mathcal{SP}(p^k,w)$. Then for any $f\in\mathcal
G_{n!/2}$,
$$\varepsilon(\psi_{\underline\lambda},f)=\varepsilon(\psi_{\underline\lambda'},f)\cdot\varepsilon(\psi_{\underline\lambda''},f).$$ 
\end{theorem}

\begin{proof}
Let $\underline\lambda\in\mathcal{SP}(p^k,w)$. Assume
$\underline\lambda'\neq\emptyset$ and $\underline\lambda''\neq\emptyset$.
Set $\underline c=(|\lambda_{\underline j}|,\ \underline j\in I^k)$,
$\underline c'=(0,\ldots,0,c_{\underline p^*(k)},0,\ldots,0)$ and
$\underline c''$ such that the coordinates of $\underline c$ and
$\underline c''$ are the same, except $c_{\underline p^*(k)}''=0$. Since
$\underline\lambda''\neq\emptyset$, one has $\xi_{\underline
c}^*\neq\xi_{\underline c}$, and the restriction
$\vartheta_{\underline c}$ of $\xi_{\underline c}$ to $M^+$ is
irreducible. By Mackey's formula, 
$$\Res_{N^+}^N\Ind_{M}^N(\xi_{\underline
c})=\Ind_{M^+}^{N^+}(\vartheta_{\underline c}).$$
Thus, $\psi_{\underline \lambda}^+$ and $\psi_{\underline\lambda}^-$
appear in the Clifford theory {attached}
to $\vartheta_{\underline c}$ with respect to $M^+\unlhd N^+$.
Moreover, by an argument similar to the one in the proof of
Proposition~\ref{prop:signeregular}, the inertial group of
$\vartheta_{\underline c}$ is an extension of degree $2$ of
$N_{\underline c}^+$.
Let $t_{\underline c''}$ be an element of $N_{\underline
c''}^+$ as in the proof of Proposition~\ref{prop:signeregular}, and 
$H_{\underline c''}=\langle N_{\underline c''},t_{\underline
c''}\rangle$. Consider
\begin{equation}
\label{eq:Hcdirectprod}
H_{\underline c}=N_{w'}\times H_{\underline c''}.
\end{equation}
Then the elements of $H_{\underline c}^+=
\langle (N_{w'}\times N_{\underline
c''})^+,t_{\underline c''}\rangle$ fix $\vartheta_{\underline c}$ and
this group is
an extension of degree $2$ of $(N_{w'}\times N_{\underline
c''})^+=N_{\underline c}$. Thus, the inertial subgroup of
$\vartheta_{\underline c}$ is $H_{\underline c}^+$.

On the other hand, $E(\xi_{\underline c}\chi_{\underline\lambda})$ is
not $H_{\underline c}$-stable. Hence,
$\widetilde{\psi}_{\underline\lambda}=\Ind_{N_{\underline c}}^{
H_{\underline c}}(E(\xi_{\underline c})\chi_{\underline\lambda})$ is
irreducible and by Mackey's formula
$$\Res_{H_{\underline c}^+}^{H_{\underline
c}}(\widetilde{\psi}_{\underline\lambda})=\Ind_{N_{\underline c}^+}^{H_{\underline
c}^+}(\theta_{\underline\lambda})=\theta_{\underline\lambda}^++\theta_{\underline\lambda}^-,$$
where $\theta_{\underline\lambda}$ is the restriction of
$E(\xi_{\underline c})\chi_{\underline\lambda}$ to $N_{\underline c}^+$.
Again, by Mackey's formula,
\begin{align}
\label{eq:argument}
\psi_{\underline\lambda}^++\psi_{\underline\lambda}^-&=\Res_{N^+}^N(\psi_{\underline\lambda})\\
&=\Res_{N^+}^N\Ind_{N_{\underline c}}^{N}(E(\xi_{\underline
c})\chi_{\underline\lambda})\nonumber\\
&=\Ind_{N_{\underline
c}^+}^{N^+}(\theta_{\underline\lambda})\nonumber\\
&= \Ind_{N_{\underline
c}^+}^{N^+}(\theta_{\underline\lambda}^+) + \Ind_{N_{\underline
c}^+}^{N^+}(\theta_{\underline\lambda}^-)\nonumber.
\end{align}
In particular, we can choose the label such that
$\psi_{\underline\lambda}^\eta=\Ind_{N_{\underline
c}^+}^{N^+}(\theta_{\underline\lambda}^\eta)$ for $\eta\in\{-1,1\}$.

Let $f\in\mathcal H_{n!/2}$. By Lemma~\ref{lemma:induction}, one has
\begin{equation}
\label{eq:locproof1}
\varepsilon(\psi_{\underline\lambda},f)=\varepsilon(\widetilde{\psi}_{\underline\lambda},f).
\end{equation}
Note that $E(\xi_{\underline
c})\chi_{\underline\lambda}=E(\xi_{\underline
c'})\chi_{\underline\lambda'}\otimes
E(\xi_{\underline
c''})\chi_{\underline\lambda''}$, hence
\begin{equation}
\label{eq:locproof3}
\widetilde\psi_{\underline\lambda}=\psi_{\underline\lambda'}\otimes
\widetilde{\psi}_{\underline\lambda''},
\end{equation}
where 
$\widetilde{\psi}_{\underline\lambda''}=\Ind_{N_{\underline
c''}}^{H_{\underline c''}}(E(\xi_{\underline
c''})\chi_{\underline\lambda''})\in\Irr(H_{\underline c''})$. We remark
that the computations~(\ref{eq:argument}) and~(\ref{eq:locproof1}) 
{applied} to $N_{k,w''}$ give
\begin{equation}
\label{eq:locproof4}
\varepsilon({\psi}_{\underline\lambda''},f)=\varepsilon(\widetilde{\psi}_{\underline\lambda''},f).
\end{equation}

Now,
$E(\xi_{\underline c})\chi_{\underline\lambda}$ is
$f$-stable, thus $\widetilde\psi_{\underline\lambda}$ also is by
Lemma~\ref{lemma:induction}. Applying Proposition~\ref{prop:reducbouge}
with respect to the direct product~(\ref{eq:Hcdirectprod}), and
using~(\ref{eq:locproof3}) and~(\ref{eq:locproof4}) we obtain
that
{
\begin{equation}
\label{eq:locproof2}
\varepsilon(\widetilde{\psi}_{\underline\lambda},f)=\varepsilon(\psi_{\underline\lambda'},f)\cdot\varepsilon(\widetilde{\psi}_{\underline\lambda''},f)=\varepsilon(\psi_{\lambda'},f)\cdot\varepsilon(\psi_{\lambda''},f). 
\end{equation}}
The result follows from~(\ref{eq:locproof1}) and~(\ref{eq:locproof2}).
\end{proof}
\section{Alternating groups: The global case}
\label{sec:global}
%
Let $\lambda=\lambda^*$.
Denote by ${\mathcal C}_{\lambda}$ the conjugacy classes of $\sym_n$ of
type ${\mathfrak{D}(\lambda)}$, that is, the lengths of the elements of
${\mathfrak{D}(\lambda)}$ are the cycle lengths of any element $x\in
{\mathcal C}_{\lambda}$. Recall that the classes
${\mathcal C}_{\lambda}$ of $\sym_n$ {split} into two classes ${\mathcal
C}_{\lambda}^+$ and ${\mathcal C}_{\lambda}^-$ of $\Alt_n$, and that the 
restriction to $\Alt_n$ of the
irreducible character $\chi_{\lambda}$ splits into two constituents
$\chi_{\lambda}^+$ and $\chi_{\lambda}^-$ that take the same (integer) value on
{every class} except on ${\mathcal C}_{\lambda}^{\pm}$, and
{by~\cite[2.5.13]{James-Kerber}} the labeling
can be chosen such that for all $\eta,\,\nu\in\{-1,1\}$
\begin{equation}
\label{eq:frobenius}
\chi_{\lambda}^{\eta}(x_{\lambda}^{\nu})=\frac{1}2\left((-1)^{(n-d_{\lambda})/2}+\eta\nu i^{(n-d_{\lambda})/2}\sqrt{\prod_{h\in
{\mathfrak{D}(\lambda)}}h}\right),
\end{equation}
where $x_{\lambda}^\nu$ is a representative of ${\mathcal
C}_{\lambda}^\nu$ and $d_{\lambda}=|{\mathfrak{D}(\lambda)}|$.

{For any field automorphism $f$, if $\alpha$ is a
root of $x^2-q\in \Q[x]$, then $f(\alpha)$ is also a root of $x^2-q$.} We
denote by $\varepsilon(\alpha,f)\in\{-1,1\}$ the sign such that
\begin{equation}
\label{eq:notationbouge}
f(\alpha)=\varepsilon(\alpha,f)\alpha.
\end{equation}
Note that when $\lambda=\lambda^*$ and $f\in\mathcal
H_{n!/2}$,
\begin{equation}
\label{eq:bougealt}
\varepsilon(\chi_{\lambda},f)=\varepsilon\left(i^{(n-d_{\lambda})/2}
\sqrt{\prod_{h\in
{\mathfrak{D}(\lambda)}}h},f\right). 
\end{equation}
\medskip
\subsection{Action of Galois automorphisms on square roots}

Let $m$ be an odd number. For any integer $r$, we write
$\left(\frac{r}m\right)$ for the Jacobi symbol.

\begin{proposition}
\label{prop:racbouge}
Let $m$ be an odd number, and $f$ be a Galois automorphism. Denote by
$r$ an integer prime to $m$ such that $f(\omega_m)=\omega_m^r$. 
Then
$$f(\sqrt m)=\varepsilon(i,f)^{\frac{m-1}{2}}\left(\frac{r}m\right)\sqrt m.$$
\end{proposition}

\begin{proof}
Write $m=p_1^{a_1}\cdots p_s^{a_s}$ for the prime {factorisation}
of $m$. Define by $E$ and $F$ respectively the set of indices $1\leq
j\leq s$ such that $p_j\equiv 1$ or $3$ modulo $4$. 
 
Suppose $m\equiv 1\mod 4$. Then $\sum_{j\in F}a_j$ is even, and
\begin{equation}
\label{eq:racm}
\sqrt m=\prod_{j\in E}\sqrt{p_j}^{a_j}\cdot\left(\eta\prod_{j\in
F}(i\sqrt{p_j})^{a_j}\right),
\end{equation}
where $\eta=-1$ if $\sum_{j\in F}a_j\equiv 2\mod 4$ and $\eta=1$
otherwise.
Since $f$ is a field automorphism fixing $\eta$, we deduce
\begin{equation}
\label{eq:signeracm}
\varepsilon(\sqrt m,f)=\prod_{j\in E}\varepsilon(\sqrt
p_j,f)^{a_j}\cdot \prod_{j\in F}\varepsilon(i\sqrt p_j,f)^{a_j}.
\end{equation}
Now, if we set $q_j=\sqrt{p_j}$ if $j\in E$ and $q_j=i\sqrt{p_j}$ if $j\in
F$, {then~\cite[Thm. 1]{KennethRosenNrTh} gives}
$$\sum_{t=1}^{p_{j}-1}\left(\frac{t}p_j\right)\omega_{p_j}^t= q_j$$ 
Furthermore, one has $\omega_{p_j}=\omega_m^{m/p_j}$, so
$f(\omega_{p_j})=\omega_{p_j}^r$, and
$$f(q_j)=\sum_{t=1}^{p_{j}-1}\left(\frac{t}p_j\right)\omega_{p_j}^{rt}=
\left(\frac{r}p_j\right)q_j$$
by~\cite[Prop. 6.3.1]{KennethRosenNrTh}. Hence, $\varepsilon(q_j,f)=
\left(\frac{r}p_j\right)$ and the result follows
from~(\ref{eq:signeracm}) and the definition of the Jacobi symbol.

Suppose that $m\equiv 3\mod 4$. Then $\sum_{j\in F}a_j$ is odd, and
in the {formula}~(\ref{eq:racm}), $\eta$ is now equal to $i$ up to a sign.
When the formula~(\ref{eq:signeracm}) is multiplied by
$\varepsilon(i,f)$, the result follows.
\end{proof}
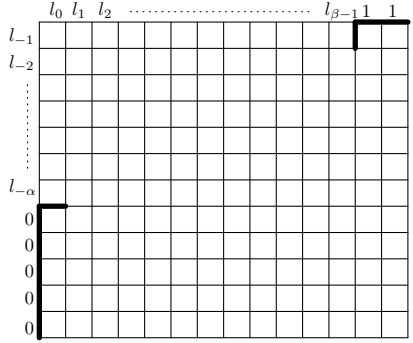
\begin{figure}
\label{fig:graphpart}
\begin{center}
\begin{tikzpicture}[line cap=round,line join=round,>=triangle
45,x=1.0cm,y=1.0cm,scale=0.7,every node/.style={scale=0.7}]
\clip(-1.6614876033057846,-5.464462809917355) rectangle (7.87570247933884,1.758677685950413);
\draw (0.5,1.)-- (0.5,-5.);
\draw (1.,1.)-- (1.,-5.);
\draw (1.5,1.)-- (1.5,-5.);
\draw (2.,1.)-- (2.,-5.);
\draw (2.5,1.)-- (2.5,-5.);
\draw (3.,-5.)-- (3.,1.);
\draw (3.5,1.)-- (3.5,-5.);
\draw (4.,-5.)-- (4.,1.);
\draw (4.5,1.)-- (4.5,-5.);
\draw (5.,-5.)-- (5.,1.);
\draw (5.5,1.)-- (5.5,-5.);
\draw (6.,-5.)-- (6.,1.);
\draw (6.5,1.)-- (6.5,-5.);
\draw (7.,-5.)-- (7.,1.);
\draw (0.,1.)-- (7.,1.);
\draw (0.,0.5)-- (7.,0.5);
\draw (7.,0.)-- (0.,0.);
\draw (0.,-0.5)-- (7.,-0.5);
\draw (7.,-1.)-- (0.,-1.);
\draw (0.,-1.5)-- (7.,-1.5);
\draw (0.,-2.)-- (7.,-2.);
\draw (0.,-2.5)-- (7.,-2.5);
\draw (0.,-3.)-- (7.,-3.);
\draw (0.,-4.)-- (7.,-4.);
\draw (7.,-3.5)-- (0.,-3.5);
\draw (0.,-4.5)-- (7.,-4.5);
\draw (7.,-5.)-- (0.,-5.);
\draw (0.,-5.)-- (0.,1.);
\draw (0.07,1.5) node[anchor=north west]
{$l_0$};
\draw (0.5,1.5) node[anchor=north west]
{$l_1$};
\draw (1,1.5) node[anchor=north west]
{$l_2$};
\draw [dotted] (1.7269421487603298,1.196694214876033)-- (5.214545454545453,1.196694214876033);
\draw (5.3,1.5) node[anchor=north west] {$l_{\beta-1}$};
\draw (6,1.45) node[anchor=north west] {$1$};
\draw (6.5,1.45) node[anchor=north west] {$1$};
\draw (-0.7,-1.861157024793388) node[anchor=north west] {$l_{-\alpha}$};
\draw (-0.4,-4.571900826446281) node[anchor=north west]
{$0$};
\draw (-0.4,-4) node[anchor=north west]
{$0$};
\draw (-0.4,-3.5) node[anchor=north west]
{$0$};
\draw (-0.4,-3.) node[anchor=north west]
{$0$};
\draw (-0.4,-2.5) node[anchor=north west]
{$0$};
\draw [line width=2.pt] (6.,1.)-- (6.,0.5);
\draw [line width=2.pt] (0.,-2.5)-- (0.5,-2.5);
\draw [dotted] (-0.22347107438016536,-0.19173553719008263)-- (-0.22347107438016536,-1.7785123966942145);
\draw (-0.7,1) node[anchor=north west]
{$l_{-1}$};
\draw (-0.7,0.5) node[anchor=north west]
{$l_{-2}$};
\draw [line width=2.pt] (6.,1.)-- (7.,1.);
\draw [line width=2.pt] (0.,-5.)-- (0.,-2.5);
\end{tikzpicture}
\end{center}
\caption{Construction of the rim from the sequence}
\end{figure}
\subsection{Combinatorics {of symmetric partitions}}
\label{subsec:combinatorics}
Recall a partition $\lambda$ is completely determined by the {\it rim} of
its Young diagram $Y(\lambda)$, a path constituted  of vertical and
horizontal dashes of length one. Then $\lambda$ can, by the association of
0 (resp. 1) to a vertical (resp. horizontal) dash of length one, be
expressed by its \emph{partition sequence} $\Lambda$. This is an infinite
sequence taking its values in $\{0,1\}$ and of the form $\overline 0
\cdots \overline 1$, where $\overline 0$ and $\overline 1$ mean an
infinite sequence of left-trailing and right-trailing $0$s and of $1$s,
respectively. 
We refer the reader to Example~\ref{ex:ex1}.

Let $\Lambda$ be the partition sequence associated to $\lambda$.  
Denote by
$\alpha$ and $\beta$ the numbers of zeroes and ones 
between the leftmost $1$
and the rightmost $0$ coming after it 
when we read the sequence from the left-to-right.
Then there are $\alpha+\beta$ elements in the sequence between
$\overline 0$ and $\overline 1$. We write 
\begin{equation}
\label{eq:partseq}
\Lambda=\overline
0l_{-\alpha}l_{-\alpha+1}\cdots l_{-1}l_0\cdots l_{\beta-1}\overline
1=(l_u)_{u\in\Z}.
\end{equation}
In particular $l_{-\alpha}=1$ and $l_{\beta}=0$. If there is no $0$
after the first $1$, then $\alpha=\beta=0$ and the sequence is
$\overline 0\,\overline 1$ and corresponds to the empty partition.
The bijection between this labeling of partition sequences and
partitions can be represented graphically as in
Figure~1.

\begin{example}
Consider the partition $\lambda=(7^2,5,4,3,2^2)$. 

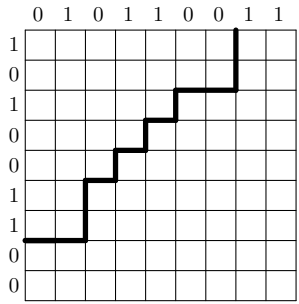
\begin{center}
\begin{tikzpicture}[line cap=round,line join=round,>=triangle
45,x=1.0cm,y=1.0cm,scale=0.8,every node/.style={scale=0.8}]
\clip(-1.6614876033057846,-4.) rectangle (5.87570247933884,1.758677685950413);
\draw (0.5,1.)-- (0.5,-3.5);
\draw (1.,1.)-- (1.,-3.5);
\draw (1.5,1.)-- (1.5,-3.5);
\draw (2.,1.)-- (2.,-3.5);
\draw (2.5,1.)-- (2.5,-3.5);
\draw (3.,-3.5)-- (3.,1.);
\draw (3.5,1.)-- (3.5,-3.5);
\draw (4.,-3.5)-- (4.,1.);
\draw (4.5,1.)-- (4.5,-3.5);
\draw (0.,1.)-- (4.5,1.);
\draw (0.,0.5)-- (4.5,0.5);
\draw (4.5,0.)-- (0.,0.);
\draw (0.,-0.5)-- (4.5,-0.5);
\draw (4.5,-1.)-- (0.,-1.);
\draw (0.,-1.5)-- (4.5,-1.5);
\draw (0.,-2.)-- (4.5,-2.);
\draw (0.,-2.5)-- (4.5,-2.5);
\draw (0.,-3.)-- (4.5,-3.);
\draw (4.5,-3.5)-- (0.,-3.5);
\draw (0.,-3.5)-- (0.,1.);

\draw (-0.4,-3.) node[anchor=north west]
{$0$};
\draw (-0.4,-2.5) node[anchor=north west]
{$0$};
\draw (-0.4,-2.) node[anchor=north west]
{$1$};
\draw (-0.4,-1.5) node[anchor=north west]
{$1$};
\draw (-0.4,-1.) node[anchor=north west]
{$0$};
\draw (-0.4,-0.5) node[anchor=north west]
{$0$};
\draw (-0.4,0) node[anchor=north west]
{$1$};
\draw (-0.4,0.5) node[anchor=north west]
{$0$};
\draw (-0.4,1) node[anchor=north west]
{$1$};
\draw (0,1.5) node[anchor=north west]
{$0$};
\draw (0.5,1.5) node[anchor=north west]
{$1$};
\draw (1,1.5) node[anchor=north west]
{$0$};
\draw (1.5,1.5) node[anchor=north west]
{$1$};
\draw (2,1.5) node[anchor=north west]
{$1$};
\draw (2.5,1.5) node[anchor=north west]
{$0$};
\draw (3,1.5) node[anchor=north west]
{$0$};
\draw (3.5,1.5) node[anchor=north west]
{$1$};
\draw (4,1.5) node[anchor=north west]
{$1$};

\draw [line width=2.pt] (0.,-2.5)-- (0.5,-2.5);
\draw [line width=2.pt] (3.5,1.)-- (3.5,0);
\draw [line width=2.pt] (3.5,0)-- (2.5,0);
\draw [line width=2.pt] (2.5,0)-- (2.5,-0.5);
\draw [line width=2.pt] (2.5,-0.5)-- (2,-0.5);
\draw [line width=2.pt] (2,-0.5)-- (2,-1);
\draw [line width=2.pt] (2,-1.)-- (1.5,-1);
\draw [line width=2.pt] (1.5,-1.)-- (1.5,-1.5);
\draw [line width=2.pt] (1.5,-1.5)-- (1,-1.5);
\draw [line width=2.pt] (1,-1.5)-- (1,-2.5);
\draw [line width=2.pt] (1.,-2.5)-- (0.,-2.5);
\end{tikzpicture}
\end{center}
The partition sequence of $\lambda$
is $\Lambda=\overline 011001010101100\overline 1$. We have
$\alpha=\beta=7$, and following the preceding convention, $l_0$ and
$l_{-1}$ are {the numbers} directly at the right and the left of the dash
$1100101|0101100$. Note that in the accompanying figure the partition sequence has been projected to the left-and-top border of the Young diagram.
\label{ex:ex1}
\end{example}

Furthermore, by~\cite[Lemma
2.2]{Olsson}, the partition sequence of $\lambda^*$, denoted by
$\Lambda^*$, is obtained from
$\Lambda$ by reading $\Lambda$ from the right to the left with 0s and
1s interchanged. In other words
\begin{equation}
\label{eq:sequencelambdaconj}
\Lambda^*= \overline
0(1-l_{\beta-1})(1-l_{\alpha-2})\cdots (1-l_{-\alpha})\overline
1=(1-l_{-u-1})_{u\in\Z}.
\end{equation}

We now describe ${\mathfrak D}(\lambda)$ the diagonal hooks of $\lambda$ using
$\Lambda$.
For $\delta\in\{0,1\}$, write
{$$H_{\delta}=\{0\leq j\leq \beta-1\mid l_j=\delta\}\quad\text{and}\quad 
K_{\delta}=\{-\alpha\leq j\leq -1\mid l_j=\delta\}.$$}
Note that if $h=|H_0|$, then $|H_1|=\beta-h$ and
$|K_1|=\beta-|H_1|=\beta-(\beta-h)=h$. Hence, $|H_0|=|K_1|$. 
On the other hand, by~\cite[p.\,9]{Olsson} each hook of $\lambda$
corresponds to a pair $(i,j)$ such that
$-\alpha\leq i<j\leq \beta-1$ with $l_i=1$ and $l_j=0$. Such a hook
$h_{(i,j)}$ has length $|j-i|$. In particular, {the longest hook of
$\lambda$ is $h_{-\alpha,\beta-1}$ and it has to be the first diagonal
hook of $\lambda$.} When we remove it from $\lambda$, we obtain a
new partition with the same sequence {as} $\lambda$ except that
$l_{-\alpha}=0$ and $l_{\beta-1}=1$. Since $|H_0|=|K_1|$, when we
iterate this process $|H_0|$ times, we obtain the empty
partition. In fact, we have removed from $\lambda$ all diagonal hooks
one by one. Thus, the diagonal hooks of $\lambda$ are labeled
by $H_0$ (and $K_1$).

\begin{example}
In Example~\ref{ex:ex1}, we see that there are four $0$s on the
horizontal and four 1s on the vertical {axe}, corresponding to the four
diagonal hooks of $\lambda$. 
\end{example}

\medskip

Let $p$ be an odd prime. We now consider a $p$-abacus with $p$
runners, labeled from $0$ to $p-1$ from left-to-right. 
We choose a position on the first runner and we label it by $0$. Then we
label positions by integers moving left-to-right to the
runner $p-1$, then wrapping around to runner $0$ one row above. In particular, the positions on the runner 0 are labeled by
$\cdots,-3p,\,-2p,\,-p,\,0,\,p,\,2p\cdots$. Now, we fill the abacus so
that there is a bead at the position labeled by $j$ if and only if
$l_j=0$. 
For example, Figure~2 
is the $p$-abacus of the empty partition.
\begin{figure}
\begin{center}
\label{fig:abacusempty}
\definecolor{sqsqsq}{rgb}{0.12549019607843137,0.12549019607843137,0.12549019607843137}
\begin{tikzpicture}[line cap=round,line join=round,>=triangle
45,x=1.0cm,y=1.0cm, scale=0.8,every node/.style={scale=0.8}]
\draw (1.140727272727281,3.8946) node[anchor=north west] {$\cdots$};
\draw (1.1589090909090993,1.8582363636363655) node[anchor=north west] {$\cdots$};
\draw (1.1952727272727357,-0.14176363636363368) node[anchor=north west] {$\cdots$};
\draw (-1.,4.)-- (-1.,-1.);
\draw (0.,4.)-- (0.,-1.);
\draw (3.,4.)-- (3.,-1.);
\draw [dash pattern=on 2pt off 2pt](-1.,4.)-- (-1.,5.);
\draw [dash pattern=on 2pt off 2pt](0.,4.)-- (0.,5.);
\draw [dash pattern=on 2pt off 2pt](3.,4.)-- (3.,5.);
\draw [dash pattern=on 2pt off 2pt](-1.,-1.)-- (-1.,-2.);
\draw [dash pattern=on 2pt off 2pt](0.,-1.)-- (0.,-2);
\draw [dash pattern=on 2pt off 2pt](3.,-1)-- (3.,-2);
\draw (-1.2,-2.4) node[anchor=north west] {$0$};
\draw (-0.2,-2.4) node[anchor=north west] {$1$};
\draw (2.531636363636372,-2.450854545454541) node[anchor=north west] {$p-1$};
\draw (-1.5,2.130963636363638) node[anchor=north west] {$0$};
\draw (-0.5,2.11278181818182) node[anchor=north west] {$1$};
\draw (3,2.0946) node[anchor=north west] {$p-1$};
\draw (3,1.1127818181818203) node[anchor=north west] {$-1$};
\draw (-1.5,3.1309636363636377) node[anchor=north west] {$p$};
\draw (-1.7,4.130963636363637) node[anchor=north west] {$2p$};
\draw (-1.768363636363628,1.1309636363636386) node[anchor=north west] {$-p$};
\begin{scriptsize}
\draw [color=black] (-1.,4.)-- ++(-2.5pt,0 pt) -- ++(5.0pt,0 pt) ++(-2.5pt,-2.5pt) -- ++(0 pt,5.0pt);
\draw [color=black] (-1.,3.)-- ++(-2.5pt,0 pt) -- ++(5.0pt,0 pt) ++(-2.5pt,-2.5pt) -- ++(0 pt,5.0pt);
\draw [color=black] (-1.,2.)-- ++(-2.5pt,0 pt) -- ++(5.0pt,0 pt) ++(-2.5pt,-2.5pt) -- ++(0 pt,5.0pt);
\draw [fill=sqsqsq] (-1.,1.) circle (2.5pt);
\draw [fill=black] (-1.,0.) circle (2.5pt);
\draw [fill=black] (-1.,-1.) circle (2.5pt);
\draw [color=black] (0.,4.)-- ++(-2.5pt,0 pt) -- ++(5.0pt,0 pt) ++(-2.5pt,-2.5pt) -- ++(0 pt,5.0pt);
\draw [color=black] (0.,3.)-- ++(-2.5pt,0 pt) -- ++(5.0pt,0 pt) ++(-2.5pt,-2.5pt) -- ++(0 pt,5.0pt);
\draw [color=black] (0.,2.)-- ++(-2.5pt,0 pt) -- ++(5.0pt,0 pt) ++(-2.5pt,-2.5pt) -- ++(0 pt,5.0pt);
\draw [fill=black] (0.,1.) circle (2.5pt);
\draw [fill=black] (0.,0.) circle (2.5pt);
\draw [fill=black] (0.,-1.) circle (2.5pt);
\draw [color=black] (3.,4.)-- ++(-2.5pt,0 pt) -- ++(5.0pt,0 pt) ++(-2.5pt,-2.5pt) -- ++(0 pt,5.0pt);
\draw [color=black] (3.,3.)-- ++(-2.5pt,0 pt) -- ++(5.0pt,0 pt) ++(-2.5pt,-2.5pt) -- ++(0 pt,5.0pt);
\draw [color=black] (3.,2.)-- ++(-2.5pt,0 pt) -- ++(5.0pt,0 pt) ++(-2.5pt,-2.5pt) -- ++(0 pt,5.0pt);
\draw [fill=black] (3.,1.) circle (2.5pt);
\draw [fill=black] (3.,0.) circle (2.5pt);
\draw [fill=black] (3.,0.) circle (2.5pt);
\draw [fill=black] (3.,-1.) circle (2.5pt);
\end{scriptsize}
\end{tikzpicture}
\caption{$p$-abacus of the empty partition}
\end{center}
\end{figure}
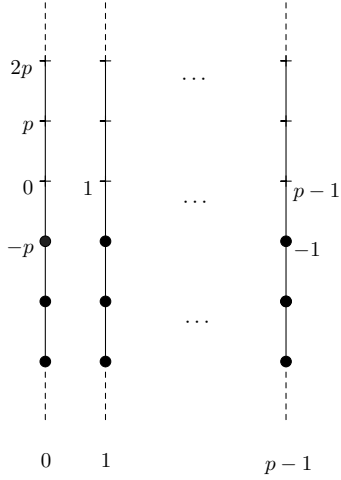

We can also read the diagonal hooks $\mathfrak{D}(\lambda)$ directly off of the $p$-abacus: they
are parametrized by the beads labeled by a non-negative integer.
More precisely,
if we set 
\begin{equation}
\label{eq:partitionsequencegamma}
l_{\gamma,j}=l_{jp+\gamma}
\end{equation}
for all $j\in \Z$, then the
beads on runner $\gamma$ can be interpreted as the partition
sequence $(l_{\gamma,j})_{j\in\Z}$ of a partition $\lambda_{\gamma}$.

\begin{remark}
In general, this labeling of the sequence is not compatible with that
of~(\ref{eq:partseq}). Indeed, there is no reason that there {should be} exactly
the same number of $1$s below  $l_{\gamma,0}$ as the number of $0$s
above it.
\end{remark}

{We define $\lambda_{\gamma}$ as the partition whose partition sequence
can be read {off} the beads on runner $\gamma$. That is, the abacus
position $\gamma +mp$ corresponds to a so-called {$\gamma$-position} $m$; that is, if
$\lambda$ has {a bead} in abacus position $\gamma+mp$ then $\lambda_{\gamma}$ {has a bead in position $m$ on runner $\gamma$}.
Then ${\mathcal Quo}_p(\lambda)$ is the {\it p-quotient} of $\lambda$, that is, the sequence
$(\lambda_0,\ldots,\lambda_{p-1})$. 
   
Now, for $0\leq\gamma\leq p-1$, define
{$$\mathcal X_{\gamma}=\{j\in\Z\mid pj+\gamma\geq 0\ \text{and }
l_{\gamma,j}=0\}.$$}
Therefore, each $j\in\mathcal X_\gamma$ {labels} a diagonal hook of
$\lambda$. Such hooks will be called \emph{diagonal hooks of $\lambda$
arising from {runner $\gamma$}}. Let ${\mathcal Cor}_p(\lambda)$ be the
{\it p-core} of $\lambda$, that is, the partition one obtains by removing
all the $p$-hooks of $\lambda$. {Such a partition is well-defined
\cite[p.\,79]{James-Kerber}.} Then, $\lambda$ is uniquely determined by
${\mathcal Cor}_p(\lambda)$ and ${\mathcal Quo}_p(\lambda)$.  

Let ${\mathcal Cor}^{(0)}_p(\lambda)={\mathcal Cor}_p(\lambda)$. Now
consider the $p$-tuple of $p$-abaci, one for each of the
$\lambda_{\gamma}\in {\mathcal Quo}_p(\lambda)$ above. Then ${\mathcal
Cor}^{(1)}_p(\lambda)$ will be a $p$-tuple defined to be the sequence
$({\mathcal Cor}_p(\lambda_{\gamma}))$ for $0\leq\gamma\leq p-1$. This
naturally induces a $p^2$-tuple $({\mathcal Quo}_p(\lambda_0),\cdots,
{\mathcal Quo}_p(\lambda_{p-1}))$, that defines ${\mathcal
Cor}^{(2)}_p(\lambda)$. Iterating this process 
we define ${\mathcal Cor}^{(k)}_p(\lambda)$ for any non-negative integer
$k$, and obtain at the end 
%
the $p$-core tower $\mathcal{CT}(\lambda)$ of $\lambda$ as
in~(\ref{eq:coretower}).

\begin{example}
\label{ex:ex2}
We continue with Example~\ref{ex:ex1}. Consider $p=3$. Then the
$p$-abacus of $\lambda$ is
 \begin{center}
\definecolor{sqsqsq}{rgb}{0.12549019607843137,0.12549019607843137,0.12549019607843137}
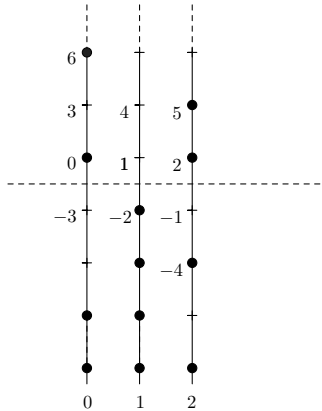
\begin{tikzpicture}[line cap=round,line join=round,>=triangle
45,x=1.0cm,y=1.0cm, scale=0.7,every node/.style={scale=0.7}]
\draw (-1.,4.)-- (-1.,-2.3);
\draw (0.,4.)-- (0.,-2.3);
\draw (1.,4.)-- (1.,-2.3);
\draw [dash pattern=on 2pt off 2pt](-1.,4.)-- (-1.,5.);
\draw [dash pattern=on 2pt off 2pt](0.,4.)-- (0.,5.);
\draw [dash pattern=on 2pt off 2pt](1.,4.)-- (1.,5.);
\draw [dash pattern=on 2pt off 2pt](-1.,-1.)-- (-1.,-2.);
\draw [dash pattern=on 2pt off 2pt](0.,-1.)-- (0.,-2);
\draw [dash pattern=on 2pt off 2pt](-2.5,1.5)-- (3.5,1.5);
\draw (-1.2,-2.4) node[anchor=north west] {$0$};
\draw (-0.2,-2.4) node[anchor=north west] {$1$};
\draw (0.78,-2.4) node[anchor=north west] {$2$};
\draw (-1.5,2.130963636363638) node[anchor=north west] {$0$};
\draw (-1.5,3.130963636363638) node[anchor=north west] {$3$};
\draw (-0.5,2.11278181818182) node[anchor=north west] {$1$};
\draw (-0.5,3.11278181818182) node[anchor=north west] {$4$};
\draw (0.5,2.0946) node[anchor=north west] {$2$};
\draw (0.5,3.0946) node[anchor=north west] {$5$};
\draw (0.25,1.1127818181818203) node[anchor=north west] {$-1$};
\draw (0.25,0.1127818181818203) node[anchor=north west] {$-4$};
\draw (-0.5,2.11278181818182) node[anchor=north west] {$1$};
\draw (-0.72,1.11) node[anchor=north west] {$-2$};
\draw (-1.5,4.130963636363637) node[anchor=north west] {$6$};
\draw (-1.768363636363628,1.1309636363636386) node[anchor=north west]
{$-3$};
\begin{scriptsize}
\draw [color=black] (-1.,1.)-- ++(-2.5pt,0 pt) -- ++(5.0pt,0 pt) ++(-2.5pt,-2.5pt) -- ++(0 pt,5.0pt);
\draw [color=black] (-1.,3.)-- ++(-2.5pt,0 pt) -- ++(5.0pt,0 pt) ++(-2.5pt,-2.5pt) -- ++(0 pt,5.0pt);
\draw [color=black] (-1.,0.)-- ++(-2.5pt,0 pt) -- ++(5.0pt,0 pt) ++(-2.5pt,-2.5pt) -- ++(0 pt,5.0pt);
\draw [fill=sqsqsq] (-1.,4.) circle (2.5pt);
\draw [fill=black] (-1.,2.) circle (2.5pt);
\draw [fill=black] (-1.,-1.) circle (2.5pt);
\draw [fill=black] (-1.,-2.) circle (2.5pt);
\draw [color=black] (0.,4.)-- ++(-2.5pt,0 pt) -- ++(5.0pt,0 pt) ++(-2.5pt,-2.5pt) -- ++(0 pt,5.0pt);
\draw [color=black] (0.,3.)-- ++(-2.5pt,0 pt) -- ++(5.0pt,0 pt) ++(-2.5pt,-2.5pt) -- ++(0 pt,5.0pt);
\draw [color=black] (0.,2.)-- ++(-2.5pt,0 pt) -- ++(5.0pt,0 pt) ++(-2.5pt,-2.5pt) -- ++(0 pt,5.0pt);
\draw [fill=black] (0.,1.) circle (2.5pt);
\draw [fill=black] (0.,0.) circle (2.5pt);
\draw [fill=black] (0.,-1.) circle (2.5pt);
\draw [fill=black] (0.,-2.) circle (2.5pt);
\draw [color=black] (1.,4.)-- ++(-2.5pt,0 pt) -- ++(5.0pt,0 pt) ++(-2.5pt,-2.5pt) -- ++(0 pt,5.0pt);
\draw [color=black] (1.,1.)-- ++(-2.5pt,0 pt) -- ++(5.0pt,0 pt) ++(-2.5pt,-2.5pt) -- ++(0 pt,5.0pt);
\draw [color=black] (1.,-1.)-- ++(-2.5pt,0 pt) -- ++(5.0pt,0 pt) ++(-2.5pt,-2.5pt) -- ++(0 pt,5.0pt);
\draw [fill=black] (1.,3.) circle (2.5pt);
\draw [fill=black] (1.,2.) circle (2.5pt);
\draw [fill=black] (1.,0.) circle (2.5pt);
\draw [fill=black] (1.,-2.) circle (2.5pt);
\end{scriptsize}
\end{tikzpicture}
\end{center}
Then $\lambda$ has four diagonal hooks corresponding
to the beads in positions $0$, $2$, $5$ and $6$. We have
$$\mathcal X_0=\{0,2\},\quad \mathcal
X_1=\emptyset\quad\text{and}\quad\mathcal X_2=\{0,1\}.$$
{By the discussion after Example 4.2,} the diagonal hooks arising from the $0$-runner have length $1$ and $13$. 
The ones arising from the $2$-runner have length $5$ and $11$. The partition sequences of $\lambda_0$, $\lambda_1$ and
$\lambda_2$ are respectively $\overline 011010\overline 1$, $\overline
0\,\overline 1$ and $\overline 0 10100\overline 1$. Thus,
$$\lambda_0=(3,2),\quad \lambda_1=\emptyset\quad\text{and}\quad
\lambda_2=(2^2,1).$$
\end{example}
\medskip

Suppose $\lambda=\lambda^*$. Then $\Lambda^*=\Lambda$, and
$l_{-\alpha}=1-l_{\alpha-1}$. {Since, by definition, $\alpha$ is the number of zeroes before the leftmost 1, and $\beta$ is the number of ones after the leftmost 0, this switch between 0 and 1 in each position implies that $\alpha=\beta$.} Moreover, for $0\leq u\leq \alpha-1$ and
$\delta\in\{0,1\}$, we have $l_u=\delta$ if and only if
$l_{-u-1}=1-\delta$. Denote by $\phi:\Z\mapsto\Z, u\rightarrow -u-1$.
{We define
{$$\mathcal Y_{\gamma}=\{j\in\Z\mid pj+\gamma\leq -1\ \text{and }
l_{\gamma,j}=1\}.$$}
\begin{lemma} Suppose $\phi$ is as above. Then the following hold. 
\begin{enumerate}[(1)]
\item $\phi$ is a bijection from ${\mathbb Z}$ to ${\mathbb Z}$.
\item $\phi$ induces a bijection $\phi|_{H_0}: H_0\rightarrow H_1$
with inverse map $\phi|_{H_1}:H_1\rightarrow H_0$.
\item $\phi^2=id.$
\item $\phi$ induces a bijection from $\mathcal X_{\gamma}$ to $\mathcal
Y_{p-\gamma-1}$.
\end{enumerate}
\end{lemma}
\begin{proof}
(1) and (3) are immediate. For (2), note, in particular, $\phi$ and the diagonal
hooks of $\lambda$ are the $h_{u,\phi(u)}$ for $u\in H_0$ of length
$2u+1$. For $u\in H_0$, we denote the corresponding diagonal hook-length
by
\begin{equation}
\label{eq:lengthdiag}
d_{u}=2u+1.
\end{equation}
To see (4), suppose that $u=jp+\gamma$ for $j\in\Z$ and $0\leq \gamma\leq p-1$. Then
$-u-1=-jp-\gamma-1=-(j+1)p+p-1-\gamma$ with $0\leq p-1-\gamma\leq p-1$.
Since $l_{\phi(u)}=1$ if and only if $l_u=0$, we have
\begin{equation}
\label{eq:runnersym}
l_{p-1-\gamma,j}=1-l_{\gamma,-(j+1)}
\end{equation}
which is the partition sequence of the conjugate partition of
$\lambda_{\gamma}$.
\end{proof}}
Assume that $\gamma\neq(p-1)/2$. Since $\mathcal X_{p-1-\gamma}$ labels
the diagonal hooks of $\lambda$ arising from runner $(p-1-\gamma)$,
$\mathcal Y_{\gamma}$ does too. Hence, the diagonal hooks of $\lambda$
arising from the runners $\gamma$ and $(p-1-\gamma)$ are parametrized by
$\mathcal X_{\gamma}\cup\mathcal Y_{\gamma}$. By~(\ref{eq:lengthdiag}),
for $x\in\mathcal X_{\gamma}$ and $x'\in\mathcal Y_{\gamma}$, the
corresponding diagonal hook-lengths of $\lambda$ are 
\begin{equation}
\label{eq:longXY}
d_{x}=2(xp+\gamma)+1\quad\text{and}\quad d_{x'}=2((-x'-1)p+p-1-\gamma)+1.
\end{equation}

Denote by $\Gamma$ a set {of representatives} of
$\{\gamma,p-1-\gamma\}$ for {$\{0,\ldots, j,\ldots,
p-1\}\backslash\{(p-1)/2\}$. }
By the discussion above, we have the following.
\begin{corollary}\label{eq:paramemptydiag}
The diagonal hooks of $\lambda$ are
parametrized by the elements of
\[
\mathcal X_{(p-1)/2}\cup\bigcup_{\gamma\in\Gamma}\mathcal
(X_{\gamma}\cup \mathcal Y_{\gamma}).
\]
\end{corollary}

Assume now that $\lambda=\lambda^*$ with ${\mathcal Cor}_p(\lambda)=\emptyset$. Furthermore, assume
that $\lambda_{(p-1)/2}=\emptyset$ where $\lambda_{\frac{(p-1)}{2}}\in {\mathcal Quo}_p(\lambda)$.
Let $0\leq \gamma\leq p-1$. Consider the partition sequence 
$(l_{\gamma,j})_{j\in\Z}$ as in~(\ref{eq:partitionsequencegamma}). Since
the $p$-abacus of Figure~2 
is the one that we obtain after removing all the $p$-hooks of $\lambda$
(because ${\mathcal Cor}_p(\lambda)$ is empty), it follows from the
construction of the $p$-quotient that the number of beads above $j=0$ is
the same as the number of empty positions under and strictly below $j=0$.
In particular, the sequence
$(l_{\gamma,j})_{j\in\Z}$ 
is compatible with the labeling of~(\ref{eq:partseq}), and the beads
over $j=0$ correspond to the diagonal hooks of $\lambda_{\gamma}$ and
are in bijection with the diagonal hooks of $\lambda$ arising from
runner $\gamma$. 

Since $\lambda_{p-1-\gamma}^*=\lambda_{\gamma}$, they have the same number of diagonal hooks. If $d$ is the length of the
$j$th-diagonal hook of $\lambda_{\gamma}$, then we denote by $d^*$ the length of the
$j$th-diagonal hook of $\lambda_{p-1-\gamma}$. 
Write $x\in \mathcal X_{\gamma}$ and $x^*\in \mathcal Y_{\gamma}$ such that $d=d_x$
and $d^*=d_{\phi(x^*)}$.
Then~(\ref{eq:lengthdiag}) gives 
\begin{equation}
\label{eq:dx}
d_x=2(xp+\gamma)+1\quad\text{and}\quad d_{x^*}=2(\phi(x^*)p+(p-1)-\gamma)+1.
\end{equation}
Hence, if we set $w_{x,x^*}=x-x^*$, then
\begin{equation}
\label{eq:brother}
d_x+d_{x^*}=2pw_{x,x^*}.
\end{equation}
{Moreover, by~(\ref{eq:paramemptydiag}) 
$$\mathfrak d(\lambda)=\bigcup_{\gamma\in\Gamma}\{d_x,\,d_{x^*}\,|\,x\in \mathcal
X_{\gamma}\},$$}
where $\mathfrak d(\lambda)$ is defined in~(\ref{eq:hooklength})

\begin{example}
Consider the partition $\lambda=(7^2,5,4,3,2^2)$ in
Example~\ref{ex:ex2}. 
We see from the $3$-abacus that
${\mathcal Cor}_3(\lambda)$ is empty. We also see that
$$\mathcal Y_{0}=\{-2,-1\}\quad\text{and}\quad \mathcal Y_2=\{-1,-3\}.$$
The bijection between $\mathcal Y_0$ and $\mathcal X_2$ is
$$-2\mapsto \phi(-2)=2-1=1\quad\text{and}\quad -1\mapsto
\phi(-1)=1-1=0.$$ Then $\mathfrak d(\lambda)$ is given
by~(\ref{eq:longXY})
{$$\{d_x\mid x\in \mathcal
X_{0}\}=\{d_0,\,d_2\}=\{1,13\}\quad\text{and}\quad
\{d_x\mid x\in \mathcal
Y_{0}\}=\{d_{-2},\,d_{-1}\}=\{5,11\}.$$}
In particular, the diagonal hooks of length $1$ of $\lambda_0$ and
$\lambda_2$ are associated with $1\in \mathcal X_0$ and $-1\in\mathcal
Y_0$. Similarly, the ones of length $4$ correspond to $2\in\mathcal X_0$
and $-3\in\mathcal Y_0$. It follows that
$$1^*=-1\quad\text{and}\quad 2^*=-3.$$
\end{example}
\subsection{Diagonal hooks of regular partitions}
Let $p$ be an odd prime, $n$ an integer divisible by $p$, and
$\lambda=\lambda^*$ be a partition of $n$. Let $n=n_1p+n_2p^2+\cdots+n_sp^s$ be its $p$-adic
expansion. Write $I=\{0,\ldots,p-1\}$ as above, and
{the $p$-core tower} $\mathcal{CT}(\lambda)$ of $\lambda$ as
in~(\ref{eq:coretower}). We
assume that the ${\mathcal Cor}_p(\lambda)=\emptyset$.  We say that $\lambda$ is a \emph{regular partition} when 
$c_k(\lambda)=n_k$ and $\lambda_{\underline{p}^*}=\emptyset$
where $\underline{p}^*\in I^k$ for any $1\leq k\leq s$.
On the other hand, $\lambda$ is called \emph{singular} whenever 
$\lambda_{\underline j}=\emptyset$, except possibly for $\underline j=\underline
p^*(k)\in I^k$, where $\underline p^*(k)$ is defined
in Equation~(\ref{eq:firstetoile}).

For $\lambda$ as above, we also define $\mathfrak r(\lambda)$ and $\mathfrak s(\lambda)$ the {\it regular} and {\it singular} parts (respectively) by giving their
$p$-core towers as follows. For $k\geq 0$ and $\underline j\in I^k$, if
$\underline j\neq \underline p^*(k)$, then we set $\lambda'_{\underline
j}=\lambda_{\underline j}$ and $\lambda''_{\underline j}=\emptyset$.
Otherwise, if $\underline j=\underline p^*(k)$, then write
$\lambda'_{\underline p^*(k)}=\emptyset$ and $\lambda''_{\underline
p^*(k)}=\lambda_{\underline p^*(k)}$. 

Therefore, the $p$-core towers of
$\mathfrak r(\lambda)$ and $\mathfrak s(\lambda)$ are given by
{\begin{equation}
\label{eq:ctlambdaprimeseconde}
{\mathcal Cor}^{(k)}_p(\mathfrak r(\lambda))=
\{\lambda'_{\underline j}\mid \underline j\in I^k\}\quad\text{and}\quad
{\mathcal Cor}^{(k)}_p(\mathfrak s(\lambda))=\{\lambda''_{\underline
j}\mid \underline j\in I^k\}\quad \text{for }k\geq 0.
\end{equation}}
Recall $\underline p^*(k)\in
I^k$. Then that $c_{k}(\mathfrak s(\lambda))=|\lambda_{\underline
p^*(k)}|$ and $c_{k}(\mathfrak
r(\lambda))=c_k(\lambda)-c_{k}(\mathfrak s(\lambda))$ by construction.
Hence, if we set $n'=\sum c_k(\mathfrak r(\lambda))p^k$ and
$n''=\sum c_k(\mathfrak s(\lambda))p^k$, then $n=n'+n''$ and ${\mathfrak
r(\lambda)}$ and ${\mathfrak s(\lambda)}$ are respectively regular and singular
partitions of $n'$ and $n''$ in the previous sense.
\begin{proposition}
\label{prop:crochetsfreres}
Let $n$ be an integer with $p$-adic expansion
$n=n_1p+n_2p^2+\cdots+n_sp^s$, where $p$ is an odd prime.
Let $\lambda$ be a regular partition with $p$-core tower
{${\mathcal Cor}^{(k)}_p(\lambda)=\{\lambda_{\underline
j}\mid \underline j\in I^{k}\}$} for $k\geq 0$.
For any integer $0\leq i\leq s-1$, write ${\mathcal H}_i$ for the set of diagonal
hooks lengths of $\lambda$ which are divisible by $p^i$ but not by $p^{i+1}$.
Then the elements of ${\mathcal H}_i$ are of the form $t_{u,i}=p^iu$ and
$t_{u,i}^*=p^i(w_{u,i}p-u)$, where $u\in U_i$ is an odd integer
relatively prime to $p$, and $w_{u,i}\in W_j$ is an even integer.
\end{proposition}
\begin{proof}
We proceed by induction on $s\geq 1$. Suppose that $s=1$.  Then
$n=n_1p$. Note that ${\mathcal Cor}_p(\lambda)=\emptyset$ by
assumption, thus we are in the situation described above.
By~(\ref{eq:dx}) we set $t_{u,0}=d_x$ and
$t_{u,0}^*=d_{x^*}$, and~(\ref{eq:brother}) gives that $t_{u,0}^* =
pw_u-t_{u,0}$
with $w_{u,0}=2w_{x,x^*}$. In particular, $t_{u,0}$ and $t_{u,0}^*$ are odd and
prime to $p$
and $w_{u,0}$ is even. 
The result is true for $s=1$.

Let $s\geq 1$. Suppose that the result holds for $s$. Let
$n=n_1p+n_2p^2+\cdots+n_sp^s+n_{s+1}p^{s+1}$, and $\lambda$ be a
partition of $n$ that satisfies the assumption. Consider
$\lambda'=\lambda_{(p-1)/2}\in{\mathcal Quo}_p(\lambda)$ and
$n'=|\lambda'|$. One has $n=p(n'+\sum_{j\neq (p-1)/2}|\lambda_j|)$
because
${\mathcal Cor}_p(\lambda)=\emptyset$, and $n'$ is divisible by $p$ because
${\mathcal Cor}_p(\lambda')=\emptyset$, since $\lambda$ is regular. Thus, 
the $p$-adic expansion of $n'$ is then of the form
$n'_1p+\cdots+n'_h p^{h}$ with $h\leq s$. 
By induction, the diagonal hooks of $\lambda'$
are as required. Now, there is a bijection $f$ between the diagonal hooks of
$\lambda$ divisible by $p$ and the diagonal hooks of
$\lambda_{(p-1)/2}$ such that $|f(h_{mm})|=p|h_{mm}|$, where $h_{mm}$ is
a diagonal hook of $\lambda_{(p-1)/2}$. In particular, for $1\leq i\leq
s$, we have $H_i=f(H'_{i-1})$ where $H'_{i-1}$ is the set of diagonal
hooks of $\lambda'$ divisible by $p^{i-1}$ but not by $p^i$. On the
other hand, since ${\mathcal Cor}_p(\lambda)=\emptyset$, $H_0$ is the set
of diagonal hooks arising from ${\mathcal Quo}_p(\lambda)=(\lambda_{0},\ldots,\lambda_{(p-3)/2},\emptyset,\lambda_{(p+1)/2},\ldots,
\lambda_{p-1})$,
and~(\ref{eq:dx}) and~(\ref{eq:brother}) give the result.
\end{proof}

\begin{proposition}
\label{prop:bougereg}
Let $\lambda$ be a regular partition of $n$. 
If $f\in\mathcal K_{n!/2}$ then
$\varepsilon(\chi_{\lambda},f)=1$. Moreover, 
$$\varepsilon(\chi_{\lambda},\sigma_{n!/2})=(-1)^{\frac{(p-1)n}4}.$$
\end{proposition}

\begin{proof}
First, we remark that if the $p$-adic expansion of $n$ is
$n_1p+\cdots+n_sp^s$ then each $n_i$ is even since $n_i=2\sum_{\underline
j}|\lambda_{\underline j}|$, where the sum runs
over $\underline j\neq \underline p^*(k)$ and $\underline j$ is a
representative of $\{\underline j,\underline j^*\}$. Here we use that
{${\lambda}$ is a symmetric} partition and that
$|\lambda_{\underline j}|=|\lambda_{\underline
j}^*|=|\lambda_{\underline j^*}|$. Now, by
Proposition~\ref{prop:crochetsfreres}, we have
\begin{align*}
\prod_{h\in {\mathfrak{d}(\lambda)}}h&=\prod_{i=0}^{s-1}\prod_{u\in
U_i}t_{u,i}t_{u,i}^*\\
&=\prod_{i=0}^{s-1}\prod_{u\in
U_i}p^{2i}u(w_{u,i}p-u).
\end{align*}  
Let $f$ be in $\mathcal H_{n!/2}$. With the 
notation~(\ref{eq:notationbouge}), we have
\begin{align}
\label{eq:cal0}
\varepsilon\left(\sqrt{\prod_{h\in{\mathfrak{
D}(\lambda)}}h},f\right)&=\varepsilon\left(\prod_{i=0}^{s-1}\prod_{u\in
U_i}\sqrt{u(w_{u,i}p-u)},f\right)\\
\nonumber&=\prod_{i=0}^{\pu{s-1}}\varepsilon\left(\prod_{u\in
U_i}\sqrt{u(w_{u,i}p-u)},f\right).
\end{align}
Note that $u$ and $(w_{u,i}p-u)$ are odd. Furthermore, 
{
\begin{equation}
\label{eq:cal1}
u(w_{u,i}p-u)=
\begin{cases}
2-u^2\equiv 1\mod 4&\text{if } w_{u,i}\equiv 2\mod 4,\\
-u^2\equiv -1\mod 4&\text{if } w_{u,i}\equiv 0\mod 4.\\
\end{cases}
\end{equation}
}
We also have
$$n=\sum_{h\in D(\lambda)}h=\sum_{i=0}^{s-1}p^{i+1}\sum_{u\in
U_i}w_{u,i}.$$ Since $w_{u,i}$ is even, there is an integer $w_{u,i}'$
such that $w_{u,i}=2w_{u,i}'$, and
\begin{align}
\label{eq:cal2}
\frac n 2&=\sum_{i=0}^{s-1}p^{i+1}\sum_{u\in
U_i}w_{u,i}'\\
\nonumber&\equiv \sum_{i=0}^{s-1}\sum_{u\in
U_i}w_{u,i}'\mod 2,
\end{align}
because $p$ is odd. Now, write {$A=\{w'_{u,i}\mid 0\leq i\leq s-1,\,u\in
U_i\}$}, and $A_{\text{even}}$ and $A_{\text{odd}}$ {for} the
subsets of even and odd elements of $A$, respectively. Then
$|A|=\frac{d_{\lambda}}{2}$ and~(\ref{eq:cal2}) gives 
$$\frac n 2\equiv \sum_{w\in A_{\text{odd}}}w\equiv \sum_{w\in
A_{\text{odd}}}1\equiv |A_{\text{odd}}| \mod 2.$$
Since $|A|=|A_{\text{odd}}|+|A_{\text{even}}|$, we deduce
from~(\ref{eq:cal1}) that
\begin{equation}
\label{eq:cal3}
\prod_{i=0}^{s-1}\prod_{u\in U_i}u(w_{u,i}p-u)\equiv
(-1)^{|A_{\text{even}}|}\equiv (-1)^{|A|-|A_{\text{odd}|}}\equiv
(-1)^{\frac{n-d_{\lambda}}2}\mod 4. 
\end{equation}
Thus, by~(\ref{eq:cal0}) and 
Proposition~\ref{prop:racbouge} we obtain
\begin{equation}
\label{eq:cal4}
\varepsilon\left(\sqrt{\prod_{h\in
{\mathfrak{d}(\lambda)}}h},f\right)=\varepsilon(i,f)^{\frac{(p-1)(n-d_{\lambda})}4}\prod_{i=0}^{s-1}\prod_{u\in
U_i}\left(\frac{r}{u}\right)\left(\frac{r}{w_{u,i}p-u}\right),
\end{equation}
where $r$ is such that $f(\omega_m)=\omega_m^r$ for $m=\prod_{i,u}
u(w_{u,i}p-u)$. Note that if $f\in\mathcal K_{n!/2}$, then $f$ acts
trivially on $i$ and on $\omega_m$, that is $r=1$, 
and~(\ref{eq:bougealt}) implies that $\varepsilon(\chi_{\lambda},f)=1$. 
Assume that $f=\sigma_{n!/2}$, that is $r=p$. On the other hand, by
quadratic reciprocity, one has
\begin{align}
\label{eq:calculsymbol}
\left(\frac{p}{u}\right)\left(\frac{p}{w_{u,i}p-u}\right)&=(-1)^{\frac{p-1}2
\left(\frac{u-1}2+\frac{w_{u,i}p-u-1}2\right)}\left(\frac{-1}{p}\right)\\
\nonumber &=(-1)^{\frac{p-1}2
\left(\frac{u-1}2+\frac{w_{u,i}p-u-1}2+1\right)}\\
\nonumber &=(-1)^{\frac{(p-1)w_{u,i}}4}\\
\nonumber &=(-1)^{\frac{(p-1)w_{u,i}'}2}\\
\nonumber &=
\begin{cases}
1&\text{if }w_{u,i}'\equiv 0 \mod 2,\\
-1&\text{if }w_{u,i}'\equiv 1\mod 2.  
\end{cases}
\end{align}
Using~(\ref{eq:bougealt}), it follows that
\begin{align*}
\varepsilon(\chi_{\lambda},\sigma_{n!/2})&=(-1)^{\frac{(p-1)(n-d_{\lambda})}4}\varepsilon\left(\sqrt{\prod_{h\in
{\mathfrak{d}(\lambda)}}h},f\right)\\
&=(-1)^{\frac{(p-1)(n-d_{\lambda})}4}\cdot
(-1)^{\frac{(p-1)(n-d_{\lambda})}4}\cdot(-1)^{\frac{(p-1)|A_{\text{odd}}|}2}\\
&=(-1)^{\frac{(p-1)n}{4}},
\end{align*}
as required.
\end{proof}
\subsection{Diagonal hooks of partitions with non-empty $p$-core}
\label{subsec:nontriv}
For any partition $\lambda$, we denote by ${\mathcal Q}_p(\lambda)$ the
partition with the same $p$-quotient as $\lambda$ but with empty $p$-core.
That is, ${\mathcal Quo}_p({\mathcal Q}_p(\lambda))={\mathcal Quo}_p(\lambda)$
but ${\mathcal Cor}_p({\mathcal Q}_p(\lambda))=\emptyset.$

Let $\lambda=\lambda^*$. 
{Write $\mathcal M=(m_u)_{u\in\Z}$,  $\mathcal M'=(m'_u)_{u\in\Z}$ and
$\Lambda=(l_u)_{u\in\Z}$ for the partition sequences with the labeling as
in~(\ref{eq:partseq}) associated to $\lambda$,}
${\mathcal Cor}_p(\lambda)$ and ${\mathcal Q}_p(\lambda)$ respectively.

Since $\lambda=\lambda^*$ we have ${\mathcal Cor}_p(\lambda)={\mathcal
Cor}^*_{p}{(\lambda)}$ by~\cite[Prop.\,3.5]{Olsson}. Let
$0\leq\gamma\leq p-1$. By definition of a {$p$-core,}
if $m'_{\gamma}=0$,
then there is an integer $\delta_{\gamma}>0$ such that
$m'_{pj+\gamma}=0$ if and only if $j\leq \delta_{\gamma}-1$.
Since
${\mathcal Cor}_p(\lambda)={\mathcal Cor}^{*}_{p}(\lambda)$ it follows from~\S\ref{subsec:combinatorics}
that $m'_{pj+(p-1)-\gamma}=0$ if and only if $j<-\delta_{\gamma}$.
If $m'_{\gamma}=1$ and $m'_{-p-\gamma}=0$ then $m'_{(p-1)-\gamma}=1$ and
$m'_{-p+(p-1)-\gamma}=0$. In this last case, we set $\delta_{\gamma}=0$.
Let $\gamma$ be such that $\delta_{\gamma}>0$. {Define
$\Delta_{\gamma}=\{0\leq j\leq \delta_{\gamma}-1\}$.} Then elements of ${\mathfrak D}({\mathcal Cor}_p(\lambda))$ are {labeled} by the elements of
$\cup_{\delta_{\gamma}>0}{\Delta}_{\gamma}$. In particular ${\mathcal Cor}_p(\lambda)$ has
$\sum_{\delta_{\gamma}>0}\delta_{\gamma}$ diagonal hooks.

We construct the $p$-abacus of $\lambda$ from that of ${\mathcal Q}_p(\lambda)$ as
follows. 
If $\delta_{\gamma}=0$ then the runners $\gamma$ and
$(p-1-\gamma)$ of ${\mathcal Q}_p(\lambda)$ and $\lambda$ are identical.
If $\delta_{\gamma}>0$, then runner $\gamma$ of $\lambda$ (resp. the
runner $p-1-\gamma$ of $\lambda$) is obtained by shifting up (resp. down) the corresponding runner of ${\mathcal Q}_p(\lambda)$ $\delta_{\gamma}$ positions.
It follows that, for all $0\leq\gamma\leq p-1$ such that
$\delta_{\gamma}\geq 0$, one has
\begin{equation}
\label{eq:relationquotientvidequotient}
m_{(j+\delta_{\gamma})p+\gamma}=
l_{jp+\gamma}\quad\text{and}\quad
m_{(j-\delta_{\gamma})p+p-1-\gamma}=l_{jp+p-1-\gamma}\quad \text{for all
}j\in\Z.
\end{equation}

We will now describe how to obtain ${\mathfrak D}(\lambda)$ from ${\mathfrak D}({\mathcal Q}_p(\lambda))$. For $\gamma\in\Gamma\cup\{(p-1)/2\}$, we denote by
$\mathcal X_{\gamma}$ and $\mathcal Y_{\gamma}$ (respectively $\mathcal
X'_{\gamma}$ and $\mathcal Y'_{\gamma}$) the sets as
in~(\ref{eq:paramemptydiag}) that label the diagonal hooks of
{${\mathcal Q}_p(\lambda)$ (respectively, of $\lambda$)}. 

We remark that if $\delta_{\gamma}=0$, then $\mathcal
X_{\gamma}=\mathcal X'_{\gamma}$ and $\mathcal Y_{\gamma}=\mathcal
Y'_{\gamma}$, that is the hooks of 
$\lambda$ and ${\mathcal Q}_p(\lambda)$ arising from runner $\gamma$ are the same. Note
that $\delta_{(p-1)/2}=0$, {since $\lambda=\lambda^*$}.

Suppose $\delta_{\gamma}>0$. {We introduce four possibilities in passing from the diagonal hooks of ${\mathcal Q}_p(\lambda)$ to those of $\lambda$.}

\begin{enumerate}[(i)]
\item
Any $x\in \mathcal X_{\gamma}$ corresponds to a hook labeled by $x+\delta_{\gamma}\in \mathcal X'_{\gamma}$ of $\lambda$ on the 
$\gamma$-runner. More precisely,
by~(\ref{eq:relationquotientvidequotient}) we can
associate to the hook of length $d_x$ of
${\mathcal Q}_p(\lambda)$ labeled by $x$ given
in~(\ref{eq:dx}), a hook of $\lambda$ of length
\begin{equation}
\label{eq:imx0}
c(d_x)=2((x+\delta_{\gamma})p+\gamma)+1.
\end{equation} 
{We will call this \it{an increase of the length of an existing hook with respect to $\gamma$}.}
\item
Similarly, for $x\in\mathcal Y_{\gamma}$ such that
$x<-\delta_{\gamma}$, we have
$\delta_{\gamma}+x <0$, and $\delta_{\gamma}+x\in\mathcal Y'_{\gamma}$
by~(\ref{eq:relationquotientvidequotient}). 
By~(\ref{eq:dx}), we associate to $d_x$ a hook of $\lambda$ of length
\begin{equation}
\label{eq:imx1}
c(d_x)=2(\phi(\delta_{\gamma}+x)p+(p-1)-\gamma)+1.
\end{equation}
{We will refer to this as \it{an increase of the length of an existing hook with respect to $\gamma^*=p-\gamma-1$}.}
\item Let $-\delta_{\gamma}\leq x\leq -1$ be such that $x\notin \mathcal
Y_{\gamma}$, that is $l_{xp+\gamma}=0$.
Then $x+\delta_{\gamma}\geq 0$ and
by~(\ref{eq:relationquotientvidequotient}), $x+\delta_{\gamma}\in
\mathcal X'_{\gamma}$.
Hence, a new diagonal hook of length
$$c_x=2((\delta_{\gamma}+x)p+\gamma)+1$$ appears in $\lambda$. This is also a {diagonal} hook of
${\mathcal Cor}_p(\lambda)$.
{We will call this \it{the appearance of a new hook with respect to $\gamma$.}}
\item Finally, let $-\delta_{\gamma}\leq x\leq -1$ be such that $x\in \mathcal
Y_{\gamma}$, that is $l_{xp+\gamma}=1$. Then $x+\delta_{\gamma}\notin
\mathcal X'_{\gamma}$. Then the hook of ${\mathcal Q}_p(\lambda)$ labeled by $x$ gives
no hook of $\lambda$.
{We will call this \it{the disappearance of an existing hook with respect to $\gamma^*=p-\gamma-1$.}}
\end{enumerate}

\begin{remark}
\label{rk:diagcore}
Let $\mathcal A_{\gamma}$ and $\mathcal
B_{\gamma}$ be the set of
$-\delta_{\gamma}\leq x\leq -1$ such that $l_{px+\gamma}=0$ and
$l_{px+\gamma}=1$, respectively. 
Then $\mathcal A_{\gamma}\sqcup\mathcal B_{\gamma}$ labels the
diagonal hooks of ${\mathcal Cor}_p(\lambda))$ as follows: associate the set of diagonal hooks of
${\mathcal Cor}_p(\lambda)$ of length
\begin{equation}
\label{eq:paramdiagcore}
c_x=2((\delta_{\gamma}+x)p+\gamma)+1
\end{equation}
to $\mathcal A_{\gamma}\sqcup\mathcal B_{\gamma}$.
\end{remark}
In the next example we use the fact that the $p$-abacus of {
${\mathcal Cor}_p(\lambda)$ is obtained} from the $p$-abacus of $\lambda$
by placing beads in empty positions one position below them on each runner until this is no longer possible, and then
reading off the {resulting} partition from the new $p$-abacus
configuration.  by~\cite[p.\,79]{James-Kerber}. 
\begin{example}
Let {$\lambda=(16,11,3,2^8,1^5)$}. We find ${\mathfrak D}(\lambda)$
using the $3$-abaci of ${\mathcal Cor}_3(\lambda)$ and {${\mathcal
Q}_3(\lambda)$}.
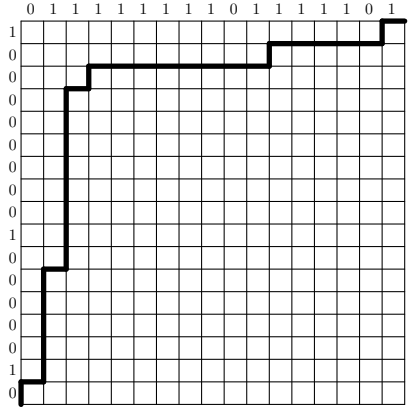
\begin{center}
\begin{tikzpicture}[line cap=round,line join=round,>=triangle
45,x=1.0cm,y=1.0cm,scale=0.6,every node/.style={scale=0.6}]
\clip(-1.6614876033057846,-8.) rectangle (10.87570247933884,1.758677685950413);
\draw (5.,-7.5)-- (5.,1.);
\draw (5.5,1.)-- (5.5,-7.5);
\draw (6.,-7.5)-- (6.,1.);
\draw (6.5,1.)-- (6.5,-7.5);
\draw (7.,-7.5)-- (7.,1.);
\draw (7.5,1.)-- (7.5,-7.5);
\draw (8.,-7.5)-- (8.,1.);
\draw (8.5,-7.5)-- (8.5,1.);

\draw (0.5,1.)-- (0.5,-7.5);
\draw (1.,1.)-- (1.,-7.5);
\draw (1.5,1.)-- (1.5,-7.5);
\draw (2.,1.)-- (2.,-7.5);
\draw (2.5,1.)-- (2.5,-7.5);
\draw (3.,-7.5)-- (3.,1.);
\draw (3.5,1.)-- (3.5,-7.5);
\draw (4.,-7.5)-- (4.,1.);
\draw (4.5,1.)-- (4.5,-7.5);
\draw (0.,1.)-- (8.5,1.);
\draw (0.,0.5)-- (8.5,0.5);
\draw (8.5,0.)-- (0.,0.);
\draw (0.,-0.5)-- (8.5,-0.5);
\draw (8.5,-1.)-- (0.,-1.);
\draw (0.,-1.5)-- (8.5,-1.5);
\draw (0.,-2.)-- (8.5,-2.);
\draw (0.,-2.5)-- (8.5,-2.5);
\draw (0.,-3.)-- (8.5,-3.);
\draw (0.,-4.)-- (8.5,-4.);
\draw (8.5,-3.5)-- (0.,-3.5);
\draw (0.,-4.5)-- (8.5,-4.5);
\draw (0.,-5)-- (8.5,-5);
\draw (0.,-5.5)-- (8.5,-5.5);
\draw (0.,-6.)-- (8.5,-6.);
\draw (0.,-6.5)-- (8.5,-6.5);
\draw (0.,-7.)-- (8.5,-7.);
\draw (0.,-7.5)-- (8.5,-7.5);
\draw (0.,-7.5)-- (0.,1.);

\draw (-0.4,-7) node[anchor=north west]
{$0$};
\draw (-0.4,-6.5) node[anchor=north west]
{$1$};
\draw (-0.4,-6) node[anchor=north west]
{$0$};
\draw (-0.4,-5.5) node[anchor=north west]
{$0$};
\draw (-0.4,-5) node[anchor=north west]
{$0$};
\draw (-0.4,-4.5) node[anchor=north west]
{$0$};
\draw (-0.4,-4) node[anchor=north west]
{$0$};
\draw (-0.4,-3.5) node[anchor=north west]
{$1$};

\draw (-0.4,-3.) node[anchor=north west]
{$0$};
\draw (-0.4,-2.5) node[anchor=north west]
{$0$};
\draw (-0.4,-2.) node[anchor=north west]
{$0$};
\draw (-0.4,-1.5) node[anchor=north west]
{$0$};
\draw (-0.4,-1.) node[anchor=north west]
{$0$};
\draw (-0.4,-0.5) node[anchor=north west]
{$0$};
\draw (-0.4,0) node[anchor=north west]
{$0$};
\draw (-0.4,0.5) node[anchor=north west]
{$0$};
\draw (-0.4,1) node[anchor=north west]
{$1$};

\draw (0,1.5) node[anchor=north west]
{$0$};
\draw (0.5,1.5) node[anchor=north west]
{$1$};
\draw (1,1.5) node[anchor=north west]
{$1$};
\draw (1.5,1.5) node[anchor=north west]
{$1$};
\draw (2,1.5) node[anchor=north west]
{$1$};
\draw (2.5,1.5) node[anchor=north west]
{$1$};
\draw (3,1.5) node[anchor=north west]
{$1$};
\draw (3.5,1.5) node[anchor=north west]
{$1$};
\draw (4,1.5) node[anchor=north west]
{$1$};
\draw (4.5,1.5) node[anchor=north west]
{$0$};
\draw (5,1.5) node[anchor=north west]
{$1$};
\draw (5.5,1.5) node[anchor=north west]
{$1$};
\draw (6,1.5) node[anchor=north west]
{$1$};
\draw (6.5,1.5) node[anchor=north west]
{$1$};
\draw (7,1.5) node[anchor=north west]
{$1$};
\draw (7.5,1.5) node[anchor=north west]
{$0$};
\draw (8,1.5) node[anchor=north west]
{$1$};

\draw [line width=2.pt] (8.5,1.)-- (8,1);
\draw [line width=2.pt] (8,1)-- (8,0.5);
\draw [line width=2.pt] (8,0.5)-- (5.5,0.5);
\draw [line width=2.pt] (5.5,0.5)-- (5.5,-0);
\draw [line width=2.pt] (5.5,-0.)-- (1.5,-0);
\draw [line width=2.pt] (1.5,-0.)-- (1.5,-0.5);
\draw [line width=2.pt] (1.5,-0.5)-- (1.,-0.5);
\draw [line width=2.pt] (1,-0.5)-- (1,-4.5);
\draw [line width=2.pt] (1,-4.5)-- (0.5,-4.5);
\draw [line width=2.pt] (0.5,-4.5)-- (0.5,-7);
\draw [line width=2.pt] (0.5,-7)-- (0,-7);
\draw [line width=2.pt] (-0.,-7.5)-- (-0.,-7);
\end{tikzpicture}
\end{center}
In particular, the $3$-abaci of $\lambda$ and of ${\mathcal Cor}_3(\lambda)$ are depicted below:

 \begin{center}
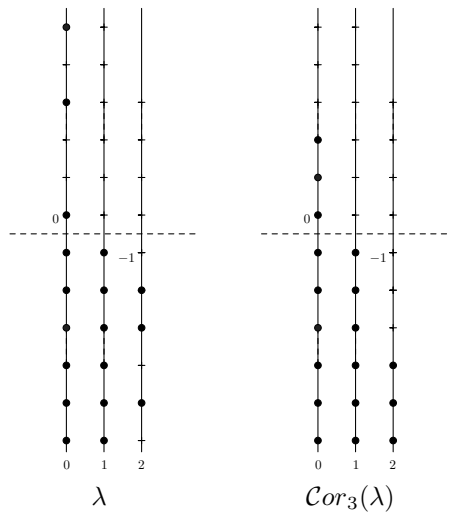
\begin{tabular}{ccc}
\definecolor{sqsqsq}{rgb}{0.12549019607843137,0.12549019607843137,0.12549019607843137}
\begin{tikzpicture}[line cap=round,line join=round,>=triangle
45,x=1.0cm,y=1.0cm, scale=0.5,every node/.style={scale=0.5}]
\draw (-1.,7.5)-- (-1.,-4.3);
\draw (0.,7.5)-- (0.,-4.3);
\draw (1.,7.5)-- (1.,-4.3);
\draw [dash pattern=on 2pt off 2pt](-1.,4.)-- (-1.,5.);
\draw [dash pattern=on 2pt off 2pt](0.,4.)-- (0.,5.);
\draw [dash pattern=on 2pt off 2pt](1.,4.)-- (1.,5.);
\draw [dash pattern=on 2pt off 2pt](-1.,-1.)-- (-1.,-2.);
\draw [dash pattern=on 2pt off 2pt](0.,-1.)-- (0.,-2);
\draw [dash pattern=on 2pt off 2pt](-2.5,1.5)-- (2.5,1.5);
\draw (-1.2,-4.4) node[anchor=north west] {$0$};
\draw (-0.2,-4.4) node[anchor=north west] {$1$};
\draw (0.78,-4.4) node[anchor=north west] {$2$};
\draw (-1.5,2.130963636363638) node[anchor=north west] {$0$};
\draw (0.25,1.1127818181818203) node[anchor=north west] {$-1$};
\begin{scriptsize}
\draw [color=black] (-1.,4.)-- ++(-2.5pt,0 pt) -- ++(5.0pt,0 pt) ++(-2.5pt,-2.5pt) -- ++(0 pt,5.0pt);
\draw [color=black] (-1.,6.)-- ++(-2.5pt,0 pt) -- ++(5.0pt,0 pt) ++(-2.5pt,-2.5pt) -- ++(0 pt,5.0pt);
\draw [color=black] (-1.,3.)-- ++(-2.5pt,0 pt) -- ++(5.0pt,0 pt) ++(-2.5pt,-2.5pt) -- ++(0 pt,5.0pt);
\draw [fill=sqsqsq] (-1.,7.) circle (2.5pt);
\draw [fill=black] (-1.,5.) circle (2.5pt);
\draw [fill=black] (-1.,2.) circle (2.5pt);
\draw [fill=black] (-1.,1.) circle (2.5pt);
\draw [fill=black] (-1.,0.) circle (2.5pt);
\draw [fill=sqsqsq] (-1.,-1.) circle (2.5pt);
\draw [fill=black] (-1.,-2.) circle (2.5pt);
\draw [fill=black] (-1.,-3.) circle (2.5pt);
\draw [fill=black] (-1.,-4.) circle (2.5pt);

\draw [color=black] (0.,7.)-- ++(-2.5pt,0 pt) -- ++(5.0pt,0 pt) ++(-2.5pt,-2.5pt) -- ++(0 pt,5.0pt);
\draw [color=black] (0.,6.)-- ++(-2.5pt,0 pt) -- ++(5.0pt,0 pt) ++(-2.5pt,-2.5pt) -- ++(0 pt,5.0pt);
\draw [color=black] (0.,5.)-- ++(-2.5pt,0 pt) -- ++(5.0pt,0 pt) ++(-2.5pt,-2.5pt) -- ++(0 pt,5.0pt);
\draw [color=black] (0.,4.)-- ++(-2.5pt,0 pt) -- ++(5.0pt,0 pt) ++(-2.5pt,-2.5pt) -- ++(0 pt,5.0pt);
\draw [color=black] (0.,3.)-- ++(-2.5pt,0 pt) -- ++(5.0pt,0 pt) ++(-2.5pt,-2.5pt) -- ++(0 pt,5.0pt);
\draw [color=black] (0.,2.)-- ++(-2.5pt,0 pt) -- ++(5.0pt,0 pt) ++(-2.5pt,-2.5pt) -- ++(0 pt,5.0pt);
\draw [fill=black] (0.,1.) circle (2.5pt);
\draw [fill=black] (0.,0.) circle (2.5pt);
\draw [fill=black] (0.,-1.) circle (2.5pt);
\draw [fill=black] (0.,-2.) circle (2.5pt);
\draw [fill=black] (0.,-3.) circle (2.5pt);
\draw [fill=black] (0.,-4.) circle (2.5pt);

\draw [color=black] (1.,5.)-- ++(-2.5pt,0 pt) -- ++(5.0pt,0 pt) ++(-2.5pt,-2.5pt) -- ++(0 pt,5.0pt);
\draw [color=black] (1.,4.)-- ++(-2.5pt,0 pt) -- ++(5.0pt,0 pt) ++(-2.5pt,-2.5pt) -- ++(0 pt,5.0pt);
\draw [color=black] (1.,3.)-- ++(-2.5pt,0 pt) -- ++(5.0pt,0 pt) ++(-2.5pt,-2.5pt) -- ++(0 pt,5.0pt);
\draw [color=black] (1.,2.)-- ++(-2.5pt,0 pt) -- ++(5.0pt,0 pt) ++(-2.5pt,-2.5pt) -- ++(0 pt,5.0pt);
\draw [color=black] (1.,1)-- ++(-2.5pt,0 pt) -- ++(5.0pt,0 pt) ++(-2.5pt,-2.5pt) -- ++(0 pt,5.0pt);
\draw [color=black] (1.,0)-- ++(-2.5pt,0 pt) -- ++(5.0pt,0 pt) ++(-2.5pt,-2.5pt) -- ++(0 pt,5.0pt);

\draw [color=black] (1.,1.)-- ++(-2.5pt,0 pt) -- ++(5.0pt,0 pt) ++(-2.5pt,-2.5pt) -- ++(0 pt,5.0pt);
\draw [color=black] (1.,-2.)-- ++(-2.5pt,0 pt) -- ++(5.0pt,0 pt) ++(-2.5pt,-2.5pt) -- ++(0 pt,5.0pt);
\draw [color=black] (1.,-4.)-- ++(-2.5pt,0 pt) -- ++(5.0pt,0 pt) ++(-2.5pt,-2.5pt) -- ++(0 pt,5.0pt);
\draw [fill=black] (1.,0.) circle (2.5pt);
\draw [fill=black] (1.,-1.) circle (2.5pt);
\draw [fill=black] (1.,-3.) circle (2.5pt);
\end{scriptsize}
\end{tikzpicture}
&&\definecolor{sqsqsq}{rgb}{0.12549019607843137,0.12549019607843137,0.12549019607843137}
\begin{tikzpicture}[line cap=round,line join=round,>=triangle
45,x=1.0cm,y=1.0cm, scale=0.5,every node/.style={scale=0.5}]
\draw (-1.,7.5)-- (-1.,-4.3);
\draw (0.,7.5)-- (0.,-4.3);
\draw (1.,7.5)-- (1.,-4.3);
\draw [dash pattern=on 2pt off 2pt](-1.,4.)-- (-1.,5.);
\draw [dash pattern=on 2pt off 2pt](0.,4.)-- (0.,5.);
\draw [dash pattern=on 2pt off 2pt](1.,4.)-- (1.,5.);
\draw [dash pattern=on 2pt off 2pt](-1.,-1.)-- (-1.,-2.);
\draw [dash pattern=on 2pt off 2pt](0.,-1.)-- (0.,-2);
\draw [dash pattern=on 2pt off 2pt](-2.5,1.5)-- (2.5,1.5);
\draw (-1.2,-4.4) node[anchor=north west] {$0$};
\draw (-0.2,-4.4) node[anchor=north west] {$1$};
\draw (0.78,-4.4) node[anchor=north west] {$2$};
\draw (-1.5,2.130963636363638) node[anchor=north west] {$0$};
\draw (0.25,1.1127818181818203) node[anchor=north west] {$-1$};
\begin{scriptsize}
\draw [color=black] (-1.,5.)-- ++(-2.5pt,0 pt) -- ++(5.0pt,0 pt) ++(-2.5pt,-2.5pt) -- ++(0 pt,5.0pt);
\draw [color=black] (-1.,6.)-- ++(-2.5pt,0 pt) -- ++(5.0pt,0 pt) ++(-2.5pt,-2.5pt) -- ++(0 pt,5.0pt);
\draw [color=black] (-1.,7.)-- ++(-2.5pt,0 pt) -- ++(5.0pt,0 pt) ++(-2.5pt,-2.5pt) -- ++(0 pt,5.0pt);
\draw [fill=sqsqsq] (-1.,3.) circle (2.5pt);
\draw [fill=black] (-1.,4.) circle (2.5pt);
\draw [fill=black] (-1.,2.) circle (2.5pt);
\draw [fill=black] (-1.,1.) circle (2.5pt);
\draw [fill=black] (-1.,0.) circle (2.5pt);
\draw [fill=sqsqsq] (-1.,-1.) circle (2.5pt);
\draw [fill=black] (-1.,-2.) circle (2.5pt);
\draw [fill=black] (-1.,-3.) circle (2.5pt);
\draw [fill=black] (-1.,-4.) circle (2.5pt);

\draw [color=black] (0.,7.)-- ++(-2.5pt,0 pt) -- ++(5.0pt,0 pt) ++(-2.5pt,-2.5pt) -- ++(0 pt,5.0pt);
\draw [color=black] (0.,6.)-- ++(-2.5pt,0 pt) -- ++(5.0pt,0 pt) ++(-2.5pt,-2.5pt) -- ++(0 pt,5.0pt);
\draw [color=black] (0.,5.)-- ++(-2.5pt,0 pt) -- ++(5.0pt,0 pt) ++(-2.5pt,-2.5pt) -- ++(0 pt,5.0pt);
\draw [color=black] (0.,4.)-- ++(-2.5pt,0 pt) -- ++(5.0pt,0 pt) ++(-2.5pt,-2.5pt) -- ++(0 pt,5.0pt);
\draw [color=black] (0.,3.)-- ++(-2.5pt,0 pt) -- ++(5.0pt,0 pt) ++(-2.5pt,-2.5pt) -- ++(0 pt,5.0pt);
\draw [color=black] (0.,2.)-- ++(-2.5pt,0 pt) -- ++(5.0pt,0 pt) ++(-2.5pt,-2.5pt) -- ++(0 pt,5.0pt);
\draw [fill=black] (0.,1.) circle (2.5pt);
\draw [fill=black] (0.,0.) circle (2.5pt);
\draw [fill=black] (0.,-1.) circle (2.5pt);
\draw [fill=black] (0.,-2.) circle (2.5pt);
\draw [fill=black] (0.,-3.) circle (2.5pt);
\draw [fill=black] (0.,-4.) circle (2.5pt);

\draw [color=black] (1.,5.)-- ++(-2.5pt,0 pt) -- ++(5.0pt,0 pt) ++(-2.5pt,-2.5pt) -- ++(0 pt,5.0pt);
\draw [color=black] (1.,4.)-- ++(-2.5pt,0 pt) -- ++(5.0pt,0 pt) ++(-2.5pt,-2.5pt) -- ++(0 pt,5.0pt);
\draw [color=black] (1.,3.)-- ++(-2.5pt,0 pt) -- ++(5.0pt,0 pt) ++(-2.5pt,-2.5pt) -- ++(0 pt,5.0pt);
\draw [color=black] (1.,2.)-- ++(-2.5pt,0 pt) -- ++(5.0pt,0 pt) ++(-2.5pt,-2.5pt) -- ++(0 pt,5.0pt);
\draw [color=black] (1.,1)-- ++(-2.5pt,0 pt) -- ++(5.0pt,0 pt) ++(-2.5pt,-2.5pt) -- ++(0 pt,5.0pt);
\draw [color=black] (1.,0)-- ++(-2.5pt,0 pt) -- ++(5.0pt,0 pt) ++(-2.5pt,-2.5pt) -- ++(0 pt,5.0pt);

\draw [color=black] (1.,1.)-- ++(-2.5pt,0 pt) -- ++(5.0pt,0 pt) ++(-2.5pt,-2.5pt) -- ++(0 pt,5.0pt);
\draw [color=black] (1.,-0.)-- ++(-2.5pt,0 pt) -- ++(5.0pt,0 pt) ++(-2.5pt,-2.5pt) -- ++(0 pt,5.0pt);
\draw [color=black] (1.,-1.)-- ++(-2.5pt,0 pt) -- ++(5.0pt,0 pt) ++(-2.5pt,-2.5pt) -- ++(0 pt,5.0pt);
\draw [fill=black] (1.,-2.) circle (2.5pt);
\draw [fill=black] (1.,-4.) circle (2.5pt);
\draw [fill=black] (1.,-3.) circle (2.5pt);
\end{scriptsize}
\end{tikzpicture}
\\
$\lambda$&&${\mathcal Cor}_3(\lambda)$
\end{tabular}
\end{center}

We can obtain {${\mathcal Cor}_{3}(\lambda)=(7,5,3,2^2,1^2)$} from
$\lambda$ by pushing down beads and reading off the {resulting} bead positions. We have $\delta_0=3$ and
$\delta_1=0$, and $\lambda_{(3)}$ has three diagonal hooks. Now consider the partition ${\mathcal Q}_3(\lambda)$ of Example~\ref{ex:ex1}. More
precisely, by the previous discussion, ${\mathfrak D}(\lambda)$ can be obtained from 
the $3$-abacus of {${\mathcal Q}_3(\lambda)$} and the $\Delta_{\gamma}$. 
The $3$-abacus of $\lambda$ is obtained by shifting up
the runner 0 of {${\mathcal Q}_3(\lambda)$}
$\delta_0$ positions and by shifting down
 the runner 2 $-\delta_0$ positions. 

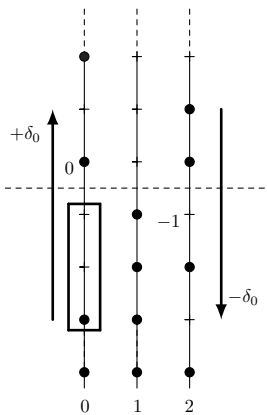
\begin{center}
\begin{tabular}{ccc}
\definecolor{sqsqsq}{rgb}{0.12549019607843137,0.12549019607843137,0.12549019607843137}
\begin{tikzpicture}[line cap=round,line join=round,>=triangle
45,x=1.0cm,y=1.0cm, scale=0.7,every node/.style={scale=0.7}]
\draw (-1.,4.)-- (-1.,-2.3);
\draw (0.,4.)-- (0.,-2.3);
\draw (1.,4.)-- (1.,-2.3);

\draw [line width=1.pt](-1.3,1.2)--(-.7,1.2);
\draw [line width=1.pt](-1.3,-1.2)--(-.7,-1.2);
\draw [line width=1.pt](-1.3,1.2)--(-1.3,-1.2);
\draw [line width=1.pt](-0.7,1.2)--(-.7,-1.2);

\draw [line width=1.pt,-latex](-1.6,-1)--(-1.6,3); 
\draw (-1.7,2.8) node[anchor=north east] {$+\delta_{0}$};
\draw [line width=1.pt,latex-](1.6,-1) --(1.6,3);
\draw (1.6,-0.3)  node[anchor=north
west] {$-\delta_{0}$}; 

\draw [dash pattern=on 2pt off 2pt](-1.,4.)-- (-1.,5.);
\draw [dash pattern=on 2pt off 2pt](0.,4.)-- (0.,5.);
\draw [dash pattern=on 2pt off 2pt](1.,4.)-- (1.,5.);
\draw [dash pattern=on 2pt off 2pt](-1.,-1.)-- (-1.,-2.);
\draw [dash pattern=on 2pt off 2pt](0.,-1.)-- (0.,-2);
\draw [dash pattern=on 2pt off 2pt](-2.5,1.5)-- (2.5,1.5);
\draw (-1.2,-2.4) node[anchor=north west] {$0$};
\draw (-0.2,-2.4) node[anchor=north west] {$1$};
\draw (0.78,-2.4) node[anchor=north west] {$2$};
\draw (-1.5,2.130963636363638) node[anchor=north west] {$0$};
\draw (0.25,1.1127818181818203) node[anchor=north west] {$-1$};
\begin{scriptsize}
\draw [color=black] (-1.,1.)-- ++(-2.5pt,0 pt) -- ++(5.0pt,0 pt) ++(-2.5pt,-2.5pt) -- ++(0 pt,5.0pt);
\draw [color=black] (-1.,3.)-- ++(-2.5pt,0 pt) -- ++(5.0pt,0 pt) ++(-2.5pt,-2.5pt) -- ++(0 pt,5.0pt);
\draw [color=black] (-1.,0.)-- ++(-2.5pt,0 pt) -- ++(5.0pt,0 pt) ++(-2.5pt,-2.5pt) -- ++(0 pt,5.0pt);
\draw [fill=sqsqsq] (-1.,4.) circle (2.5pt);
\draw [fill=black] (-1.,2.) circle (2.5pt);
\draw [fill=black] (-1.,-1.) circle (2.5pt);
\draw [fill=black] (-1.,-2.) circle (2.5pt);
\draw [color=black] (0.,4.)-- ++(-2.5pt,0 pt) -- ++(5.0pt,0 pt) ++(-2.5pt,-2.5pt) -- ++(0 pt,5.0pt);
\draw [color=black] (0.,3.)-- ++(-2.5pt,0 pt) -- ++(5.0pt,0 pt) ++(-2.5pt,-2.5pt) -- ++(0 pt,5.0pt);
\draw [color=black] (0.,2.)-- ++(-2.5pt,0 pt) -- ++(5.0pt,0 pt) ++(-2.5pt,-2.5pt) -- ++(0 pt,5.0pt);
\draw [fill=black] (0.,1.) circle (2.5pt);
\draw [fill=black] (0.,0.) circle (2.5pt);
\draw [fill=black] (0.,-1.) circle (2.5pt);
\draw [fill=black] (0.,-2.) circle (2.5pt);
\draw [color=black] (1.,4.)-- ++(-2.5pt,0 pt) -- ++(5.0pt,0 pt) ++(-2.5pt,-2.5pt) -- ++(0 pt,5.0pt);
\draw [color=black] (1.,1.)-- ++(-2.5pt,0 pt) -- ++(5.0pt,0 pt) ++(-2.5pt,-2.5pt) -- ++(0 pt,5.0pt);
\draw [color=black] (1.,-1.)-- ++(-2.5pt,0 pt) -- ++(5.0pt,0 pt) ++(-2.5pt,-2.5pt) -- ++(0 pt,5.0pt);
\draw [fill=black] (1.,3.) circle (2.5pt);
\draw [fill=black] (1.,2.) circle (2.5pt);
\draw [fill=black] (1.,0.) circle (2.5pt);
\draw [fill=black] (1.,-2.) circle (2.5pt);
\end{scriptsize}
\end{tikzpicture}
&&
\end{tabular}

\end{center}
{Consider runner $0$ of {${\mathcal Q}_3(\lambda)$}. Since $\delta_0=3$, one
shifts it up three positions to obtain the 0-runner of $\lambda$. However
(here we abuse notation) this causes $\mathcal X_{0}\cup\mathcal
Y_{0}(\lambda)$, to be altered from $\mathcal X'_{0}\cup\mathcal
Y'_{0}(\lambda)$, and hence the number of diagonal hooks of $\lambda$ arising
from runner $0$ {is different the number} of diagonal hooks of
{${\mathcal Q}_3(\lambda)$} arising from its runner $0$. In particular, the
diagonal hooks in {${\mathcal Q}_3(\lambda)$} corresponding to
positions $2$ and
$5$ on runner $2$ ``disappear'' for $\lambda$ as they shift to new
positions $-1$
and $-4$, while the bead in position $-9$ on the $3$-abacus of ${\mathcal
Q}_p(\lambda)$ introduces a new diagonal for $\lambda$ as it shifts up to
position $0$.}
 
\end{example}

Recall that the Durfee square of $\lambda$ is the
largest square that can be accommodated inside the Young diagram of
$\lambda$ (see for example \cite[\S2.3]{Andrews}). Let $\lambda^{\Box}$ be the size of
the Durfee square of $\lambda$, otherwise known as the Durfee number of $\lambda$. Let ${\mathcal
Y}_{\gamma}^1=\{-\delta_{\gamma}\geq x\geq -1\mid x\not\in\mathcal{Y}_{\gamma}\}$ and ${\mathcal
Y}_{\gamma}^0=\{-\delta_{\gamma}\geq x\geq -1 \mid x\in\mathcal{Y}_{\gamma}\}$, and $\mathcal{Y}^{\Box}_{\gamma}=|{\mathcal
Y}_{\gamma}^1|-|{\mathcal Y}_{\gamma}^0|.$ Then steps (i) through (iv) in this section describe how to calculate the size of the Durfee square
of a symmetric partition from the Durfee squares of its $p$-quotient and its $p$-core.
{\begin{lemma} \label{rk:durfee}
With the above notation, we have
 \[ \lambda^{\Box}=\sum_{\lambda_{\gamma}\in {\mathcal
Quo}_p(\lambda)}{\lambda^{\Box}_{\gamma}}+\sum_{\delta_{\gamma}>0}\mathcal{Y}^{\Box}_{\gamma}.
\] 
\end{lemma}
\begin{proof} We can rewrite the equation in the statement of the theorem as follows:
\[ \lambda^{\Box}=\\\sum_{\lambda_{\gamma}, \delta_{\gamma}>0}({\lambda^{\Box}_{\gamma}}+\mathcal{Y}^{\Box}_{\gamma})+\sum_{\lambda_{\gamma},\delta_{\gamma}=0}{\lambda^{\Box}_{\gamma}}.
\] 
The second of the two sums counts the contribution to the Durfee number
from the runners that are not affected by the introduction of a core. The
first of the two sums calculates the original contribution to the Durfee
number from the runners on which the core appears, and then corrects it
using $\mathcal{Y}^{\Box}_{\gamma}$ for each $\delta_{\gamma}>0$. In
particular, $\mathcal{Y}^{\Box}_{\gamma}$ subtracts the disappearances of
existing hooks with respect to $\gamma^*=p-\gamma-1$ from the appearances
of a new hooks with respect to $\gamma$.
\end{proof}
The following two corollaries are immediate.
\begin{corollary}
If $\lambda$ is a $p$-core, that is $\lambda={\mathcal
Cor}_{p}(\lambda)$, then $\lambda^{\Box}=\sum_{\delta_{\gamma}>0}
\delta_{\gamma}$.
\end{corollary}
\begin{proof} In this case the Durfee number is calculated directly from the the $p$-core.
\end{proof}
\begin{corollary}
If $\lambda$ has empty $p$-core, that is,
$\lambda={\mathcal Q}_{p}(\lambda)$, then
$$\lambda^{\Box}=\sum_{\lambda_{\gamma}\in {\mathcal
Quo}_p(\lambda)}{\lambda^{\Box}_{\gamma}}.$$ 
\end{corollary}}
\begin{proof} In the case the $p$-core contributes nothing, no diagonal hooks appear, non disappear, and the Durfee number of $\lambda$ is the the sum of the Durfee numbers of the quotient.
\end{proof}
\subsection{The sign of the product of the diagonal hooks.}

\begin{theorem}
\label{prop:legendre}
Let $w$ and $r$ be non-negative integers, and set $n=pw+r$. 
Let $\lambda=\lambda^*$ {be a partition} of $n$ such that
$|{\mathcal Cor}_p(\lambda)|=r$ and ${\mathcal Quo}_p(\lambda)\in
\mathcal{MP}(p,w)$, where $\mathcal{MP}(p,w)$ is the set of
$p$-multipartitions of $w$.
Assume that $\lambda_{(p-1)/2}=\emptyset$.
Set $$d=\prod_{h\in {\mathfrak d}(\lambda)}h\,,\quad q=\prod_{h\in
{\mathfrak d}({\mathcal Q}_p(\lambda))}h\quad\text{and}\quad c=\prod_{h\in
{\mathfrak d}({\mathcal Cor}_p(\lambda))}h.
$$
Then
$$
\left(\frac{p}d\right)=
\left(\frac{p}q\right)
\left(\frac{p}c\right).$$
Furthermore, if $b=\sum_{\gamma}|\mathcal B_{\gamma}|$, then
$$d\equiv qc(-1)^b\mod 4,$$
where $\mathcal B_{\gamma}$ is the set defined in Remark~\ref{rk:diagcore}.
\end{theorem}

\begin{proof}
Recall from~\S\ref{subsec:nontriv} that ${\mathfrak D}(\lambda)$ is labeled by
$\mathcal X'_{\gamma}$ and $\mathcal Y'_{\gamma}$ where $\gamma\in\Gamma$.
We choose the representative $\gamma\in\Gamma$ such that
$\delta_{\gamma}\geq 0$. We also recall that ${\mathfrak D}({\mathcal Quo}_p(\lambda))$ is labeled by
$\mathcal X_{\gamma}\cup Y_{\gamma}$ and ${\mathfrak D}({\mathcal Cor}_{p}(\lambda))$ by ${\Delta}_{\gamma}$
for $\gamma\in\Gamma$.  
Furthermore, for $\gamma\in\Gamma$, if $\delta_{\gamma}=0$, then
$\mathcal X'_{\gamma}=\mathcal X_{\gamma}$,
$\mathcal{Y}'_{\gamma}=\mathcal Y'_{\gamma}$ and ${\Delta}_{\gamma}=\emptyset$.
Otherwise, if $\delta_{\gamma}>0$, then with the 
notation~(\ref{eq:imx0}),~(\ref{eq:imx1}) and~(\ref{eq:paramdiagcore})
{$$\mathcal X'_{\gamma}=\{c(d_x)\mid x\in\mathcal X_{\gamma}\}\cup
\{c_x\mid x\in\mathcal
A_{\gamma}\}\ \text{and}\ \mathcal Y'_{\gamma}=\mathcal
\{c(d_x)\mid x\in \mathcal Y_{\gamma}\ \text{such that }
x<-\delta_{\gamma}\}.$$}
Write $$M=\prod_{\delta_{\gamma}>0}\prod_{x\in \mathcal
X'_{\gamma}\cup\mathcal Y'_{\gamma}}\left(\frac{p}{d'_x}\right),$$
where $d'_x$ is the diagonal {hook-length} of $\lambda$ corresponding to $x$. We
remark that
\begin{align*}
\left(\frac{p}d\right)&=M \prod_{\delta_{\gamma}=0}\prod_{x\in\mathcal
X_{\gamma}\cup Y_{\gamma}}\left(\frac{p}{d_x}\right).
\end{align*}
But for any $c(d_x)\in\mathcal Y'_{\gamma}$, there is $c(d_{x^*})\in
\mathcal X_{\gamma}$, where $d_x$ and $d_{x^*}$ are diagonal hook lengths of
$\lambda$ as in~(\ref{eq:dx}). 
Furthermore, by~(\ref{eq:brother}),~(\ref{eq:imx0}) and (\ref{eq:imx1}) we have 
\begin{equation}
\label{eq:ca0}
c(d_x)+c(d_x^*)= 2pw_{x,x^*}.
\end{equation}
It follows that 
\begin{equation}
\label{eq:ca1}
c(d_x)c(d_{x^*})\equiv
2c(d_x)pw_{x,x^*}-1\equiv 2pw_{x,x^*}-1\equiv d_xd_{x^*}\mod 4.
\end{equation}
Hence, 
\begin{align}
\label{eq:ca2}
\left(\frac{p}{c(d_x)
c(d_{x^*})}\right)&=(-1)^{\frac{(p-1)(c(d_x)c(d_{x^*})-1)}4} 
\left(\frac{c(d_x) c(d_{x^*})}p\right)\\
\nonumber &=(-1)^{\frac{(p-1)(d_x d_{x^*}-1)}4} 
\left(\frac{-1}p\right)\\
\nonumber &=
\left(\frac{p}{d_x
d_{x^*}}\right).
\end{align}
\medskip

On the other hand, if $x\in \mathcal
Y_{\gamma}$ {is} such that $-\delta_{\gamma}\leq x\leq -1$, that is
$x\in\mathcal B_{\gamma}$, then there is a
diagonal {hook} of $\mu$ of {length} $c(d_{x^*})$ with $x^*\in\mathcal
X_{\gamma}$. {So}
\begin{align*}
M&=\prod_{\delta_{\gamma}>0}\left(\prod_{x\in \mathcal Y_{\gamma}}
\left(\frac{p}{c(d_x)
c(d_{x^*})}\right)\prod_{x\in\mathcal B_{\gamma}}
\left(\frac{p}{
c(d_{x^*})}\right)\prod_{x\in\mathcal A_{\gamma}}
\left(\frac{p}{
c_x}\right)
\right)\\
&=\prod_{\delta_{\gamma}>0}\left(\prod_{x\in \mathcal Y_{\gamma}}
\left(\frac{p}{d_x
d_{x^*}}\right)\prod_{x\in\mathcal B_{\gamma}}
\left(\frac{p}{
c(d_{x^*})}\right)\prod_{x\in\mathcal A_{\gamma}}
\left(\frac{p}{
c_x}\right)
\right).
\end{align*}

By Remark~\ref{rk:diagcore}, recall that {${\mathfrak d}({\mathcal
Cor}_p(\lambda))=\{c_x\mid x\in \mathcal
A_{\gamma}\cup\mathcal B_{\gamma}\}$}, where $c_x$ is given
in~(\ref{eq:paramdiagcore}). Then
$$
\left(\frac{p}d\right)=\left(\frac{p}q\right)\left(\frac{p}c\right)
\prod_{\delta_{\gamma}>0}
\prod_{x\in\mathcal B_{\gamma}}
\left(\frac{p}{
d_xd_{x^*}c(d_{x^*})c_x}\right).
$$
Let $\gamma$ be such that $\delta_{\gamma}>0$ and $x\in\mathcal
B_{\gamma}$. By~(\ref{eq:imx0}) and~(\ref{eq:paramdiagcore}), we have
$c_x\equiv c(d_{x^*})\mod p$. Moreover
\begin{align}
\label{eq:ca3}
c_xc(d_{x^*})&\equiv
1+2((\delta_{\gamma}+x)p+1)+2((\delta_{\gamma}+x^*)p+1)\mod 4\\
\nonumber&\equiv 1+2x+2x^*\mod 4\\
\nonumber&\equiv 1+2(x+x^*)\mod 4\\
\nonumber&=1+2(x-x^*)\mod 4\\
\nonumber&=1+2w_{x,x^*}\mod 4.
\end{align}
Hence
$$\frac{c_xc(d_{x^*})-1}2\equiv w_{x,x^*}\mod 2,$$
and we obtain that
$$\left(\frac{p}{
c(d_{x^*})c_x}\right)=(-1)^{\frac{(p-1)w_{x,x^*}}2}
\left(\frac{
c(d_{x^*})c_x}p\right)=(-1)^{\frac{(p-1)w_{x,x^*}}2}.
$$
However, the computation~(\ref{eq:calculsymbol}) shows that 
$\left(\frac{p}{
d_{x}d_{x^*}}\right)=(-1)^{\frac{(p-1)w_{x,x^*}}2}$, and
$$
\left(\frac{p}{
d_xd_{x^*}c(d_{x^*})c_x}\right)=(-1)^{\frac{(p-1)w_{x,x^*}}2}\cdot
(-1)^{\frac{(p-1)w_{x,x^*}}2}=1.$$
The result follows.

We now prove the second part of the statement. Since an odd number is
its own inverse modulo
$4$, we do the same computation as above and obtain that
$$d\equiv qc\prod_{\delta_{\gamma}>0}
\prod_{x\in\mathcal B_{\gamma}}
d_xd_{x^*}c(d_{x^*})c_x\mod 4.
$$
But by~(\ref{eq:ca1}) and~(\ref{eq:ca3}), we have
$$d_xd_{x^*}c(d_{x^*})c_x\equiv (2w_{x,x'}-1)(2w_{x,x'}+1)\equiv
4w_{x,x^*}^2-1\equiv -1\mod 4.$$
Thus, $d\equiv qc(-1)^b\mod 4$.
\end{proof}

\begin{theorem}
\label{prop:irrcorequotient}
Let $\lambda=\lambda^*$.Then for any $f\in\mathcal H_{n!/2}$,
$$\varepsilon(\chi_{\lambda},f)=
\varepsilon(\chi_{{\mathcal Q}_p(\lambda)},f)\varepsilon(\chi_{{\mathcal Cor}_p(\lambda)},f).$$
\end{theorem}

\begin{proof}
Write $n=pw+r$, where $r=|{\mathcal Cor}_p(\lambda)|$.

First, we assume that $\lambda_{(p-1)/2}=\emptyset$. We define
$\lambda$, $d$, $q$ and $c$ as in Theorem~\ref{prop:legendre}. Since
$\lambda_{(p-1)/2}=\emptyset$, we have
{$\varepsilon(\chi_{\lambda},f)=\varepsilon(\chi_{{\mathcal
Q}_p(\lambda)},f)=\varepsilon(\chi_{{\mathcal Cor}_p(\lambda)},f)=1$}
for all $f\in\mathcal K_{n!/2}$.
We consider the case $f=\sigma_{n!/2}$. In the proof
of~\ref{prop:bougereg}, we see that
$\varepsilon(\chi_{\lambda},\sigma_{n!/2})=\left(\frac{p}{
q}\right)$.
To simplify the notation, set $m=d_{{\mathcal Cor}_p(\lambda)}$. 
By Theorem~\ref{prop:legendre}, $d\equiv (-1)^bqc\mod 4$. 
Furthermore, 
\begin{align*}
d_{\mu}+d_{\lambda}&=\sum_{\gamma\in\Gamma}(|\mathcal
X_{\gamma}|+|\mathcal X_{\gamma}'|+|\mathcal Y_{\gamma}|+|\mathcal
Y'_{\gamma}|)\\
&=\sum_{\gamma\in\Gamma}(2|\mathcal X_{\gamma}|+|\mathcal
A_{\gamma}|+2|\mathcal Y_{\gamma}|-|\mathcal B_{\gamma}|)\\
&=2\sum_{\gamma\in\Gamma}(\underbrace{|\mathcal X_{\gamma})|+|\mathcal
Y_{\gamma}|}_{\text{even}})+m-2b\\
&\equiv m+2b\mod 4.
\end{align*} 
Now, we derive from the proof of Proposition~\ref{prop:bougereg}
$(-1)^{\frac{n-r+d_{\lambda}}{2}}=(-1)^{\frac{n-r-d_{\lambda}}{2}}=
(-1)^{\frac{pw-d_{\lambda}}{2}}=(-1)^{\frac{q-1}{2}}$. In particular,
$n-r+d_{\lambda}\equiv q-1\mod 4$. Thus,
\begin{align*}
n-d_{\mu}+d-1-r+m-c+1&=
n-r+m-d_{\mu}+d-c\\
&\equiv n-r+d_{\lambda}+2b+qc(-1)^b-c\mod 4\\
&\equiv q-1+2b+qc(-1)^b-c\mod 4.
\end{align*}  
If $b$ is even, then $$n-d_{\mu}+d-1-r+m-c+1\equiv
q-1+qc-c\equiv(q-1)(c+1)\equiv 0\mod 4,$$ because $q$ and $c$ are odd.
If $b$ is odd, then 
\begin{align*}
n-d_{\mu}+d-1-r+m-c+1&\equiv
q-1 +2-qc-c\mod 4\\
&\equiv (q+1)(1-c)\mod 4\\
&\equiv 0\mod 4.
\end{align*}
Finally, using Propositions~\ref{prop:bougereg} and~\ref{prop:racbouge},
and~(\ref{eq:bougealt}), we obtain
\begin{align*}
\varepsilon(\chi_{\lambda},\sigma_{n!/2})&=(-1)^{\frac{p-1}4(n-d_{\mu}+d-1)}
\left(\frac{p}{
d}\right)\\
&=(-1)^{\frac{p-1}4(n-d_{\mu}+d-1)}\left(\frac{p}{
q}\right)\left(\frac{p}{
c}\right)\\
&=(-1)^{\frac{p-1}4(n-d_{\mu}+d-1-r+m-c+1)}\varepsilon(\chi_{{\mathcal Q}_p(\lambda)},\sigma_{n!/2})
\varepsilon(\chi_{{\mathcal Cor}_p(\lambda)},\sigma_{n!/2})\\
&=\varepsilon(\chi_{{\mathcal Q}_p(\lambda)},\sigma_{n!/2})
\varepsilon(\chi_{{\mathcal Cor}_p(\lambda)},\sigma_{n!/2}).
\end{align*}

Assume now that $\lambda_{(p-1)/2}$ is non-empty. Since
$\delta_{(p-1)/2}=\emptyset$, we have $\mathcal X'_{(p-1)/2}=\mathcal
X_{(p-1)/2}$, that is, the diagonal hooks arising from the
$(p-1)/2$-runner of $\lambda$ and ${\mathcal Q}_p(\lambda)$ are the
same. Denote by $\lambda^{\vee}$ 
the partition with same {$p$-core} and $p$-quotient as $\lambda$ except
$\lambda^{\vee}_{(p-1)/2}=\emptyset$.
Then
\begin{align*}
\varepsilon(\chi_{\lambda},f)&=\varepsilon(i,f)^{\frac{p|\lambda_{(p-1)/2}|-|\mathcal
X_{(p-1)/2}|+d_{(p-1)/2}-1}2}\varepsilon(\sqrt{d_{(p-1)/2}},f)\varepsilon(\chi_{{\lambda^{\vee}}},f)\\
&=\varepsilon(i,f)^{\frac{p|\lambda_{(p-1)/2}|-|\mathcal
X_{(p-1)/2}|+d_{(p-1)/2}-1}2}\varepsilon(\sqrt{d_{(p-1)/2}},f)\varepsilon(\chi_{{\mathcal{Q}_p(\lambda)^{\vee}}},f)\varepsilon(\chi_{{\mathcal Cor}_p(\lambda)},f)\\
&=\varepsilon(\chi_{{\mathcal Q}_p(\lambda)},f)\varepsilon(\chi_{({\mathcal Cor}_p(\lambda)},f),
\end{align*}
where $d_{(p-1)/2}$ is the product of the diagonal hook lengths
arising
from the runner $(p-1)/2$.
\end{proof}
\section{Verification of Navarro's conjecture for the alternating groups}
\label{sec:preuve}

We will now prove Theorem~\ref{thm:main}. 
Let $n$ be a positive integer with $p$-adic expansion
$n=n_0+pn_1+\cdots+n_sp^s$.
Let $\lambda$ be a partition 
of $n$ with $p$-core tower
{${\mathcal Cor}^{(k)}_p(\lambda)=\{\lambda_{\underline j}\mid \underline j\in
I^k\}$} for $k\geq 0$ such that $c_k(\lambda)=\sum_{\underline j\in
I^k}|\lambda_{\underline j}|$.
We then associate to $\lambda$ the irreducible character of
$\operatorname{N}_{\sym_n}(P)$ 
$$\psi_{\lambda}=\prod_{k\geq 0}\psi_{\underline\lambda,k},$$
as above, where $\psi_{\underline\lambda,k}\in \Irr(N_k)$ as in~(\ref{eq:psiNk}).
If $\lambda$ is not symmetric, then $\chi_{\lambda}$ and
{$\psi_{\underline\lambda}$} {restrict} irreducibly to $\Alt_n$ and
$\operatorname{N}_{\Alt_n}(P)$. As above, we denote the restriction by
the same symbol.
If $\lambda$ is symmetric, then the restriction of $\chi_{\lambda}$ to
$\Alt_n$ has two
irreducible {constituents} $\chi_{\lambda}^+$ and $\chi_{\lambda}^-$.
Similarly for {$\psi_{\lambda}$}. More precisely, 
for any $k\geq 0$ and 
$\underline\lambda\in\mathcal{MP}(p^k,n_k)$, we {have}
$\underline\lambda^*=\underline\lambda$ and the restriction
of $\psi_{\underline\lambda,k}$ to $(Y_k\wr\sym_{n_k})^+$
splits into two irreducible characters $\psi_{\underline\lambda,k}^+$
and $\psi_{\underline\lambda,k}^-$. Then following \S\ref{subsec:redpb},
{we label $\psi_{\lambda}^{+}$ such that 
$\prod_k\psi_{\underline\lambda,k}^+$ is a constituent of
$\Res_{\prod(Y^k\wr\sym_{n_k})^+}(\psi_{\underline\lambda}^+)$.
In particular, $\prod_k\psi_{\underline\lambda,k}^-$ is a constituent of
$\Res_{\prod(Y^k\wr\sym_{n_k})^+}(\psi_{\underline\lambda}^-)$}.
Now, define
$\Phi:\Irr_{p'}(\Alt_n)\rightarrow\Irr_{p'}(\operatorname{N}_{\Alt_n}(P))$
by setting
\begin{equation}
\label{eq:defphi}
\Phi(\chi_{\lambda})={\psi_{\lambda}}\text{ if
}\lambda\neq\lambda^*,\quad \text{and}\quad
\Phi(\chi_{\lambda}^{\pm})={\psi_{\lambda}^{\pm}}\ \text{otherwise}.
\end{equation}

We need the following two lemmas.
\begin{lemma}
\label{lemma:regcase}
If $\lambda$ is a regular partition of $n$ and $f\in\mathcal H_{n!/2}$,
then
$$\varepsilon(\chi_{\lambda},f)=\varepsilon({\psi_{\lambda}},f).$$ 
\end{lemma}

\begin{proof}
Since $\lambda$ is regular, $n=n_1p+n_2p^2+\cdots+n_sp^s$ with $n_i$ even for
all $1\leq i\leq s$. {By Proposition~\ref{prop:reducbouge}}
$$
\varepsilon({\psi_{\lambda}},f)=\prod_{k=1}^s\varepsilon(\psi_{\underline\lambda,k},f).
$$
Hence, for $f\in\mathcal K_{n!/2}$, one has
$\varepsilon(\psi_{\lambda},f)=1$ by
Proposition~\ref{prop:signeregular}, and
{\begin{align*}
\varepsilon({\psi_{\lambda}},\sigma_{n!/2})&=\prod_{k=1}^s\varepsilon(\psi_{\underline\lambda,k},\sigma_{n!/2})\\
&=\prod_{k=1}^s(-1)^{\frac{(p-1)n_k}{2}}\\
&=(-1)^{\frac{p-1}4\sum_{k=1}^sn_k}.
\end{align*}}
For $1\leq k\leq s$, let $n'_k\in\Z$ such that $n_k=2n'_k$. We have
$$\frac n 2=\sum_{k=1}^sn'_kp^k\equiv\sum_{k=1}^s n'_k \mod 2,$$
because $p$ is odd. Thus $\sum_{k=1}^sn_k\equiv n\mod 4$, and
{$\varepsilon({\psi_{\lambda}},\sigma_{n!/2})=(-1)^{\frac{(p-1)n}4}$}. The result
now follows from Proposition~\ref{prop:bougereg}.
\end{proof}

\begin{lemma}
\label{lemma:singcase}
If $\lambda$ is a singular partition of $n$ with empty {$p$-core},
 and $f\in\mathcal H_{n!/2}$,
then
$$\varepsilon(\chi_{\lambda},f)=\varepsilon({\psi_{\lambda}},f).$$ 
\end{lemma}

\begin{proof}
By construction of $\lambda$ from its $p$-core tower and
\S\ref{subsec:nontriv}, 
for all $k\geq 1$, we have
{$${\mathfrak d}_{p^k}(\lambda)=\{p^kh\mid h\in{\mathfrak
d}_{\underline{\lambda},k}\},$$}
where ${\mathfrak d}_{p^k}(\lambda)$ is the set of diagonal {hooklengths} of $\chi_{\lambda}$ {divisible} by $p^k$ {but} not by $p^{k+1}$
and ${\mathfrak d}_{\underline{\lambda},k}$ is the set of diagonal
{hooklengths}
of $\chi_{\underline p^*(k)}$ with $\underline p^*(k)\in I^k$. In the
following, we write  $\chi_k=\chi_{\underline p^*(k)}$. In particular, if
$d_{\lambda}$ and $d_k$ are the number of diagonal hooks of $\lambda$
and the partition with empty {$p$-core tower} except the position $\underline
p^*(k)$ in the level $k$, that is equal to
$\lambda_{\underline p^*(k)}$, then
$d_{\lambda}=\sum_{k=1}^s d_k$.

Let $f\in\mathcal H_{n!/2}$. 
By~(\ref{eq:bougealt}), we obtain
\begin{align*}
\varepsilon(\chi_{\lambda},f)&=
\varepsilon(i,f)^{(n-d_{\lambda})/2}\varepsilon\left(\sqrt{\prod_{h\in
{\mathfrak d}(\lambda)}h},f\right)\\
&=\prod_{k=1}^s\varepsilon(i,f)^{\frac{n_kp^k-d_k}2}\varepsilon
\left(\sqrt{\prod_{h\in
{\mathfrak d}_{p^k}(\lambda)}h},f\right)\\
&=\prod_{k=1}^s\varepsilon(i,f)^{\frac{n_k(-1)^k-d_k}2}\varepsilon
\left(\sqrt{p}^{kd_k}\sqrt{\prod_{h\in
{\mathfrak d}_{\underline\lambda,k}(\lambda)}h},f\right)\\
&=\prod_{k=1}^s\varepsilon(i,f)^{\frac{n_k(-1)^k-d_k+n_k-d_k}{2}}
\varepsilon(\sqrt{p}^{kd_k},f)
\varepsilon(\chi_k,f)\\
&=\prod_{k=1}^s\varepsilon(i,f)^{\frac{n_k(-1)^k-d_k+n_k-d_k+2d_kk}{2}}
\varepsilon(\psi_{\underline\lambda,k},f)\quad\text{by
Prop.~\ref{prop:signesingular}}\\
&=\varepsilon(i,f)^{\frac 1
2\sum_{k=1}^s(n_k(-1)^k-d_k+n_k-d_k+2d_kk)}
\varepsilon({\psi_{\lambda}},f),
\end{align*}
where the last equality comes from 
Proposition~\ref{prop:reducbouge}. However, if $k$ is even, then
$$n_k(-1)^k-d_k+n_k-d_k+2d_kk\equiv 2(n_k-d_k)\equiv 0\mod 4$$ because 
$n_k$ and $d_k$ have the same parity.
If $k$ is odd, then
$$n_k(-1)^k-d_k+n_k-d_k+2d_kk\equiv -n_k+2d_k+n_k-2d_k\equiv 0\mod 4.$$
The result follows.
\end{proof}


\begin{lemma} 
\label{lemma:rs}
Let $\lambda$ be a symmetric partition with empty $p$-core.
Then for $f\in\mathcal H_{n!/2}$
$$\varepsilon(\chi_{\lambda},f)=\varepsilon(\chi_{{\mathfrak
r}(\lambda)},f)\varepsilon(\chi_{{\mathfrak s}(\lambda)},f),$$
where ${\mathfrak
r}(\lambda)$ and ${\mathfrak s}(\lambda)$ are the
regular and the singular parts of $\lambda$ as 
in~(\ref{eq:ctlambdaprimeseconde}).
\end{lemma}

\begin{proof}
By assumption, the $p$-core of $\lambda$ is empty. In particular,
$$\mathcal D(\lambda)=\mathcal D(\mathfrak r(\lambda))\sqcup \mathcal
D(\mathfrak s(\lambda)).$$
Write $d_{\lambda}=|\mathcal D(\lambda)|$, $d_{\mathfrak
r(\lambda)}=|\mathcal D(\mathfrak r(\lambda))|$ and $d_{\mathfrak
s(\lambda)}=|\mathcal D(\mathfrak s(\lambda))|$. We have 
$$d_{\lambda}=d_{\mathfrak r(\lambda)}+d_{\mathfrak s(\lambda)}.$$
Furthermore, we have $|\lambda|=|\mathfrak r(\lambda)|+|\mathfrak
s(\lambda)|$ by construction. Hence, for all $f\in\mathcal H_{n!}$,
Equation~(\ref{eq:bougealt}) gives
\begin{align*}
\varepsilon(\chi_{\lambda},f)&=\varepsilon\left(i^{(|\lambda|-d_{\lambda})/2}\sqrt{\prod_{h\in\mathcal
D(\lambda)}h},f\right)\\
&=\varepsilon\left(i^{(|\mathfrak r(\lambda)|-d_{\mathfrak r(\lambda)})/2}\sqrt{\prod_{h\in\mathcal
D(\mathfrak r(\lambda))}h},f\right)
\varepsilon\left(i^{(|\mathfrak s(\lambda)|-d_{\mathfrak s(\lambda)})/2}\sqrt{\prod_{h\in\mathcal
D(\mathfrak s(\lambda))}h},f\right)\\
&=\varepsilon(\chi_{{\mathfrak
r}(\lambda)},f)\varepsilon(\chi_{{\mathfrak s}(\lambda)},f),
\end{align*}
as required.
\end{proof}

Assume that $\lambda=\lambda^*$. Recall ${\mathcal Q}_p(\lambda)$ is the partition
with the same $p$-quotient as $\lambda$ and with empty $p$-core.
Proposition~\ref{prop:irrcorequotient} and Lemma~\ref{lemma:rs} give
\begin{equation}
\label{eq:main1}
\varepsilon(\chi_{\lambda},f)=\varepsilon(\chi_{{\mathcal Cor}_p(\lambda)},f)
\varepsilon(\chi_{{\mathfrak
r}({\mathcal Q}_p(\lambda))},f)\varepsilon(\chi_{{\mathfrak s}{\mathcal Q}_p(\lambda))},f).
\end{equation}
Now, by Theorem~\ref{thm:irrationaliteNk} and
Proposition~\ref{prop:reducbouge}, we have
\begin{align}
\label{eq:main2}
\varepsilon({\psi_{\lambda}},f)&=\varepsilon(\chi_{{\mathcal Cor}_p(\lambda)},f)\prod_{k=1}^s\varepsilon(\psi_{\underline{{\mathfrak
r}({\mathcal Q}_p(\lambda))},k},f)\prod_{k=1}^s\varepsilon(\psi_{\underline{{\mathfrak
s}({\mathcal Q}_p(\lambda))},k},f)\\
\nonumber&=\varepsilon(\chi_{{\mathcal Cor}_p(\lambda)},f)\varepsilon(\psi_{\mathfrak
r({\mathcal Q}_p(\lambda))},f)\varepsilon(\psi_{\mathfrak
s({\mathcal Q}_p(\lambda))},f).
\end{align}
However, by Lemmas~\ref{lemma:regcase} and~{\ref{lemma:singcase}},
we have $$\varepsilon(\chi_{\mathfrak
r({\mathcal Q}_p(\lambda))},f)=\varepsilon(\psi_{\mathfrak
r({\mathcal
Q}_p(\lambda))},f)\quad\text{and}\quad\varepsilon(\chi_{\mathfrak
s({\mathcal Q}_p(\lambda))},f)=\varepsilon(\psi_{\mathfrak
s({\mathcal Q}_p(\lambda))},f).$$ Finally (\ref{eq:main1}) and (\ref{eq:main2}) give
that
$$\varepsilon(\chi_{\lambda},f)=\varepsilon({\psi_{\lambda}},f).$$
Hence, $\Phi$ is an $\mathcal H_{n!/2}$-equivariant bijection,
as required.

\section{Blockwise Navarro's conjecture for alternating
groups}

For any finite group $G$ and any prime number $p$ dividing $|G|$, recall
that $\Irr(G)$ decomposes into families, the so-called $p$-blocks
of $G$. Write $\operatorname{Bl}(G)$ for the set of $p$-blocks of $G$.
Furthermore, we attach to any $B\in\operatorname{Bl}(G)$ its \emph{$p$-defect
group} $D$. This is 
a $p$-subgroup of $G$ which is well-defined up to conjugation.
Now, by Brauer's {first} main theorem~\cite[(15.45)]{isaacs}, we
can associate to any $p$-block $B$ of $G$ its Brauer correspondent
$B'\in\operatorname{Bl}(\operatorname{N}_G(D))$. Then the \emph{blockwise
Navarro's conjecture} asserts that the 
number of height zero characters in $B$ and $B'$ fixed by
$\sigma\in\mathcal H_{n}$ is the same.

\subsection{Case of $p$ odd}

In order to discuss blockwise Navarro's conjecture for alternating
groups, we will first recall some some facts about the $p$-blocks of 
symmetric and alternating groups.\medskip

It is well-know by the \emph{Nakayama Conjecture} that for any prime
$p$, the $p$-blocks of $\sym_n$ are labeled by the $p$-cores of
partitions of $n$. More {precisely}, two irreducible characters of
$\sym_n$ lie in the same $p$-block if and only if the partitions
labeling them have the same $p$-core; see for
example~\cite[Theorem 11.1]{Olsson}. In the following, such a $p$-core
will be called a \emph{$p$-core of $n$}. Note that there is here an
abuse of {terminology} since a $p$-core of $n$ is not in general a
partition of $n$.
For a $p$-core $\gamma$ of $n$, we denote by $B_{\gamma}$ the
corresponding $p$-block of $\sym_n$, and we define
the $p$-weight of $B_{\gamma}$ by setting
$w=\frac{n-|\gamma|}p$. 

We can describe the height zero characters of
$B_{\gamma}$ in term of the $p$-core tower of partitions labeling
characters of the block as follows. By~\cite[Proposition 11.5]{Olsson},
an irreducible character $\chi_{\lambda}$ lying in the block $B_{\gamma}$
has height zero
if and only if 
$0\leq c_k(\lambda)\leq p-1$ for all $k\geq 1$ with
$c_k(\lambda)=\sum_{\underline j\in I^k}|\lambda_{\underline j}|$, where 
the Notation is as in~(\ref{eq:coretower}). 

Furthermore, without loss of
generality, we can assume by~\cite[Proposition 11.3]{Olsson} that
any Sylow $p$-subgroup $D_{\gamma}$ of $\sym_{pw}\subseteq\sym_n$ is a
defect group of $B_{\gamma}$. Let
$pw=w_1p+w_2p^2+\cdots$ denote the $p$-adic expansion of $pw$. Then
by~\cite[page 159]{Fong}, we have
$$\operatorname{N}_{\sym_n}(D_{\gamma})/D_{\gamma}'\simeq\sym_{|\gamma|}\times\prod_{k\geq
1}Y^k\wr\sym_{w_k}.$$
Moreover, by~\cite[page 158 and 159]{Fong}, 
the set $\Irr_0(B'_{\gamma})$ of height zero characters of
the Brauer correspondent
$B'_{\gamma}\in\operatorname{N}_{\sym_n}(D_{\gamma})$ of $B_{\gamma}$ 
is
$$\Irr_0(B'_{\gamma})=\left\{\chi_{\gamma}\otimes\prod_{k\geq
1}\psi_{\underline\lambda,k}\;\big| \;
\chi_{\gamma}\in\Irr(\sym_{|\gamma|});\,
\psi_{\underline\lambda,k}\in \Irr(Y^k\wr\sym_{w_k})\right\}.$$

From now on, assume $p$ is odd.
Note that $B_{\gamma^*}=\{\chi_{\lambda^*}\in\Irr(\sym_n)\mid
{\mathcal Cor}_{p}(\lambda)=\gamma\}=B_{\gamma}^*$. In particular, if
$\gamma\neq\gamma^*$, then $B_{\gamma}\cap B_{\gamma^*}=\emptyset$ and
$B_{\gamma}$ contains no self-conjugate character. {Then}
\cite[(9.2)]{Navarro} implies that the two
$p$-blocks $B_{\gamma}$ and $B_{\gamma^*}$ cover a unique $p$-block
$b_{\gamma}$ of $\Alt_n$ (Note that $b_{\gamma}=b_{\gamma^*}$).
Furthermore, if $\gamma=\gamma^*$ and $B_{\gamma}$ {has non-zero} defect,
then there is {an} irreducible character 
$\chi_{\lambda}\in B_{\gamma}$ with $\lambda\neq\lambda^*$ and 
\cite[(9.2)]{Navarro} implies that $B_{\gamma}$ again covers a unique
$p$-block $b_{\gamma}$ of $\Alt_n$. 
Finally, for $n\geq 3$, 
if $B_{\gamma}$ has defect zero and $\gamma=\gamma^*$, then
$\{\chi_{\gamma}^+\}$ and $\{\chi_{\gamma}^-\}$ are two $p$-blocks of
$\Alt_n$ of defect zero. {These two blocks are equal to their Brauer
correspondent, and the blockwise Navarro's conjecture is then trivial in
this case}. 

We remark that $D_{\gamma}$ is a defect group of $B_{\gamma}$ since $p$
is odd, and
$\operatorname{N}_{\Alt_n}(D_{\gamma})=\operatorname{N}_{\sym_n}(D_{\gamma})^+$.
Assume that $B_{\gamma}$ has a non-zero defect. Then $B'_{\gamma}$ covers
a unique $p$-block of $\operatorname{N}_{\sym_n}(D_{\gamma})^+$. 
Indeed, if $\gamma\neq\gamma^*$
then the restrictions to $\operatorname{N}_{\sym_n}(D_{\gamma})^+$ of
the characters of $B'_{\gamma}$ form a $p$-block
$b'_{\gamma}(=b'_{\gamma^*})$ of
$\operatorname{N}_{\sym_n}(D_{\gamma})^+$ covered by $B'_{\gamma}$ and
$B'_{\gamma^*}$ by~\cite[(9.2)]{Navarro}, and if $\gamma=\gamma^*$, then $B'_{\gamma}$ has a
self-conjugate character (since the block has a non-zero defect) and
$B'_{\gamma}$ covers a unique $p$-block $b'_{\gamma}$ of
$\operatorname{N}_{\sym_n}(D_{\gamma})^+$ by~\cite[(9.2)]{Navarro}.   
Furthermore, by unicity of the covered block, 
$b_{\gamma}'$ 
is the Brauer correspondent of 
$b_{\gamma}$ by~\cite[(9.28)]{Navarro}.
Therefore, the height zero characters of this block are identified (by
lifting) with the set of irreductible characters of
$$\operatorname{N}_{\Alt_n}(D_{\gamma})/D_{\gamma}\simeq(\sym_{\gamma}\times\prod_{k\geq
1}Y^k\wr\sym_{w_k})^+.$$ 
Let $\lambda$ be a partition of $n$ with $p$-core $\gamma$ and with
height zero. Write
$\mathcal{CT}(\lambda)$ for the $p$-core tower of $\lambda$ with the
Notation as in~(\ref{eq:coretower}). In particular,
$\lambda_{\emptyset}=\gamma$.
Write $$\psi_{\lambda}=\chi_{\gamma}\otimes\prod_{k\geq
1}\psi_{\underline\lambda,k}$$ for the irreducible character of
$B_{\gamma}'$ labeled by $\lambda$ (which is well-defined since
$\chi_{\lambda}\in B_{\gamma}$ is of height zero).
Then $\psi_{\lambda}$ splits into one or two constituents of
$b'_{\gamma}$ whenever $\lambda\neq \lambda^*$ or $\lambda=\lambda^*$.
We again write $\psi_{\lambda}$ for the irreducible restriction in the
first case, and we write $\psi_{\lambda}^{\pm}$ for the two irreducible
constituents otherwise. 

\begin{theorem}Let $p$ be an odd prime. Let $\gamma$ be a $p$-core of $n$. 
We assume $w>0$. For a partition $\lambda$ of $n$ with $p$-core
$\gamma$, define
$\Phi:\Irr_0(b_{\gamma})\rightarrow \Irr_0(b_{\gamma}')$  by setting
$$\Phi(\chi_{\lambda})=\psi_{\lambda}\quad\text{if
}\lambda\neq\lambda^{*} \quad \text{and}
\quad \Phi(\chi_{\lambda}^{\pm})=\psi_{\lambda}^{\pm}\quad\text{if
}\lambda = \lambda^*.$$
Then $\Phi$ is a $\mathcal H_{n!/2}$-equivariant bijection. In
particular, blockwise Navarro's conjecture holds for the $p$-blocks of
alternating groups.
\end{theorem}

\begin{proof}
First, we remark that the map is well-defined.
We only have to consider the case of an irreducible character
$\chi_{\lambda}^{\pm}\in b_{\gamma}$ for $\lambda=\lambda^*$. 
In particular, $\mathcal Cor_p(\lambda)=\gamma$, and by
Equation~(\ref{eq:main1}) we have for any $f\in\mathcal H_{n!/2}$
$$\varepsilon(\chi_{\lambda},f)=\varepsilon(\chi_{\gamma},f)
\varepsilon(\chi_{\mathfrak
r(\mathcal Q_p(\lambda)),f})
\varepsilon(\chi_{\mathfrak
s(\mathcal Q_p(\lambda)),f}).$$
Now, applying the resuts of Section~\ref{sec:local} to $pw$ with the group
$(\prod_{k\geq
1}Y^k\wr\sym_{w_k})^+$, Proposition~\ref{prop:reducbouge} and
Theorem~\ref{thm:irrationaliteNk} give
$$\varepsilon(\prod_{k\geq 1}\psi_{\lambda,k},f)=\varepsilon(\psi_{\mathfrak
r(\mathcal Q_p(\lambda))},f)\varepsilon(\psi_{\mathfrak
r(\mathcal Q_p(\lambda))},f).$$
Again using Proposition~\ref{prop:reducbouge}, we obtain
$$\varepsilon(\psi_{\lambda},f)=\varepsilon(\chi_{\gamma},f)
\varepsilon(\psi_{\mathfrak
r(\mathcal Q_p(\lambda))},f)\varepsilon(\psi_{\mathfrak
r(\mathcal Q_p(\lambda))},f),$$
and we conclude by Lemmas~\ref{lemma:regcase} and~\ref{lemma:singcase}.
\end{proof}

\subsection{Case of $p=2$}

First, we will prove that an analogue of Theorem~\ref{prop:legendre} holds for
$p=2$.

\begin{theorem}
\label{prop:legendre2}
Assume $p=2$. We have
$$
\left(\frac{2}d\right)=
\left(\frac{2}q\right)
\left(\frac{2}c\right),$$
where $c$, $d$ and $q$ are as in Theorem~\ref{prop:legendre}.
\end{theorem}

\begin{proof}
Let $\gamma\in\{0,1\}$ be such that $\delta_{\gamma}\geq 0$.
We write $\mathcal X'_{\gamma}$, $\mathcal Y'_{\gamma}$, $\mathcal
X_{\gamma}$, $\mathcal Y_{\gamma}$ and $\Delta_{\gamma}$ for the sets
labeling ${\mathfrak D}(\lambda)$, ${\mathfrak D}({\mathcal
Quo}_p(\lambda))$, and ${\mathfrak D}({\mathcal Cor}_{p}(\lambda))$,
respectively. If $\delta_{\gamma}=0$, then the statement is
trivial. Assume $\delta_{\gamma}>0$. 
We also consider the set $\mathcal
B_{\gamma}$ as in Remark~\ref{rk:diagcore}.
Let $x\in \mathcal X_{\gamma}$ and $\varepsilon\in\{1,3\}$ be such that
$d_x=4x+\varepsilon$. Hence, $d_{x^*}=4\phi(x^*)+\varepsilon'$, where
$\varepsilon'=4-\varepsilon$. Furthermore, with the notation 
of~(\ref{eq:imx0}),
$$c(d_x)=4(x+\delta_{\gamma})+\varepsilon\quad\text{and}\quad
c(d_{x^*})=4(\phi(x^*)-\delta_{\gamma})+\varepsilon',$$
where $c(d_{x^*})$ ``exists'' if and only if $\phi(x^*)\geq \delta_{\gamma}$. 
Assume $\phi(x^*)\geq \delta_{\gamma}$. Then
$$c(d_x)c(d_{x^*})=d_xd_{x^*}+4\delta_{\gamma}(d_{x^*}-d_x)-16d^2\equiv
d_xd_{x^*}+4\delta_{\gamma}(d_{x^*}-d_x)\mod 16.$$
Since $d_{x^*}-d_x$ is even, we obtain
$(c(d_x)c(d_{x^*})^2\equiv (d_xd_{x^*})^2\mod 16$. Hence,
$(c(d_x)c(d_{x^*})^2-1)/8-((d_x d_{x^*})^2-1)/8$ is even, whence

\begin{equation}
\label{eq:fondamental2}
\left(\frac{2}{c(d_x)c(d_{x^*})}\right)=
(-1)^{\frac{(c(d_x)c(d_{x^*}))^2-1}{8}}=
(-1)^{\frac{(d_x
d_{x^*})^2-1}8}=\left(\frac{2}{d_xd_{x^*}}\right).
\end{equation}

Assume now that $0\leq \phi(x^*)\leq \delta_{\gamma}-1$. In particular,
$x^*\in\mathcal B_{\gamma}$, and 
$$c_{x^*}=4(\delta_{\gamma}-1-\phi(x^*))+\varepsilon=4\delta_{\gamma}-4-d_{x^*}+
\underbrace{\varepsilon
+\varepsilon'}_{=4}=4d-d_{x^*},$$
and we again have $(c(d_x)c_{x^*})^2\equiv
(d_xd_{x^*})^2\mod 16$. Hence,
\begin{equation}
\label{eq:fondamental3}
\left(\frac{2}{c(d_x)c_{x^*}}\right)=\left(\frac{2}{d_xd_{x^*}}\right).
\end{equation}
Now, using 
Equations~(\ref{eq:fondamental2}) and~(\ref{eq:fondamental3}),
like in the proof of Theorem~\ref{prop:legendre},
we obtain
$$
\left(\frac{2}d\right)=\left(\frac{2}q\right)\left(\frac{2}c\right)
\prod_{x\in\mathcal B_{\gamma}}
\left(\frac{2}{
d_xd_{x^*}}\right)
\left(\frac{2}{
c(d_{x^*})c_x}\right)=
\left(\frac{2}q\right)\left(\frac{2}c\right),
$$
\end{proof}

\begin{theorem}
\label{rk:pegual2}
The blockwise Navarro's conjecture holds for alternating groups
at $p=2$.
\end{theorem}

\begin{proof}
Let $b_{\gamma}$ be a $2$-block of $\Alt_n$ covered by a $2$-block
$B_{\gamma}$ of $\sym_n$ labeled by the $2$-core $\gamma$. Write
$r=|\gamma|$ and $w$ for the $2$-weight of $B_{\gamma}$. As above, we
denote by $\chi_{\lambda}$ the irreductible character of $\sym_n$ labeled
by $\lambda$. We also denote the irreducible characters of $\Alt_n$ by
$\vartheta_{\lambda}^+$ for $\lambda\neq\lambda^*$ and
$\vartheta_{\lambda}^{\pm}$ for $\lambda=\lambda^*$.  
We only have to consider the case that $\gamma$ is self-conjugate and
$w>0$.
Write $\mathcal P_{\gamma}$ for the set of partitions $\mu$ of $2w$ such that
$\chi_{\mu}$ has height zero or $\chi_{\mu}$ is of height $1$ and $\mu$ is
self-conjugate.
By~\cite[Proposition 12.5]{Olsson}, we have
$$\Irr_0(b_{\gamma})=\{\vartheta_{\lambda}^{\pm}
\mid {\mathcal Cor}_2(\lambda)=\gamma,\
{\mathcal Q}_2(\lambda)\in \mathcal P_{\gamma}\}.$$
By~\cite[(12.2)]{Olsson}, $B_{\gamma}$ covers only the block $b_{\gamma}$.
Hence, $b_{\gamma}$ is $\sym_n$-invariant, and by~\cite[Theorem
9.17]{Navarro} the defect group of $b_{\gamma}$ is $D=\Alt_n\cap \widetilde{D}$,
where $\widetilde{D}$ is the defect group of $B_{\gamma}$. Since
$\widetilde{D}$ is isomorphic to the Sylow $2$-subgroup of a $\sym_{2w}$,
it follows that $D$ is isomorphic to the Sylow $2$-subgroup of $\Alt_{2w}$
and
$$\operatorname{N}_{\Alt_n}(D)\simeq (\Sym_r\times
\operatorname{N}_{\Sym_{2w}}(D))^+.$$
We remark that
$(\operatorname{N}_{\Sym_{2w}}(D))^+=\operatorname{N}_{\Alt_{2w}}(D)$.
By~\cite[Theorem 5.6]{MichlerOlsson} applied to the principal $2$-block of
$\Alt_{2w}$, the number of $2'$-characters of the principal blocks of  
$\Alt_{2w}$ and of $\operatorname{N}_{\Alt_{2w}}(D)$ is the same. 
By~\cite[Proposition 12.5]{Olsson} $\mathcal P_{\gamma}$ labels the set of
$2'$-characters of the principal block of $\Alt_{2w}$. 
We choose a bijection $\theta$ between these two sets, and for
$\mu\in\mathcal P_{\gamma}$, we set
$$\psi_{\mu}^{\pm}=\theta(\vartheta_{\mu}^{\pm}).$$
Now, the second author proved in~\cite{RishiNathAlt} that if $w>3$, then
for all $\mu\in\mathcal P_{\gamma}$, $\vartheta^{\pm}_{\mu}$ and
$\theta(\vartheta^{\pm}_{\mu})$ are fixed by all $f\in\mathcal H_{n!/2}$.
Furthermore, for $w=1$ and $w=2$, the normalizer of the Sylow $2$-subgroup
of $\Alt_{2w}$ is $\Alt_{2w}$ itself. We can take $\theta$ to be the
identity, that is automatically 
$\mathcal H_{n!/2}$-equivariant.

Write $b'_{\gamma}$ for the Brauer correspondent of $b_{\gamma}$ 
 and
$B'_{\gamma}$ for the unique block of $\Sym_r\times
\operatorname{N}_{\Sym_{2w}}(D)$ that covers $b'_{\gamma}$ by
\cite[Corollary 9.6]{Navarro}.
%
By Clifford Theory, for any $\mu\in\mathcal P_{\gamma}$, there is
$\widetilde{\psi}_{\mu}\in\Irr(\operatorname{N}_{\Sym_{2w}}(D))$ such that
$\psi_{\mu}^{\pm}$ appears in its restriction to
$(\operatorname{N}_{\Sym_{2w}}(D))^+$ with multiplicity one. Hence, for
any $\mu\in\mathcal P_{\gamma}$, we have $(\chi_{\gamma}\otimes
\widetilde{\psi}_{\mu})^{\pm}\in \Irr_0(b_{\gamma}')$. On the other hand, 
by cardinality~\cite[Theorem 5.6]{MichlerOlsson}, we deduce that
$$\Irr_0(b'_{\gamma})=\{(\chi_{\gamma}\otimes
\widetilde{\psi}_{\mu})^{\pm}\mid \mu\in\mathcal P_{\gamma}\}.$$
Now, we define
$$\Phi:\Irr_0(b_{\gamma})\rightarrow\Irr_0(b'_{\gamma}),\quad 
\vartheta_{\lambda}^{\pm}\mapsto(\chi_{\gamma}\otimes \psi_{{\mathcal
Q}_2(\lambda)})^{\pm}.$$
We remark that $\Phi$ is a bijection by construction.
If $\lambda\neq\lambda^*$, then $\vartheta_{\lambda}^+$ and
$\Phi(\vartheta_{\lambda})^+$ are fixed by all $f\in\mathcal H_{n!/2}$.
Assume that $\lambda=\lambda^*$. Write $d_{\lambda}$, $d_{{\mathcal
Q}_2(\lambda)}$ and $d_{\gamma}$ for the
product of diagonal hooks of ${\lambda}$, $\mathcal Q_2(\lambda)$ and
$\gamma$. Then
 by~\cite[Theorem 2.2]{RishiNathAlt} and Theorem~\ref{prop:legendre2}, we obtain
for any $\lambda$ labeling a character of $b_{\gamma}$
$$
\varepsilon(\chi_{{\lambda}},f)=\left(\frac{2}{d_{{\lambda}}}\right)
=\left(\frac{2}{d_{\gamma}}\right)\left(\frac{2}{d_{\mathcal Q_2(\lambda)}}
\right)
=\varepsilon(\chi_{\gamma},f)\varepsilon(\chi_{\mathcal
Q_2(\lambda)},f).
$$
On the other hand,  by Proposition~\ref{prop:reducbouge}, for any $f\in\mathcal
H_{n!/2}$
$$\varepsilon(\chi_{\gamma}\otimes\widetilde{\psi}_{\mathcal Q_2(\lambda)},f)
=\varepsilon(\chi_{\gamma},f)\varepsilon(\widetilde{\psi}_{\mathcal
Q_2(\lambda)},f).$$
However, $\varepsilon(\widetilde{\psi}_{\mathcal Q_2(\lambda)},f)=
\varepsilon(\chi_{\mathcal Q_2(\lambda)},f)$ because $\theta$ is $\mathcal
H_{n!/2}$-equivariant. 
Thus,
$$\varepsilon(\chi_{{\lambda}},f)=
\varepsilon(\chi_{\gamma}\otimes\widetilde{\psi}_{\mathcal
Q_2(\lambda)},f),
$$
as required.

\end{proof}

\bigskip

{\bf Acknowledgements.} The first author acknowledges the support of
the ANR grant GeRepMod ANR-16-CE40-0010-01. The second author
acknowledges the support of PSC-CUNY TRADA-47-785. 
The authors would like to thank the Banff International Research Station
(BIRS) where conversations on this project began during the 2014
Global/Local Conjectures in Representation Theory of Finite Groups
workshop. Part of this work was done at the Centre Interfacultaire
Bernoulli (CIB) in the Ecole Polytechnique Federale de Lausanne
(Switzerland), during the 2016 Semester Local Representation Theory and
Simple Groups. The authors thank the CIB for the financial and logistical
support. The authors would also like to thank the IMJ-PRG at the
University of Paris Diderot and the Department of Mathematics at the
Graduate Center, City University of New York for the financial and
logistical support which allowed the completion of this project. The
authors sincerely thank Gunter Malle for his precise reading of the paper
and his helpful comments. Finally, the authors wish to thank the referee
for their careful reading of the manuscript and for their suggestions that
improves the reading of the paper. 

\bibliographystyle{abbrv}
\bibliography{references}

\end{document}